\documentclass[11pt]{amsart}
\usepackage[T1]{fontenc}
\usepackage[latin1]{inputenc}
\usepackage{a4wide}
\usepackage{setspace}
\usepackage{supertabular}
\usepackage{xypic}
\usepackage{amssymb}
\usepackage{amsthm}
\usepackage[english]{babel}
\usepackage{latexsym}
\usepackage{srcltx}
\date{}
\newtheorem{thm}{Th\'eor\`eme}[section]
\newtheorem*{thm*}{Th\'eor\`eme}

\newtheorem{conj-chap}[thm]{Conjecture}

\newtheorem{defn}[thm]{D\'efinition}

\newtheorem{prop}[thm]{Proposition}

\newtheorem{notation}[thm]{Notation}

\newtheorem{i-question}{Question}
\newtheorem*{question*}{Question}
\newtheorem{e-thm}[thm]{Theorem}
\newtheorem{e-example}[thm]{Example}
\newtheorem{e-defn}[thm]{Definition}
\newtheorem{e-rem}[thm]{Remark}
\newtheorem{e-lem}[thm]{Lemma}
\newtheorem{e-cor}[thm]{Corollary}
\newenvironment{e-proof}[1][\sc Proof.]{\begin{trivlist}
\item[\hskip \labelsep {\bfseries #1}]}{\hfill{$\square$}\end{trivlist}}
\newcommand{\Coprod}{\displaystyle\coprod}
\newcommand{\Prod}{\displaystyle\prod}
\newcommand{\fonc}[5]{
 \begin{array}{cccc}
 #1: & #2 & \longrightarrow & #3\\
     & #4 & \longmapsto & #5
 \end{array}
}
\newcommand{\appl}[4]{
 \begin{array}{cccc}
  #1 & \longrightarrow & #2\\
  #3 & \longmapsto & #4
 \end{array}
}

\begin{document}

\title{On the Weil-étale cohomology of number fields}
\author{Baptiste Morin}
\maketitle

\begin{abstract}
We give a direct description of the category of sheaves on
Lichtenbaum's Weil-étale site of a number ring. Then we apply this
result to define a spectral sequence relating Weil-étale cohomology
to Artin-Verdier étale cohomology. Finally we construct complexes of
étale sheaves computing the expected Weil-étale cohomology.
\end{abstract}

\footnotetext{ \emph{2000 Mathematics subject classification} :
14F20 (primary) 14G10 (secondary). \emph{Keywords} : étale
cohomology, Weil-étale cohomology, topos, Dedekind zeta function.}

\section{Introduction}
Stephen Lichtenbaum has conjectured in \cite{Lichtenbaum} the
existence of a Weil-étale topology for arithmetic schemes. The
associated cohomology groups with coefficients in motivic complexes
of sheaves should be finitely generated and closely related to
special values of zeta-functions. For example, Lichtenbaum predicts
that the Weil-étale cohomology groups with compact support
$H_{Wc}^i(Y;\mathbb{Z})$ exist, are finitely generated and vanish
for $i$ large, where $Y$ is a scheme of finite type over $Spec\,
\mathbb{Z}$. The order of annulation and the special value of the
zeta function $\zeta_Y(s)$ at $s=0$ should be given by
$$\mbox{ord}_{s=0}\,\zeta_Y(s)=\chi'_c(Y,\mathbb{Z})\mbox{ and }\zeta^*_Y(0)=\pm\chi_c(Y,\mathbb{Z}),$$
where $\chi_c(Y,\mathbb{Z})$ and $\chi'_c(Y,\mathbb{Z})$ are the
Euler characteristics defined in \cite{Lichtenbaum}. Lichtenbaum has
also defined a candidate for the Weil-étale cohomology when
$Y=Spec\,\mathcal{O}_K$, the spectrum of a number ring. Assuming
that the groups $H_W^i(\bar{Y};\mathbb{Z})$ vanish for $i\geq4$, he
has proven his conjecture in this case. However, Matthias Flach has
shown in \cite{MatFlach} that the groups $H_W^i(\bar{Y};\mathbb{Z})$
defined in \cite{Lichtenbaum} are in fact infinitely generated for
any even integer $i\geq4$. The aim of the present work is to study
in more details Lichtenbaum's definition and its relation to
Artin-Verdier étale cohomology.

Let $K$ be a number field and let $\bar{Y}$ be the Arakelov
compactification of $Spec\,\mathcal{O}_K$. In the second section we
define a topos $\mathfrak{F}_{_{L/K,S}}$, said to be flask, using the Weil group
$W_{L/K,S}$ associated to a finite Galois extension $L/K$ and a
finite set $S$ of places of $K$ containing the archimedean ones and
the places which ramify in $L$. We also define a topos
$\mathfrak{F}_{_{W,\bar{Y}}}$ using the full Weil-group $W_K$.
The first main result of this paper shows that the topos
$\mathfrak{F}_{_{L/K,S}}$ is canonically equivalent to the category
of sheaves on the Lichtenbaum Weil-étale site $T_{L/K,S}$. This
gives a simple description of the categories of sheaves
on those Weil-étale sites. In the spirit of \cite{SGA4}, it is sometimes
easier to work directly with these flask topoi rather than with
their generating sites $T_{L/K,S}$. Finally, this exhibits the somewhat unexpected
behavior of these categories of sheaves.

In the third section we compute the groups
$H_W^i(\bar{Y};\mathbb{Z}):=\underrightarrow{lim}\,H^i(\mathfrak{F}_{_{L/K,S}},\mathbb{Z})$
and $H^i(\mathfrak{F}_{_{W,\bar{Y}}},\mathbb{Z})$. Then we observe that the canonical map
$\underrightarrow{lim}\,H^i(\mathfrak{F}_{_{L/K,S}},\mathbb{Z})\rightarrow
H^i(\mathfrak{F}_{_{W,\bar{Y}}},\mathbb{Z})$ is not an isomorphism
for $i=2,3$. This points out that the current Weil-étale cohomology is not
defined as the cohomology of a site (i.e. of a topos).

In the seventh section we study the relation between the flask topoi
and the Artin-Verdier étale topos. This is then applied to define a
spectral sequence relating Weil-étale cohomology to étale
cohomology. The last section is devoted to the construction of
complexes of étale sheaves on
$\bar{Y}=\overline{Spec\,\mathcal{O}_K}$, where $K$ is a totally
imaginary number field. The étale hypercohomology of these complexes
yields the expected Weil-étale cohomology with and without compact
support. This last result was
suggested by a question of Matthias Flach. The existence of these complexes is a
necessary condition for the existence of a Weil-étale topos (i.e. a
topos whose cohomology is the conjectural Weil-étale cohomology)
over the Artin-Verdier étale topos.

\section{Notation}

Let $K$ be a number field and let $\bar{K}/K$ be an algebraic
closure of $K$. We denote by $Y$ the spectrum of the ring of
integers $\mathcal{O}_K$ of $K$. Following Lichtenbaum's
terminology, we call $\bar{Y}=(Y;Y_{\infty})$ the set of all
valuations of $K$, where $Y_{\infty}$ is the set of archimedean
valuations of $K$. This set $\bar{Y}$ is endowed with the Zariski
topology. The trivial valuation $v_0$ of $K$ corresponds to the
generic point of $Y$. We denote by $\bar{Y}^0$ the set of closed
points of $\bar{Y}$ (i.e. the set of non-trivial valuations of $K$).

\subsection{The global Weil group.}

Let $\bar{K}/L/K$ be a finite Galois extension of the number field
$K$. Let $S$ be a finite set of places of $K$ containing the
archimedean ones and the places which ramify in $L$. We denote by
$I_L$ and $C_L$ the idèle group and the idèle class group of $L$
respectively. Let $U_{L,S}$ be the subgroup of $I_L$ consisting of
those idèles which are $1$ at valuations lying over $S$, and units
at valuations not lying over $S$. It is well known that $U_{L,S}$ is
a cohomologically trivial $G(L/K)$-module. The natural map
$U_{L,S}\rightarrow C_L$ is injective and the $S$-idèle class group
$C_{L,S}$ is defined by $C_{L,S}=C_L/U_{L,S}$, as a topological
group. For any $i\in\mathbb{Z}$, the map
$$\widehat{H}^i(G(L/K),C_L)\longrightarrow\widehat{H}^i(G(L/K),C_{L,S})$$
is an isomorphism since $U_{L,S}$ is cohomologically trivial. By
class field theory, there exists a canonical class in
$\widehat{H}^2(G(L/K),C_{L,S})$ which yields a group extension
$$0\rightarrow C_{L,S}\rightarrow W_{L/K,S}\rightarrow G(L/K)\rightarrow 0.$$
If we assume that $S$ is the set of all non-trivial valuations of
$K$, then $W_{L/K,S}$ is the relative Weil group $W_{L/K}$. By
\cite{Lichtenbaum} Lemma 3.1, the global Weil group is the
projective limit
$$W_K=\underleftarrow{lim}\,\,W_{L/K,S},$$
over finite Galois $\bar{K}/L/K$ and finite $S$ as above.
\subsection{Galois groups and Weil groups.}
\subsubsection{}
For any valuation $v$ of $K$, we choose a valuation $\bar{v}$ of
$\bar{K}$ lying over $v$ and we denote by $D_{v}$ the associated
decomposition group and by $I_{v}$ the inertia group. We set
$$K_v^h:=\bar{K}^{D_{v}},\,\,\,K_v^{sh}:=\bar{K}^{I_{v}}\mbox{ and }G_{k(v)}:=Gal(K_v^{sh}/K_v^h)=D_v/I_v.$$
If $v\in Y$, then $k(v)$ is the residue field of the scheme $Y$ at
$v$. For any archimedean valuation $v$, the Galois group
$G_{k(v)}=\{1\}$ is trivial since $D_{v}=I_{v}$. Note that for the
trivial valuation $v=v_0$, one has $D_{v_0}=G_K$ and
$I_{v_0}=\{1\}$, hence $G_{k(v_0)}=G_K$.

Let $K_v$ be the completion of $K$ with respect to the valuation
$v$. Thus for $v=v_0$ the trivial valuation, $K_{v_0}$ is just $K$.
The choice of the valuation $\bar{v}$ of $\bar{K}$ lying over $v$
induces an embedding
$$\mathfrak{o}_v:D_v=G_{K_v}\longrightarrow G_K.$$
We choose a global Weil group $\alpha_{v_0}:W_K\rightarrow G_K$. For
any non-trivial valuation $v$, we choose a local Weil group
$\alpha_{K_v}:W_{K_v}\rightarrow G_{K_v}$ and a Weil map
$\theta_v:W_{K_v}\rightarrow W_{K}$ so that the diagram
 \[ \xymatrix{
 W_{K_v}\ar[r]^{\theta_v}\ar[d]_{\alpha_{K_v}} &W_K\ar[d]_{\alpha_{v_0}}  \\
 G_{K_v}  \ar[r]^{\mathfrak{o}_v}& G_K
} \] is commutative. For any valuation $v$, let
$W_{k(v)}:=W_{K_v}/I_v$ be the Weil group of the residue field at
$v$. Note that $W_{k(v)}$ is isomorphic to $\mathbb{Z}$
(respectively $\mathbb{R}$) as a topological group whenever $v$ is
ultrametric (respectively archimedean). We denote by
$$q_v:W_{K_v}\longrightarrow W_{k(v)}\mbox{ and
}\mathfrak{q}_v:G_{K_v}\longrightarrow G_{k(v)}$$ the canonical
continuous projections. One has $K_{v_0}=K$, $D_{v_0}=G_K$,
$I_{v_0}=\{1\}$, and $W_{k(v_0)}=W_{K_{v_0}}/I_{v_0}=W_K$. We set
$\theta_{v_0}=q_{v_0}=Id_{W_K}$ and
$\mathfrak{o}_{v_0}=\mathfrak{q}_{v_0}=Id_{G_K}$.

\subsubsection{}Let $v$ be a non-trivial valuation of $K$ and let $W_{K_v}\rightarrow
W_K$ be a Weil map. Consider the morphism
$$W_{K_v}\longrightarrow W_K\longrightarrow W_{L/K}=W_K/W_L^c,$$ where $L/K$ is a finite Galois
extension. Here $W_L^c$ is the closure of the commutator subgroup of
$W_L$. The valuation $\bar{v}$ lying over $v$ defines a valuation
$w$ of $L$ and the morphism $W_{K_v}\rightarrow W_{L/K}$ factors
through $W_{K_v}/W_{L_w}^c=W_{L_w/K_v}$. We get the following
commutative diagram
 \[ \xymatrix{
 0\ar[r]&L_w^{\times}\ar[r]\ar[d]& W_{L_w/K_v}\ar[d]\ar[r]& G(L_w/K_v)\ar[d]\ar[r]&0 \\
 0\ar[r]&C_L^{\times}\ar[r]& W_{L/K}\ar[r]& G(L/K)\ar[r]&0
 } \]
where the rows are both exact. The map $W_{L_w/K_v}\rightarrow
W_{L/K}$ is injective and the image of $W_{K_v}$ in $W_{L/K}$ is
isomorphic to $W_{L_w/K_v}$. Let $S$ be a finite set of places of
$K$ containing the archimedean ones and the places which ramify in
$L$. The group $U_{L,S}$ injects in $W_{L/K}$ and there is an
isomorphism $W_{L/K,S}\simeq W_{L/K}/U_{L,S}$. Hence the image of
$W_{K_v}$ in $W_{L/K,S}$ is isomorphic to $W_{L_w/K_v}$ for $v\in
S$. For $v$ not in $S$, the image of $W_{K_v}$ in $W_{L/K,S}$ is
isomorphic to the quotient of $W_{L_w/K_v}$ by
$\mathcal{O}_{L_w}^{\times}$. The canonical map $W_{K_v}\rightarrow
W_{k(v)}$ factors through $W_{L_w/K_v}$ hence through
$W_{L_w/K_v}/\mathcal{O}_{L_w}^{\times}$. We denote by
$\widetilde{W}_{K_v}$ the image of $W_{K_v}$ in $W_{L/K,S}$. For any
trivial valuation $v$ of $K$, the Weil map $W_{K_v}\rightarrow
W_{K}$ and the quotient map $W_{K_v}\rightarrow W_{k(v)}$ induce
morphisms $\theta_v:\widetilde{W}_{K_v}\rightarrow W_{L/K,S}$ and
$q_v:\widetilde{W}_{K_v}\rightarrow W_{k(v)}$ respectively.

\subsection{Left exact sites.} Let $\mathcal{C}$ be a category and let $\mathcal{J}$ be
a Grothendieck topology on $\mathcal{C}$. Recall that a category
$\mathcal{C}$ has finite projective limits if and only if
$\mathcal{C}$ has a final object and fiber products.
\begin{defn}
The site $(\mathcal{C};\mathcal{J})$ is said to be \emph{left exact}
whenever $\mathcal{C}$ has finite projective limits and
$\mathcal{J}$ is subcanonical.
\end{defn}
Note that any Grothendieck topos is equivalent the category of
sheaves of sets on a left exact site (see \cite{SGA4} IV Théorème
1.2).
\begin{defn}
A family of morphisms $\{X_i\rightarrow X;\,i\in I\}$ of the
category $\mathcal{C}$ is said to be a \emph{covering family} of $X$
if the sieve of $X$ generated by this family lies in
$\mathcal{J}(X)$.
\end{defn}
The covering families define a pretopology on $\mathcal{C}$ which
generates the topology $\mathcal{J}$, since $\mathcal{C}$ is left
exact. A \emph{morphism of left exact sites} is a functor
$a:\mathcal{C}\rightarrow\mathcal{C'}$ preserving finite projective
limits (i.e. $a$ is left exact), which is continuous. This means
that the functor
$$\appl{\widehat{\mathcal{C'}}}{\widehat{\mathcal{C}}}{\mathcal{P}}{\mathcal{P}\circ a},$$
sends sheaves to sheaves, where $\widehat{\mathcal{C}}$ is the
category of presheaves on $\mathcal{C}$ (contravariant functors from
$\mathcal{C}$ to the category of sets). We denote by
$\widetilde{(\mathcal{C},\mathcal{J})}$ the topos of sheaves of sets
on the site $(\mathcal{C};\mathcal{J})$. A morphism of left exact
sites
$a:(\mathcal{C},\mathcal{J})\rightarrow(\mathcal{C}',\mathcal{J}')$
induces a morphism of topoi
$\widetilde{a}=(\widetilde{a}^*,\widetilde{a}_*)$ such that the
square
 \[ \xymatrix{
 \widetilde{(\mathcal{C},\mathcal{J})}\ar[r]^{\widetilde{a}^*} & \widetilde{(\mathcal{C}',\mathcal{J}')}  \\
 \mathcal{C}\ar[u]_{}  \ar[r]^{a}& \mathcal{C'}\ar[u]_{}
} \] is commutative, where the vertical arrows are given by Yoneda
embeddings (which are fully faithful since the topologies are
sub-canonical) and $\widetilde{a}^*$ is the inverse image of
$\widetilde{a}$. We denote by $Et_X$ the small étale site of a
scheme $X$. The étale topos of $X$ (i.e. the category of sheaves of
sets on $Et_X$) is denoted by $X_{et}$. A morphism of schemes
$u:X\rightarrow Y$ gives rise to a morphism of left exact sites
$$\fonc{u^*}{Et_Y}{Et_X}{(U\rightarrow
Y)}{(U\times_YX\rightarrow X)},$$ hence to a morphism of topoi
$(u^*;u_*): X_{et}\rightarrow Y_{et}$. A diagram of topoi
\[ \xymatrix{
  \mathcal{S}_1\ar[d]_{b} \ar[r]^{a} & \mathcal{S}_2\ar[d]_{d}   \\
  \mathcal{S}_3  \ar[r]^{c}& \mathcal{S}_4
} \] is said to be \emph{commutative} if there is
 a canonical isomorphism of morphisms of topoi
$c\circ b\simeq d\circ a$, or in other words, an isomorphism in the
category $\underline{Homtop}\,(\mathcal{S}_1;\mathcal{S}_4)$ between
the objects $c\circ b$ and $d\circ a$. Strictly speaking, such a
diagram is only pseudo-commutative. In what follows, a topos is
always a Grothendiek topos and a morphism is a geometric morphism.

\subsection{The classifying topos of a topological group.}
Let $G$ be a topological group. The \emph{small classifying topos}
$B^{sm}_G$ is the category of sets on which $G$ acts continuously.
If $G$ is discrete or profinite (or more generally totally
disconnected) then the cohomology of the topos $B^{sm}_G$ is
precisely the cohomology of the group $G$.

For $G$ any topological group, we denote by $B_{Top}G$ the category
of $G$-topological spaces (which are elements of a given universe)
endowed with the local-section topology $\mathcal{J}_{ls}$ (see
\cite{Lichtenbaum} section 1), and $B_G$ is the topos of sheaves of
sets on this site.

Alternatively, let $Top$ be the category of topological spaces
(which are elements of a given universe) endowed with the open cover
topology $\mathcal{J}_{open}$. Recall that the open cover topology
is generated by the pre-topology for which a family of continuous
maps $\{U_i\rightarrow U\}$ is a cover when it is an open cover in
the usual sense. By (\cite{MatFlach} Lemma 1), one has
$\mathcal{J}_{ls}=\mathcal{J}_{open}$ on the category $Top$. We
denote by $\mathcal{T}$ the topos of sheaves of sets on the site
$(Top,\mathcal{J}_{open})$. Since the Yoneda embedding commutes with
projective limits, a topological group $G$ defines a group-object
$y(G)$ of $\mathcal{T}$. The \emph{classifying topos} $B_G$ of the
topological group $G$ is the topos of $y(G)$-objects of
$\mathcal{T}$. Recall that the data of an object $\mathcal{F}$ of
$\mathcal{T}$ is equivalent to the following. For any topological
space $X$, a sheaf $F_X$ on $X$ (i.e. an étalé space over $X$), and
for any continuous map $u:X'\rightarrow X$ a morphism
$\varphi_u:u^*F_X\rightarrow F_X'$ satisfying the natural
transitivity condition for a composition $v\circ u:X''\rightarrow
X'\rightarrow X$. Moreover, $\varphi_u$ is an isomorphism whenever
$u$ is an open immersion or more generally an étalement. This gives
a description of the topoi $\mathcal{T}$ and $B_G$. By
\cite{MatFlach} Corollary 2, the two preceding definitions of $B_G$
are equivalent. In other words, $(B_{Top}G;\mathcal{J}_{ls})$ is a
site for the classifying topos $B_G$.

\subsection{Cohomology of the Weil group.}
Let $\mathcal{E}$ be a topos. There is a unique morphism
$u:\mathcal{E}\rightarrow\underline{Set}$. The left exact functor
$\Gamma_{\mathcal{E}}:=u_*=Hom_{\mathcal{E}}(e_{\mathcal{E}},-)$ is
called the \emph{global sections functor}. Here $e_{\mathcal{E}}$
denotes the final object of $\mathcal{E}$. For any abelian object
$\mathcal{A}$ of $\mathcal{E}$, one has
$$H^i(\mathcal{E},\mathcal{A}):=R^i(\Gamma_{\mathcal{E}})(\mathcal{A}).$$
For any topological group $G$ and any abelian object of $B_G$ (in
particular a topological $G$-module), the cohomology of $G$ is
defined by (see \cite{MatFlach})
$$H^i(G,\mathcal{A}):=H^i(B_G,\mathcal{A}).$$
The following result is due to Stephen Lichtenbaum for $i\leq3$ and
to Matthias Flach for $i>3$. Denote by
$A^{\mathcal{D}}:=Hom_{cont}(A,\mathbb{R}/\mathbb{Z})$ the
Pontryagin dual of a locally compact abelian group $A$. The kernel
of the absolute value map $C_K\rightarrow\mathbb{R}_+^{\times}$ is
denoted by $C_K^1$.
\begin{e-thm}
Let $K$ be a totally imaginary number field and let $\mathbb{Z}$ be
the discrete $W_K$-module with trivial action. Then
\begin{eqnarray*}
H^i(W_K;\mathbb{Z})&=&\mathbb{Z} \mbox{ for }i=0,\\
                           &=& (C_K^1)^{\mathcal{D}}\mbox{ for }i=2,\\
                           &=& 0\mbox{ for i odd},
                           \end{eqnarray*}
and $H^i(W_K;\mathbb{Z})$ is an abelian group of infinite rank, in
particular nonzero, for even $i\geq 4$.
\end{e-thm}

\section{The flask topoi associated to a number field}

\subsection{Definition of the flask topoi.}
Let $L/K$ be an algebraic extension and let $S$ be a set of
non-trivial valuations of the number field $K$ containing all the
valuations of $F$ which ramify in $K$ and the archimedean ones. In
what follows, either $L/K$ is a finite Galois extension and $S$ is a
finite set, or $L=\bar{K}/K$ is an algebraic closure of $K$ and $S$
is the set of all non-trivial valuations of $K$.  Recall that
$\widetilde{W}_{K_v}$ denotes the image of ${W}_{K_v}$ in
$W_{L/K,S}$. The chosen Weil map and the quotient map induce
continuous morphisms
$$\theta_v:\widetilde{W}_{K_v}\rightarrow W_{L/K,S} \mbox{ and }q_v:\widetilde{W}_{K_v}\rightarrow W_{k(v)},$$
for any valuation $v$ of $K$. For the trivial valuation $v_0$, the
maps $\theta_{v_0}$ and $q_{v_0}$ are just $Id_{W_{L/K,S}}$.

\begin{defn}
We define a category $\mathfrak{F}_{_{L/K,S}}$ as follows. The
objects of this category are of the form
$\mathcal{F}=(F_v;f_v)_{v\in \bar{Y}}$, where $F_v$ is an object of
$B_{W_{k(v)}}$ for $v\neq v_0$ (respectively of $B_{W_{L/K,S}}$ for
$v=v_0$) and
$$f_v:q_v^*(F_v)\longrightarrow \theta_v^*(F_{v_0})$$ is a morphism of
$B_{\widetilde{W}_{K_v}}$ so that $f_{v_0}=Id_{F_{v_0}}$. A morphism
$\phi$ from $\mathcal{F}=(F_v;f_v)_{v\in\bar{Y}}$ to
$\mathcal{F'}=(F'_v;f'_v)_{v\in\bar{Y}}$ is family of morphisms
$\phi_v:F_v\rightarrow F'_v\in Fl(B_{W_{k(v)}})$ (and $\phi_{v_0}\in
Fl(B_{W_{L/K,S}})$) so that
 \[ \xymatrix{
  q_v^*(F_v)\ar[d]_{f_v} \ar[r]^{q_v^*(\phi_v)} & q_v^*(F'_v)\ar[d]_{f'_v}   \\
  \theta_v^*(F_{v_0})  \ar[r]^{\theta_v^*(\phi_{v_0})}& \theta_v^*(F'_{v_0})
} \] is a commutative diagram of $B_{\widetilde{W}_{K_v}}$. In what
follows, $F_v$ (respectively $\phi_v$) is called the
\emph{$v$-component} of the object $\mathcal{F}$ (respectively of
the morphism $\phi$).

For $L=\bar{K}$ and $S$ the set of all non-trivial valuations of
$K$, one has $W_{L/K,S}=W_K$, $\widetilde{W}_{K_v}={W}_{K_v}$ and we
set
$$\mathfrak{F}_{_{L/K,S}}=\mathfrak{F}_{_{W;\bar{Y}}}.$$
\end{defn}

The aim of this section is to prove that the category
$\mathfrak{F}_{_{L/K,S}}$ is a Grothendieck topos.
\begin{prop}\label{limitfinies+colim-dans-flask-component-wise}
Arbitrary inductive and finite projective limits exist in
$\mathfrak{F}_{_{L/K,S}}$, and are calculated componentwise.
\end{prop}

\begin{e-proof}
In order to simplify the notations we assume here that
$\mathfrak{F}_{_{L/K,S}}=\mathfrak{F}_{_{W;\bar{Y}}}$.  Let $I$ be a
small category and let $G:I\rightarrow\mathfrak{F}_{_{W;\bar{Y}}}$
be an arbitrary functor. For any valuation $v$ of $K$, one has a
canonical functor
$$\fonc{i_v^*}{\mathfrak{F}_{_{W;\bar{Y}}}}{B_{W_{k(v)}}}{\mathcal{F}}{F_v}.$$
For any valuation $v$, we set
$$G_v:=i_v^*\circ G:I\rightarrow B_{W_{k(v)}}.$$ The inductive limit
$$L_v:=\displaystyle{\lim_{\longrightarrow}}_{I}\,\,G_v$$
exists in the topos $B_{W_{k(v)}}$. A map $i\rightarrow j$ of the
category $I$ induces a map $G(i)\rightarrow G(j)$ of the category
$\mathfrak{F}_{_{W;\bar{Y}}}$. Hence for any valuation $v$, one has
a commutative diagram of $B_{W_{K_v}}$:
 \[ \xymatrix{
 q_v^*\circ G_v(i)\ar[d] \ar[r] & q_v^*\circ G_v(j)\ar[d]   \\
 \theta_v^*\circ G_{v_0}(i)  \ar[r]& \theta_v^*\circ G_{v_0}(j)
} \] By the universal property of inductive limits, one has an
induced morphism
$$\displaystyle{\lim_{\longrightarrow}}_{I}\,\,q_v^*\circ G_v\longrightarrow\displaystyle{\lim_{\longrightarrow}}_{I}\,\,\theta_v^*\circ G_{v_0},$$
where the limits are calculated in the topos $B_{W_{K_v}}$. We get a
map
$$l_v:q_v^*(L_v)=q_v^*(\displaystyle{\lim_{\longrightarrow}}_{I}\,\,
G_v)=\displaystyle{\lim_{\longrightarrow}}_{I}\,\,q_v^*\circ
G_v\longrightarrow\displaystyle{\lim_{\longrightarrow}}_{I}\,\,\theta_v^*\circ
G_{v_0}=
\theta_v^*(\displaystyle{\lim_{\longrightarrow}}_{I}\,\,G_{v_0})=\theta_v^*(L_{v_0}),$$
since $q_v^*$ and $\theta_v^*$ commute with arbitrary inductive
limits. This yields an object
$$\displaystyle{\lim_{\longrightarrow}}_{I}\,\, G =\mathcal{L}:=(L_v;\,\,l_v)_{v\in\bar{Y}}$$
of $\mathfrak{F}_{_{W;\bar{Y}}}$. Now, one has to check that
$\mathcal{L}$ is the inductive limit of the functor $G$. For any
object $\mathcal{X}$ of $\mathfrak{F}_{_{W;\bar{Y}}}$, denote by
$k_{\mathcal{X}}:I\rightarrow\mathfrak{F}_{_{W;\bar{Y}}}$ the
constant functor associated to $\mathcal{X}$. By construction, there
is a natural transformation
$$a:G\longrightarrow k_{\mathcal{L}}$$ such that any other natural transformation
$$b:G\longrightarrow k_{\mathcal{X}}$$ factors through $a$.
Indeed, the $v$-component of $\mathcal{L}$ is defined as the
inductive limit of $G_v$ in $B_{W_{k(v)}}$ and the morphism $l_v$ is
defined as the limit of the corresponding system of compatible maps
of $B_{W_{K_v}}$. The proof for finite projective limits is
identical.
\end{e-proof}

\begin{prop}\label{flasktopos-is-topos}
The category $\mathfrak{F}_{_{L/K,S}}$ is a topos.
\end{prop}

\begin{e-proof}
Again, we assume that
$\mathfrak{F}_{_{L/K,S}}=\mathfrak{F}_{_{W;\bar{Y}}}$ (i.e. $L$ is
an algebraic closure of $K$ and $S$ is the set of non-trivial
valuations of $K$). To see that it is a topos, we use Giraud's
criterion (see \cite{SGA4} IV Théorème 1.2). Axioms $\emph{(G1)}$,
$\emph{(G2)}$ and $\emph{(G3)}$ follow from proposition
\ref{limitfinies+colim-dans-flask-component-wise} and the fact that
$q_v^*$ and $\theta_v^*$ commute with finite projective limits and
arbitrary inductive limits.

\ \\
\emph{(G1) The category $\mathfrak{F}_{_{W;\bar{Y}}}$ has finite
projective limits.\\} More explicitly, $\mathfrak{F}_{_{W;\bar{Y}}}$
has a final object $(e_{W_{k(v)}};f_v)_{v\in \bar{Y}}$. Here
$e_{W_{k(v)}}$ is the final object of $B_{W_{k(v)}}$ and $f_v$ is
the unique map from the final object of $B_{W_{K_v}}$ to itself. Let
$\phi:\mathcal{F}\rightarrow\mathcal{X}$ and
$\phi':\mathcal{F'}\rightarrow\mathcal{X}$ be two maps of
$\mathfrak{F}_{_{W;\bar{Y}}}$ with the same target
$\mathcal{X}=(X_v;\xi_v)$. The fiber product
$\mathcal{F}\times_{\mathcal{X}}\mathcal{F'}$ is defined as the
object $(F_v\times_{X_v}F'_v;\,f_v\times_{\xi_v}f'_v)_{v\in
\bar{Y}}$, where the fiber products are calculated in the categories
$B_{W_{k(v)}}$ and $B_{W_{K_v}}$ respectively.

\ \\
\emph{(G2) All (set-indexed) sums exist in
$\mathfrak{F}_{_{W;\bar{Y}}}$, and are disjoint and stable.}
\ \\
The initial object of $\mathfrak{F}_{_{W;\bar{Y}}}$ is
$(\emptyset_{W_{k(v)}};f'_v)_{v\in \bar{Y}}$, where
$\emptyset_{W_{k(v)}}$ is the initial object of $B_{W_{k(v)}}$ and
$f'_v:\emptyset_{W_{K_v}}\rightarrow\emptyset_{W_{K_v}}$ is the
trivial map. Moreover, fiber products are computed componentwise in
$\mathfrak{F}_{_{W;\bar{Y}}}$, and an isomorphism $\phi$ from
$\mathcal{F}=(F_v;f_v)_{v\in \bar{Y}}$ to
$\mathcal{F'}=(F'_v;f'_v)_{v\in \bar{Y}}$ is a family of compatible
isomorphisms $\phi_v:F_v\rightarrow F'_v\in Fl(B_{W_{k(v)}})$. Then
one easily sees that $\emph{(G2)}$ is satisfied by
$\mathfrak{F}_{_{W;\bar{Y}}}$ since it is satisfied by
$B_{W_{k(v)}}$ for any valuation $v$.

\ \\
\emph{(G3) The equivalence relations are effective and universal.}
\ \\
Again this follows from the fact that arbitrary inductive limits
exist and are computed componentwise in
$\mathfrak{F}_{_{W;\bar{Y}}}$.

\ \\
\emph{(G4) The category $\mathfrak{F}_{_{W;\bar{Y}}}$ has a
small set of generators.} \ \\
This axiom, however, requires some argument. Choose a small set
$\{X_{v;i};\,i\in I_v\}$ of generators of $B_{W_{k(v)}}$, for any
valuation $v$. Recall that the morphism of topological groups
$\theta_v:W_{K_v}\rightarrow W_K$ induces the sequence of three
adjoint functors
$${\theta_v}_!\,\,\,;\,\,\,\,\,\,\theta_v^*\,\,\,;\,\,\,\,\,\,{\theta_v}_*$$
between $B_{W_{K_v}}$ and $B_{W_K}$, since $\theta_v^*$ commutes
with arbitrary projective and inductive limits (see \cite{SGA4}
IV.4.5.1). The functors $\theta_v^*$ and $\theta_{v*}$ are
respectively the inverse image and the direct image of the
(essential) morphism $B_{\theta_v}:B_{W_{K_v}}\rightarrow B_{W_K}$.
Denote by $y:Top\rightarrow\mathcal{T}$ the Yoneda embedding. The
functor ${\theta_v}_!$ is defined by
$$\fonc{{\theta_v}_!}{B_{W_{K_v}}}{B_{W_K}}{F}{y(W_K)\times^{y(W_{K_v})}F:=(y(W_K)\times F)/{y(W_{K_v})}},$$
where $y(W_{K_v})$ acts on the left on $F$ and  by
right-translations on $y(W_K)$.

Let $v\neq v_0$ be a non-trivial valuation and let $i\in I_v$. We
define an object $\mathcal{X}_{v;i}$ of
$\mathfrak{F}_{_{W;\bar{Y}}}$ as follows :
$$\mathcal{X}_{v;i}=({\theta_v}_!(q_v^*({X}_{v;i}))\,\,\,;\,\,\,{X}_{v;i}\,\,\, ;\,\,\,
(\emptyset_{B_{W_{k(w)}}})_{w\neq{v_0;v}}\,\,\,;\,\,\,\,(\xi_w)_{w\in
\bar{Y}}).$$ Here the map
$$\xi_v:q_v^*({X}_{v;i})\longrightarrow\theta_v^*\circ{\theta_v}_!(q_v^*({X}_{v;i}))$$
is given by adjunction, and $\xi_w$ is the trivial map for any
$w\neq v,v_0$. For the trivial valuation $v_0$ and for any $i\in
I_{v_0}$, we set
$$\mathcal{X}_{v_0;i}:=({X}_{v_0;i}\,\,\,;\,\,\,(\emptyset_{B_{W_{k(w)}}})_{w\neq v_0}).$$
The family $\{\mathcal{X}_{v;i};\,\,\,v\in \bar{Y};\,\,\,i\in I_v\}$
is set indexed. We claim that it is a generating family of
$\mathfrak{F}_{_{W;\bar{Y}}}$. Let $\mathcal{F}=(F_v;f_v)_v$ be an
object of $\mathfrak{F}_{_{W;\bar{Y}}}$, let $v$ be a valuation of
$K$ and let $t_v:X_{v;i}\rightarrow F_v$ be a morphism in
$B_{W_{k(v)}}$. One needs to show that there exists a canonical
morphism
$$t:\mathcal{X}_{v;i}\longrightarrow \mathcal{F}$$ so that the $v$-component of $t$ is $t_v$.
It is obvious for the trivial valuation $v=v_0$. Let $v\neq v_0$ be
a non-trivial valuation. Consider the morphism
$$f_v\circ q_v^*(t_v):q_v^*(X_{v;i})\longrightarrow q_v^*(F_v)\longrightarrow \theta_v^*(F_{v_0}).$$
By adjunction, there is an identification
\begin{equation}\label{lichtenbaum-adjunction}
Hom_{B_{W_K}}({\theta_v}_!(q_v^*(X_{v;i}));F_{v_0})=Hom_{B_{W_{K_v}}}(q_v^*(X_{v;i});\theta_v^*(F_{v_0})).
\end{equation}
Hence there exists a unique morphism
$$t_0:{\theta_v}_!(q_v^*(X_{v;i}))\longrightarrow F_{v_0}$$ of
$B_{W_K}$ corresponding to $f_v\circ q_v^*(t_v)$ via
(\ref{lichtenbaum-adjunction}) so that the diagram
 \[ \xymatrix{
  q_v^*(X_{v;i})\ar[d]_{\xi_v} \ar[r]^{q_v^*(t_v)}
  &q_v^*(F_v)\ar[d]_{f_v}   \\
  \theta_v^*\theta_{v!}(q_v^*X_{i;v})  \ar[r]^{\,\,\,\,\,\,\,\theta_v^*(t_0)}&
  \theta_v^*F_{v_0}
} \] is commutative. We get a morphism
$t:\mathcal{X}_{v;i}\rightarrow \mathcal{F}$ of the category
$\mathfrak{F}_{_{W;\bar{Y}}}$.

 Now, consider two parallel arrows
$\phi,\,\,\varphi:\mathcal{F}\rightarrow\mathcal{E}$ so that, for
any arrow $t:\mathcal{X}_{v;i}\rightarrow\mathcal{F}$, one has
$\phi\circ t=\varphi\circ t$. The family $\{X_{v;i};\,i\in I_v\}$ is
a family of generators of $B_{W_{k(v)}}$ and each morphism
$t_v:X_{v;i}\rightarrow F_v$ induces a morphism
$t:\mathcal{X}_{v;i}\rightarrow \mathcal{F}$. It follows that
$$\phi_v=\varphi_v\in Fl(B_{W_{k(v)}}),$$ for any $v\in \bar{Y}$. By
definition of the morphisms in the category
$\mathfrak{F}_{_{W;\bar{Y}}}$, the functor
$$(i_v^*)_{v\in\bar{Y}}:\mathfrak{F}_{_{W;\bar{Y}}}\longrightarrow
\coprod_{v\in\bar{Y}}B_{W_{k(v)}}$$ is faithful. It follows that
$\phi=\varphi$. This shows that the family
$\{\mathcal{X}_{v;i};\,\,\,v\in \bar{Y};\,\,\,i\in I_v\}$ is a small
collection of generators of $\mathfrak{F}_{_{W;\bar{Y}}}$. Therefore
the category $\mathfrak{F}_{_{W;\bar{Y}}}$ is a topos.
\end{e-proof}

A topos is said to be \emph{compact} if any cover of the final
object by sub-objects has a finite sub-cover. A first consequence of
this artificial construction is the fact that this property is not satisfied by these flask topoi, as it is shown
below. As a consequence, the global sections functor (and a fortiori cohomology) does not commute
with filtered inductive limits (not even with direct sums).

\begin{prop}\label{prop-relat-flask-not-compact}
The topos $\mathfrak{F}_{_{L/K,S}}$ is not compact.
\end{prop}

\begin{e-proof}
For any non-trivial valuation $v$ of $K$, let $E_v$ be the object of
$\mathfrak{F}_{_{L/K,S}}$ defined as follows. The $w$-component of
$E_{v}$ is the initial object $\emptyset$ of $B_{W_{k(w)}}$ for
$w\neq v,v_0$, and the final object for $w=v,v_0$ (i.e. the sheaf
represented by the one point space with trivial action). Let $e$ be
the final object of $\mathfrak{F}_{_{L/K,S}}$. The unique map
$E_v\rightarrow e$ is mono hence $E_v$ is a sub-object of the final
object of $\mathfrak{F}_{_{L/K,S}}$. The family $\{E_v\rightarrow
e,\,v\neq v_0\}$ is epimorphic. It is therefore a covering family of
$e$ by sub-objects. However, any finite sub-family is not a covering
family.
\end{e-proof}

\subsection{The morphisms associated to the valuations}
A valuation $v$ of the number field $K$ can be seen as a morphism
$v\rightarrow\bar{Y}$ inducing in turn a morphism of topoi.

\begin{prop}\label{closed-embedding-flask-topos}
For any non-trivial valuation $v$, there is a closed embedding :
$$i_v:=i_{_{L,S,v}}:B_{W_{k(v)}}\longrightarrow\mathfrak{F}_{_{L/K,S}}.$$
\end{prop}
\begin{e-proof}
For any valuation $v\neq v_0$, the functor
$$\fonc{i_v^*}{\mathfrak{F}_{_{L/K,S}}}{B_{W_{k(v)}}}{\mathcal{F}}{F_v}.$$
commutes with arbitrary inductive limits and finite projective
limits, since these limits are computed componentwise in the topos
$\mathfrak{F}_{_{L/K,S}}$. Hence $i_v^*$ is the pull-back of a
morphism of topoi
$i_v:B_{W_{k(v)}}\rightarrow\mathfrak{F}_{_{L/K,S}}.$ The same
argument shows that there is a morphism
$j_{_{L/K,S}}:B_{W_{L/K,S}}\rightarrow\mathfrak{F}_{_{L/K,S}}$.
Moreover, one easily sees that the functor
$$\fonc{i_{v*}}{B_{W_{k(v)}}}{\mathfrak{F}_{_{L/K,S}}}{F_v}{(e_{W_{K}};\,\,\,F_v;\,\,\,(e_{W_{k(w)}})_{w\neq v;v_0}\,\,\,)}.$$
is right adjoint to $i_{v}^*$, where $e_{W_{k(w)}}$ is the final
object of $B_{W_{k(w)}}$. Since the adjunction transformation
$Id\rightarrow i_v^*\circ i_{v*}$ is obviously an isomorphism, the
morphism $i_v$ is an embedding (see \cite{SGA4} IV. Définition
9.1.1). Consider the sub-terminal object $U:=((e_{W_{k(w)}})_{w\neq
v};\,\,\,\emptyset_{W_{k(v)}})$ of $\mathfrak{F}_{_{L/K,S}}$. It
defines an open sub-topos
$$j_v:\mathcal{U}:={(\mathfrak{F}_{_{L/K,S}})}_{/U}\longrightarrow\mathfrak{F}_{_{L/K,S}}.$$
The image of $i_{v*}$ is exactly the strictly full sub-category of
$\mathfrak{F}_{_{L/K,S}}$ defined by the objects $X$ such that
$j_v^*(X)$ is the final object of $\mathcal{U}$. Hence the image of
$i_{v}$ is the closed complement of the open sub-topos $\mathcal{U}$
(see \cite{SGA4} IV Proposition 9.3.4).
\end{e-proof}

The following corollary follows from the fact that $i_{v*}$ is a
closed embedding (see \cite{SGA4} IV.14).
\begin{e-cor}
The functor induced by $i_{v*}$ between the categories of abelian
sheaves is exact.
\end{e-cor}
More precisely, the functor $i_{v*}$ (between abelian categories)
has a left adjoint $i_{v}^*$ and a right adjoint $i^!_{v}$ (in fact
one has six adjoint functors). This last functor is defined as
follows :
$$
\fonc{i_v^!}{Ab\,(\mathfrak{F}_{_{L/K,S}})}{Ab\,(B_{W_{k(v)}})}{(F_w,f_w)_{w\in\bar{Y}}}{Ker
(f_v)}
$$

A morphism of topos $j=(j^*,j_*)$ is said to be \emph{essential} if
the
inverse image $j^*$ has a left adjoint $j_!$.

\begin{prop}
There is an essential morphism
$j:=j_{_{L/K,S}}:B_{W_{L/K,S}}\longrightarrow\mathfrak{F}_{_{L/K,S}}$.
\end{prop}
\begin{e-proof}
The functor
$$\fonc{j^*}{\mathfrak{F}_{_{L/K,S}}}{B_{W_{L/K,S}}}{\mathcal{F}}{F_{v_0}}$$
commutes with arbitrary inductive limits and finite projective
limits. Therefore $j^*$ has a right adjoint $j_*$ and thus is the
pull-back of a morphism of topoi. We define
$$
\fonc{j_!}{B_{W_{L/K,S}}}{\mathfrak{F}_{_{L/K,S}}}{\mathcal{L}}{(F_v,f_v)_{v\in\bar{Y}}}
$$
where $F_{v_0}=\mathcal{L}$ and $F_v=\emptyset$ is the initial
object of $B_{W_{k(v)}}$ for any $v\neq v_0$. The map $f_v$ is the
unique map from the initial object of $B_{\widetilde{W}_{K_v}}$ to
$\theta_v^*\mathcal{L}$. Clearly, $j_!$ is left adjoint to $j^*$.
\end{e-proof}

\begin{prop}
The direct image functor $j_*$ is given by
$$
\fonc{j_*}{B_{W_{L/K,S}}}{\mathfrak{F}_{_{L/K,S}}}{\mathcal{L}}{(q_{v*}\theta_v^*\mathcal{L},l_v)_{v\in\bar{Y}}}
$$
where the map
$$l_v:q_{v}^*q_{v*}\theta_v^*\mathcal{L}\longrightarrow\theta_v^*q_{v_0*}\theta_{v_0}^*\mathcal{L}=\theta_v^*\mathcal{L}$$
is induced by the natural transformation $q_{v}^*q_{v*}\rightarrow
Id$, for any valuation $v$.
\end{prop}
\begin{e-proof}
One has to show that $j_*$ is right adjoint to $j^*$. Let
$\mathcal{L}$ be an object of $B_{W_{L/K,S}}$ and let
$\mathcal{F}=(F_v;f_v)_{v\in\bar{Y}}$ be an object of
$\mathfrak{F}_{_{L/K,S}}$. For any map
$\phi_0:F_0\rightarrow\mathcal{L}$ of $B_{W_{L/K,S}}$ and any
non-trivial valuation $v$, consider the map
$$\theta_v^*(\phi_0)\circ
f_v:q_v^*F_v\rightarrow\theta_v^*F_0\rightarrow\theta_v^*\mathcal{L}.$$
Since $q_v^*$ is left adjoint to $q_{v*}$, there exists a unique map
$\phi_v:F_v\rightarrow q_{v*}\theta_v^*\mathcal{L}$ such that the
diagram
 \[ \xymatrix{
  q_v^*F_v\ar[r]^{q_v^*\phi_v} \ar[d]_{f_v}
  &q_v^*q_{v*}\theta_v^*\mathcal{L}\ar[d]_{l_v}   \\
  \theta_v^*F_0  \ar[r]^{\,\,\,\,\,\,\,\theta_v^*(\phi_0)}&
  \theta_v^*\mathcal{L}
} \] is commutative. We obtain a functorial isomorphism
$$Hom_{B_{W_{L/K,S}}}(j^*\mathcal{F},\mathcal{L})\simeq Hom_{\mathfrak{F}_{_{L/K,S}}}(\mathcal{F},j_*\mathcal{L}).$$
\end{e-proof}

\begin{e-cor}
The morphism $j:B_{W_{L/K,S}}\rightarrow\mathfrak{F}_{_{L/K,S}}$ is
an embedding.
\end{e-cor}
\begin{e-proof}
Indeed for any object $\mathcal{L}$ of $B_{W_{L/K,S}}$, the natural
map $j^*j_*\mathcal{L}\rightarrow\mathcal{L}$ is just the identity
of $\mathcal{L}$.
\end{e-proof}

If there is no risk of ambiguity, we denote
$\widetilde{W}_{k(v)}={W}_{k(v)}$ for $v\neq v_0$,
$\widetilde{W}_{k(v_0)}={W}_{L/K,S}$ and $j=i_{v_0}$.
\begin{prop}
The family of functors
$$\{i^*_v:\mathfrak{F}_{_{L/K,S}}\rightarrow B_{\widetilde{W}_{k(v)}},\,v\in\bar{Y}\}$$
is conservative.
\end{prop}
\begin{e-proof}
This follows immediately from the definitions.
\end{e-proof}

\begin{prop}
The family of functors
$$\{i^*_v:\mathfrak{F}_{_{L/K,S}}\longrightarrow B_{W_{k(v)}};\,\,\,v\in\bar{Y}^0\}$$
is not conservative.
\end{prop}
\begin{e-proof}In order to simplify the notations, we assume here that
$\mathfrak{F}_{_{L/K,S}}=\mathfrak{F}_{_{W;\bar{Y}}}$ Let
$\emptyset$ be the initial object of $\mathfrak{F}_{_{W;\bar{Y}}}$
and let $\mathcal{G}$ be the object whose $v_0$-component is the
final object of $B_{W_K}$ while its $v$-component is the initial
object of $B_{W_{k(v)}}$ for any $v\neq v_0$. Consider the morphism
$\phi:\emptyset\rightarrow\mathcal{G}$. Then  $\phi$ is not an
isomorphism while
$i_v^*(\phi):\emptyset_{W_{k(v)}}\rightarrow\emptyset_{W_{k(v)}}$ is
an isomorphism for any closed point $v$.
\end{e-proof}

\subsection{The transition maps.}

Let $(L/K,S)$ and $(L'/K,S')$ be as above. If $L\subset L'$ in
$\bar{K}$ and $S\subset S'$, then there is a canonical morphism
$$p:W_{L'/K,S'}\longrightarrow W_{L/K,S}.$$
\begin{prop}
There is an induced morphism of topoi
$$t:\mathfrak{F}_{_{L'/K,S'}}\longrightarrow\mathfrak{F}_{_{L/K,S}}.$$
For $L''/L'/L$ and $S\subset S'\subset S''$, the diagram
\[ \xymatrix{
\mathfrak{F}_{_{L''/K,S''}}\ar[r]\ar[rd]&\mathfrak{F}_{_{L'/K,S'}}\ar[d]\\
&\mathfrak{F}_{_{L/K,S}} }\] is commutative.
\end{prop}
In the following proof, for any non-trivial valuation $v$ we denote
by $W_{K_v,L,S}$ the image of $W_{K_v}$ in $W_{L/K,S}$ (this is
group is denoted by $\widetilde{W}_{K_v}$ in the rest of the paper).
Let $\theta_{v,L,S}:W_{K_v,L,S}\rightarrow W_{L/K,S}$ and
$q_{v,L,S}:W_{K_v,L,S}\rightarrow W_{k(v)}$ be the induced
morphisms. One has a continuous map $p_v:W_{K_v,L',S'}\rightarrow
W_{K_v,L,S}$.
\begin{e-proof}
Let $\mathcal{F}=(F_v;f_v)_{v\in\bar{Y}}$ be an object of
$\mathfrak{F}_{_{L/K,S}}$. Then,
$$t^*\mathcal{F}=(p^*F_{v_0},F_{v},p_v^*{f}_v)$$ does define an object of $\mathfrak{F}_{_{L'/K,S'}}$.
Indeed, $p_v^*{f}_v$ gives a map
$$q_{v,L',S'}^*F_v=p_v^*q_{v,L,S}^*F_v\longrightarrow p_v^*\theta_{v,L,S}^*F_0=\theta_{v,L',S'}^*p^*F_0,$$
since the diagram of topological groups
\[ \xymatrix{
W_{L'/K,S'}\ar[d]_{p} &W_{K_v,L',S'} \ar[l]_{\theta_{v,L',S'}}\ar[dr]^{q_{v,L',S'}} \ar[d]^{p_v}&   \\
W_{L/K,S}&W_{K_v,L,S} \ar[l]_{\theta_{v,L,S}}\ar[r]^{q_{v,L,S}}&
W_{k(v)} }
\]
is commutative. This yields a functor
$$t^*:\mathfrak{F}_{_{L/K,S}}\longrightarrow\mathfrak{F}_{_{L'/K,S'}},$$
which commutes with finite projective limits and arbitrary inductive
limits, by Proposition
\ref{limitfinies+colim-dans-flask-component-wise}. Hence $t^*$ is
the pull-back of a morphism of topoi $t$. The diagram of the
proposition is easily seen to be commutative, using the
commutativity of following triangles :
\[ \xymatrix{
W_{L''/K,S''}\ar[r]\ar[rd]&W_{L'/K,S'}\ar[d]&\mbox{and} &W_{K_v,L'',S''}\ar[r]\ar[rd]&W_{K_v,L',S'}\ar[d]\\
&W_{L/K,S}& & &W_{K_v,L,S} }\]
\end{e-proof}

\begin{e-rem}\label{rem-fiber-flask-topos}
The family $(\mathfrak{F}_{_{L/K,S}})_{_{L/K,S}}$ is a
\emph{projective system of topoi}. Indeed, consider the filtered set
$I/_K$ consisting of pairs $(L/K,S)$, where $L/K$ is a finite Galois
sub-extension of $\bar{K}/K$ and $S$ is a finite set of places of
$K$ containing the archimedian ones and the places ramified in
$L/K$. There is an arrow $(L'/K,S')\rightarrow (L/K,S)$ if and only
if $L\subseteq L'\subseteq \bar{K}$ and $S\subseteq S'$. The
previous proposition shows that one has a pseudo-functor
$$
\fonc{\mathfrak{F}_\bullet}{I/_K}{\underline{Topos}}{(L/K,S)}{\mathfrak{F}_{_{L/K,S}}}
$$
\end{e-rem}

\begin{prop}\label{prop-transitivity-incl-generic}
The following diagrams
\[ \xymatrix{
B_{W_{L'/K,S'}}\ar[d]^{}\ar[r]^{}&\mathfrak{F}_{_{L'/K,S'}}\ar[d]^{}&
&
B_{W_{k(v)}}\ar[dr]\ar[r]^{}&\mathfrak{F}_{_{L'/K,S'}}\ar[d]\\
B_{W_{L/K,S}}\ar[r]^{}&\mathfrak{F}_{_{L/K,S}}&&&\mathfrak{F}_{_{L/K,S}}
}\] are both commutative, for any non-trivial valuation $v$.
\end{prop}
\begin{e-proof}
This follows immediately from the definition of $t$.
\end{e-proof}
For any $v\in\bar{Y}$, there is a canonical morphism of topological
groups
$$l_v:W_{k(v)}\rightarrow \mathbb{R}.$$
For the trivial valuation $v=v_0$, the map $l_{v_0}$ is defined as
follows:
$$l_{v_0}:W_{k(v_0)}=W_K\longrightarrow
W_K^{ab}\simeq
C_K\longrightarrow\mathbb{R}^{>0}\longrightarrow\mathbb{R}.$$ The
first map is the projection from $W_K$ to its maximal abelian
Hausdorff quotient. The second map is given by the absolute value
map from the idèle class group $C_K$ to $\mathbb{R}^{>0}$. The third
map is the logarithm.

Let $Pic(\bar{Y})$ be the topological group obtained by dividing the
idèle class group $C_K$ by the unit idèles. This group is known as
the \emph{Arakelov Picard group}. The map $C_K\rightarrow\mathbb{R}$
defined as above induces a map $Pic(\bar{Y})\rightarrow\mathbb{R}$.
One has also a continuous morphism $W_{k(v)}\rightarrow
Pic(\bar{Y})$, for any non-trivial valuation $v$. This yields the
map
$$l_v:W_{k(v)}\rightarrow Pic(\bar{Y})\rightarrow\mathbb{R}.$$
Note that if $v$ is ultrametric, then $l_v$ sends the canonical
generator of  $W_{k(v)}$ to $log(N(v))\in\mathbb{R}$, where
$N(v)=|k(v)|$ is the norm of the closed point $v$ of the scheme $Y$.
Finally, The map $l_{v_0}$ induces a morphism
$l_{L,S}:W_{L/K,S}\rightarrow\mathbb{R}$. We have an induced
morphism of classifying topoi :
$$B_{l_{L,S}}:B_{W_{L/K,S}}\longrightarrow B_{\mathbb{R}}.$$

\begin{prop}\label{morph-flask-ds-BR}
There is a morphism
$$f_{_{L/K,S}}:\mathfrak{F}_{_{L/K,S}}\longrightarrow
B_{\mathbb{R}}$$ so that $f_{_{L/K,S}}\circ i_v$ is isomorphic to
$B_{l_v}$, for any closed point $v$ of $\bar{Y}$.
\end{prop}

\begin{e-proof}
In order to simplify the notations, we assume that
$\mathfrak{F}_{_{L/K,S}}=\mathfrak{F}_{_{W;\bar{Y}}}$. The functor
${(B_{l_v}^*)}_v:B_{\mathbb{R}}\rightarrow\coprod_v B_{W_{k(v)}}$
factors through $\mathfrak{F}_{_{W;\bar{Y}}}$. Indeed, for any
object $\mathcal{F}$ of $B_{\mathbb{R}}$, define
$$f^*(\mathcal{F}):=(B_{l_v}^*(\mathcal{F})\,;\,Id_{B_{L_v}^*(\mathcal{F})})_{v\in\bar{Y}}.$$
Here
$$Id_{B_{L_v}^*(\mathcal{F})}:q_v^*B_{l_v}^*(\mathcal{F})=B_{L_v}^*(\mathcal{F})
\longrightarrow
B_{L_v}^*(\mathcal{F})=\theta_v^*B_{l_{v_0}}^*(\mathcal{F})$$ is the
identity of the object $B_{L_v}^*(\mathcal{F})$ of the category
$B_{W_{K_v}}$, where $L_v:W_{K_v}\rightarrow \mathbb{R}$ is the
canonical morphism. This is well defined since the square

\[ \xymatrix{
  W_{K_v} \ar[d]_{\theta_v} \ar[r]^{q_v} & W_{k(v)}\ar[d]^{l_v}   \\
  W_{K}   \ar[r]^{l_{v_0}}& \mathbb{R}
} \] is commutative and $L_v:=l_{v_0}\circ\theta_v=l_v\circ q_v$.
This yields a functor
$$f^*:B_{\mathbb{R}}\longrightarrow\mathfrak{F}_{_{W;\bar{Y}}},$$
which commutes with finite projective limits and arbitrary inductive
limits, by Proposition
\ref{limitfinies+colim-dans-flask-component-wise}. Hence $f^*$ is
the pull-back of a morphism of topoi $f$ such that there is an
isomorphism of functors $i_v^*\circ f^*\simeq B_{l_v}^*$. For finite
$L/K$ and finite $S$, the same construction is valid by replacing
$W_K$ with $W_{L/K,S}$ and $W_{K_v}$ with $\widetilde{W}_{K_v}$.
\end{e-proof}

\begin{prop}\label{compatible-morphisms-flask-BR}
The following diagram is commutative
\[ \xymatrix{
\mathfrak{F}_{_{L'/K,S'}}   \ar[dr] \ar[r] & \mathfrak{F}_{_{L/K,S}}\ar[d]   \\
& B_{\mathbb{R}} } \] for any $\bar{K}/L'/L/K$ and $S\subset S'$.
\end{prop}
The proof is left to the reader.
\begin{prop}
Let $\mathcal{L}$ be an object of $\mathcal{T}$ with trivial
$y(W_{L/K,S})$-action. We denote also by $\mathcal{L}$ the object of
$B_{\mathbb{R}}$ defined by $\mathcal{L}$ with trivial
$y(\mathbb{R})$-action. There is an isomorphism
$$j_{_{L/K,S}*}\mathcal{L}\simeq f_{_{L/K,S}}^*\mathcal{L}.$$
\end{prop}
\begin{e-proof}On the one hand, one has $f_{_{L/K,S}}^*\mathcal{L}=(F_v,f_v)_{v\in\bar{Y}}$, where
$F_v$ is defined by the trivial action on $\mathcal{L}$, for any
valuation $v$. The map $f_v$ is given by the identity of
$\mathcal{L}$.

On the other hand, one has
$j_{_{L/K,S}*}\mathcal{L}=(q_{v*}\theta_v^*\mathcal{L},l_v)$. Let
$v$ be a non-trivial valuation of $K$. The object
$\theta_v^*\mathcal{L}$ is $\mathcal{L}$ with trivial
$y(\widetilde{W}_{K_v})$-action. The map $\mathcal{L}\rightarrow
q_{v*}q_v^*\mathcal{L}$ given by adjunction is an isomorphism. It
follows that $q_{v*}\mathcal{L}$ is $\mathcal{L}$ with trivial
$y(W_{k(v)})$-action.
\end{e-proof}

\begin{e-rem}In particular the following assertions hold.
Let $\mathbb{Z}$ be the constant object of $\mathcal{T}$. Then
$j_{_{L/K,S}*}\mathbb{Z}$ is the constant object of
$\mathfrak{F}_{_{L/K,S}}$ associated to $\mathbb{Z}$. Let
$\widetilde{\mathbb{R}}$ be the object of $\mathcal{T}$ represented
by the topological group $\mathbb{R}$. Then
$j_{_{L/K,S}*}\widetilde{\mathbb{R}}$ is the constant object of
$\mathfrak{F}_{_{L/K,S}}$ (over $\mathcal{T}$) associated to
$\widetilde{\mathbb{R}}$. In other words, the $v$-component of
$j_{_{L/K,S}*}\widetilde{\mathbb{R}}$ is $\widetilde{\mathbb{R}}$
for any valuation $v$ with $Id_{\widetilde{\mathbb{R}}}$ as
specialization maps.
\end{e-rem}

\section{Cohomology}
In order to use the results of \cite{MatFlach}, we assume in this
section that $K$ is a totally imaginary number field. We compute the
cohomology of the total flask topos and the Lichtenbaum Weil-étale
cohomology of any open subset of $\bar{Y}$. The Lichtenbaum
Weil-étale cohomology is defined as a direct limit and requires some
precautions to be computed rigourously.

\subsection{Preliminaries}
Recall that $\theta_{v_0}=q_{v_0}=Id_{W_{L/K,S}}$. In particular the
direct image of the induced morphism of topoi
$q_{v_0*}:B_{W_{L/K,S}}\rightarrow B_{W_{L/K,S}}$ is the identity
functor. Hence $R^n(q_{v_0*})=0$ for $n\geq 1$.

\begin{prop}\label{prop-higher-direct-im-generic}
Let $\mathcal{A}$ be an abelian object of $B_{W_{L/K,S}}$. For any
$n\geq 0$, one has
$$R^n(j_{_{L,S*}})(\mathcal{A})=(R^n(q_{v*})\theta_v^*\mathcal{A},t_v).$$
Here the map $t_v$ is the trivial map
$$t_v:q_v^*R^n(q_{v*})\theta_v^*\mathcal{A}\longrightarrow\theta_v^*R^n(q_{v_0*})\theta_{v_0}^*\mathcal{A}=0,$$
for $n\geq 1$.
\end{prop}
\begin{e-proof}
In this proof, we denote the morphism $j_{_{L,S}}$ simply by $j$ .
For any $n\geq 1$, one has
\begin{equation}\label{spectral-seque-generic-point}
j^*R^n(j_*)\mathcal{A}=R^n(j^*j_*)\mathcal{A}=R^n(Id)\mathcal{A}=0.
\end{equation}
Indeed, the functor $j^*$ is exact and $j_*$ preserves injective
objects. Hence the spectral sequence
$$R^p(j^*)R^q(j_*)(\mathcal{A})\Rightarrow
R^{p+q}(j^*j_*)(\mathcal{A})$$ degenerates and
(\ref{spectral-seque-generic-point}) follows. Let $v$ be a
non-trivial valuation. One has
$i_v^*j_*\mathcal{A}=q_{v*}\theta_v^*\mathcal{A}$ hence
$$i_v^*R^n(j_*)\mathcal{A}=R^n(i_v^*j_*)\mathcal{A}=R^n(q_{v*}\theta_v^*)\mathcal{A},$$
since $i_v^*$ is exact. By (\cite{SGA4} IV.5.8), $\theta_v$ is a
localization morphism, since $y(\widetilde{W}_{K_v})$ is a sub-group
of $y(W_{L/K,S})$ in $\mathcal{T}$.  It follows that $\theta_v^*$ is
exact and preserves injective objects. The associated spectral
sequence yields
$$R^n(q_{v*}\theta_v^*)\mathcal{A}=R^n(q_{v*})\theta_v^*\mathcal{A}.$$
The proposition follows.
\end{e-proof}

Recall that one has a projective system of topoi (see Remark
\ref{rem-fiber-flask-topos})
$$
\fonc{\mathfrak{F}_\bullet}{I/_K}{\underline{Topos}}{(L/K,S)}{\mathfrak{F}_{_{L/K,S}}}
$$
The \emph{total topos $Top\,(\mathfrak{F}_\bullet)$} is defined as
follows (see \cite{SGA4} VI.7.4). An object of
$Top\,(\mathfrak{F}_\bullet)$ is given by a family of objects
$\mathcal{F}_{_{L/K,S}}$ of $\mathfrak{F}_{_{L/K,S}}$ for
$(L/K,S)\in I/_K$, endowed with a system of compatible maps
$f_t:t^*\mathcal{F}_{_{L/K,S}}\rightarrow\mathcal{F}_{_{L'/K,S'}}$.
Here $t:\mathfrak{F}_{_{L'/K,S'}}\rightarrow\mathfrak{F}_{_{L/K,S}}$
is the morphism of topoi induced by the map
$(L'/K,S')\rightarrow(L/K,S)$ in $I/_K$. The maps $f_t$ are
compatible in the following sense. For any pair of transition maps
$$t\circ t':\mathfrak{F}_{_{L''/K,S''}}\longrightarrow\mathfrak{F}_{_{L'/K,S'}}\longrightarrow\mathfrak{F}_{_{L/K,S}},$$
one has
$$f_{t'}\circ t'^*(f_t)=f_{t\circ t'}:t'^* t^*\mathcal{F}_{_{L/K,S}}
\rightarrow
t'^*\mathcal{F}_{_{L'/K,S'}}\rightarrow\mathcal{F}_{_{L''/K,S''}}.$$
The arrows of the category $Top\,(\mathfrak{F}_\bullet)$ are the
obvious ones.
\begin{e-example}\label{Licht-sheaf-comes-from-flask-topos}
Denote by
$t_{L,S}:\mathfrak{F}_{_{W,\bar{Y}}}\longrightarrow\mathfrak{F}_{_{L/K,S}}$
the canonical morphism. Let $\mathcal{F}$ be an object of
$\mathfrak{F}_{_{W,\bar{Y}}}$ and let
$\mathcal{F}_{L,S}:=t_{L,S,*}\mathcal{F}$. One has
$\mathcal{F}_{L,S}=t_*\mathcal{F}_{L',S'}$ where
$t:\mathfrak{F}_{_{L'/K,S'}}\rightarrow\mathfrak{F}_{_{L/K,S}}$ is
the transition map. By adjunction, we have a map
$$f_t:t^*\mathcal{F}_{L,S}=t^*t_*\mathcal{F}_{L',S'}\longrightarrow\mathcal{F}_{L',S'}$$
These maps $f_t$ are compatible hence $(\mathcal{F}_{L,S},f_t)$ is
an object of $Top\,(\mathfrak{F}_\bullet)$.

Consider a discrete abelian group $A$. For any transition map
$t:\mathfrak{F}_{_{L'/K,S'}}\rightarrow\mathfrak{F}_{_{L/K,S}}$, we
have $t^*A= A$ since the morphism from $\mathfrak{F}_{_{L'/K,S'}}$
to the final topos $\underline{Set}$ is unique. Therefore any
discrete abelian group defines an abelian object of
$Top\,(\mathfrak{F}_\bullet)$ (which is the constant abelian object
of the topos $Top\,(\mathfrak{F}_\bullet)$ associated to $A$).

More generally, Proposition \ref{compatible-morphisms-flask-BR}
provides us with a morphism $Top\,(\mathfrak{F}_\bullet)\rightarrow
B_{\mathbb{R}}$ whose inverse image functor is given by
$$
\appl{B_{\mathbb{R}}}{Top\,(\mathfrak{F}_\bullet)}{\mathcal{F}}{(f^*_{_{L/K,S}}\mathcal{F})_{_{L/K,S}}}
.$$ Therefore, any abelian object $\mathcal{A}$ of $B_{\mathbb{R}}$
defines an abelian object $(f^*_{_{L/K,S}}\mathcal{A})$ of
$Top\,(\mathfrak{F}_\bullet)$. We denote this abelian object of
$Top\,(\mathfrak{F}_\bullet)$ also by $\mathcal{A}$.
\end{e-example}

\begin{prop}\label{relative-flask-total-flask}
Let $(L/K,S)$ be an element of $I/_K$. There is an essential
morphism
$$(\delta_!,\delta^*,\delta_*):\mathfrak{F}_{_{L/K,S}}\longrightarrow Top\,(\mathfrak{F}_\bullet),$$
whose inverse image is the functor
$$
\fonc{\delta^*}{Top\,(\mathfrak{F}_\bullet)}{\mathfrak{F}_{_{L/K,S}}}{(\mathcal{F}_{_{L,S}})_{_{(L,S)\in
I/_K}}}{\mathcal{F}_{_{L,S}}}
$$
Furthermore, $\delta_!$ is exact hence $\delta^*$ preserves
injective objects.
\end{prop}
\begin{proof}
This is (\cite{SGA4} VI. Lemme 7.4.12).
\end{proof}

\begin{e-defn}Let $\mathcal{A}=(\mathcal{A}_{_{L/K,S}},f_t)_{_{L/K,S}}$
be an abelian object of the total topos
$Top\,(\mathfrak{F}_\bullet)$. \emph{Lichtenbaum's Weil-étale
cohomology} with coefficients in $\mathcal{A}$ is defined as the
inductive limit
$$\underrightarrow{H}^i(\mathfrak{F}_{_{L/K,S}},\mathcal{A}):=\underrightarrow{lim}_{_{L/K,S}}H^i(\mathfrak{F}_{_{L/K,S}},\mathcal{A}_{_{L/K,S}}).$$
where $(L/K,S)$ runs over the set of finite Galois extensions and
finite $S$.
\end{e-defn}

We denote by $p_{L,S}:W_K\rightarrow W_{L/K,S}$ the canonical map
and also by $p_{L,S}:B_{W_K}\rightarrow B_{W_{L/K,S}}$ the induced
morphism of classifying topoi.
\begin{e-lem}
Let $\mathcal{A}$ be an abelian object of $B_{W_K}$ and define
$\mathcal{A}_{L,S}:=p_{L,S,*}\mathcal{A}$. The family
$(R^q(j_{_{L,S*}})\mathcal{A}_{L,S})$ defines an abelian object
$\underrightarrow{R}^qj_*\mathcal{A}$ of the total topos
$Top\,(\mathfrak{F}_\bullet)$.
\end{e-lem}
\begin{e-proof}
The diagram of topoi
\[ \xymatrix{
 B_{W_{L'/K,S'}}\ar[r]^{j_{_{L',S'}}} \ar[d]_{p}
  &\mathfrak{F}_{_{L'/K,S'}}\ar[d]_{t}   \\
 B_{W_{L/K,S}}\ar[r]^{j_{_{L,S}}}&
 \mathfrak{F}_{_{L/K,S}}
} \] is commutative. In other words, there is an isomorphism
$j_{_{L,S,*}}p_*\simeq t_*j_{_{L',S',*}}$. We get a transformation
$$t^*j_{_{L,S,*}}p_*\simeq t^*t_*j_{_{L',S',*}}\longrightarrow j_{_{L',S',*}}.$$
which is given by adjunction $ t^*t_*\rightarrow Id$. There is an
induced transformation
\begin{equation}\label{1-transformation}
t^*R^n(j_{_{L,S,*}}p_*)=R^n(t^*j_{_{L,S,*}}p_*)\longrightarrow
R^n(j_{_{L',S',*}}),
\end{equation}
where the identity comes from the exactness of $t^*$. Now, the Leray
spectral sequence
$$R^i(j_{_{L,S,*}})R^j(p_*)\Rightarrow R^{i+j}(j_{_{L,S,*}}p_*)$$
yields a natural transformation
\begin{equation}\label{2-transformation}
R^n(j_{_{L,S,*}})p_*\longrightarrow R^n(j_{_{L,S,*}}p_*).
\end{equation}
Composing (\ref{1-transformation}) and (\ref{2-transformation}), we
get a transformation
$$t^*R^n(j_{_{L,S,*}})p_*\longrightarrow R^n(j_{_{L',S',*}}).$$
On the other hand one has
$p_{*}\mathcal{A}_{L',S'}=\mathcal{A}_{L,S}$, since the diagram
\[ \xymatrix{
B_{W_{K}}   \ar[dr] \ar[r] & B_{W_{L'/K,S'}}\ar[d]   \\
& B_{W_{L/K,S}} } \] is commutative. Hence there is a canonical map
$$f_t:t^*R^n(j_{_{L,S,*}})\mathcal{A}_{L,S}=t^*R^n(j_{_{L,S,*}})p_*\mathcal{A}_{L',S'}\longrightarrow R^n(j_{_{L',S',*}})\mathcal{A}_{L',S'}.$$
This yields a system of compatible maps hence an abelian object of
$Top\,(\mathfrak{F}_\bullet)$.
\end{e-proof}
Consider for example a topological ${W_K}$-module $A$. Let
$\mathcal{A}$ be the abelian object of $B_{W_K}$ represented by $A$.
Then $\mathcal{A}_{L,S}$ is the object of $B_{W_{L/K,S}}$
represented by $A^{N_{L,S}}$, where $N_{L,S}$ is the kernel of the
map $p_{L,S}:W_K\rightarrow W_{L/K,S}$. The map $f_t$ is induced by
the inclusion $A^{N_{L,S}}\hookrightarrow A^{N_{L',S'}}$.
\begin{prop}\label{prop-spectral-sequence-Lichtenbaum}
Let $\mathcal{A}$ be an abelian object of $B_{W_K}$. There is a
spectral sequence
$$\underrightarrow{H}^p(\mathfrak{F}_{_{L/K,S}},\underrightarrow{R}^qj_*\mathcal{A})\Rightarrow
\underrightarrow{lim} \,H^{p+q}(B_{W_{L/K,S}},\mathcal{A}_{L,S}).$$ If
$\mathcal{A}$ is represented by a continuous discrete ${W_K}$-module
$A$ or by the topological group $\mathbb{R}$ with trivial action,
then we have a spectral sequence
$$\underrightarrow{H}^p(\mathfrak{F}_{_{L/K,S}},\underrightarrow{R}^qj_*A)\Rightarrow
H^{p+q}(B_{W_{K}},A).$$
\end{prop}
\begin{e-proof}
The composition
$B_{L/K,S}\rightarrow\mathfrak{F}_{_{L/K,S}}\rightarrow\underline{Set}$
yields a Leray spectral sequence
$$H^p(\mathfrak{F}_{_{L/K,S}},R^q(j_{_{L,S,*}})\mathcal{A}_{L,S})\Rightarrow
H^{p+q}(B_{W_{L/K,S}},\mathcal{A}_{L,S}).$$ Passing to the limit (which
is valid thanks to the previous Lemma), we get the first spectral
sequence of the proposition. By \cite{MatFlach} Lemma 10, one has
${\underrightarrow{lim}}_{L,S}H^p(B_{W_{L/K,S}},A^{N_{L,S}})=H^p(B_{W_{K}},A)$
for  $A=\mathbb{R}$ or $A$ a continuous discrete ${W_K}$-module.
This yields the second spectral sequence of the proposition.
\end{e-proof}

\subsection{Computation of the Weil-étale cohomology}
In what follows, we consider a non-trivial valuation $v$ of $K$.
\begin{notation}We denote by
$W_{K_v}^1$ the maximal compact sub-group of $W_{K_v}$. Hence
$W_{K_v}^1=I_v$ is the inertia subgroup for ultrametric $v$ and
$W_{K_v}^1\simeq\mathbb{S}^1$ for complex $v$. Let
$\widetilde{W}^1_{K_v}$ be the image of $W_{K_v}^1$ in $W_{L/K,S}$.
\end{notation}
Consider the commutative square
\[ \xymatrix{
 B_{\widetilde{W}_{K_v}}/_{(\widetilde{W}_{K_v}/\widetilde{W}_{K_v}^1)}\ar[r]^{} \ar[d]_{a}
  &B_{W_{k(v)}}/_{EW_{k(v)}}\ar[d]_{b}   \\
  B_{\widetilde{W}_{K_v}} \ar[r]^{q_v}&
 B_{W_{k(v)}}
} \] The first horizontal arrow is just (see \cite{SGA4} IV.5.8)
$$e_{\widetilde{W}_{K_v}^1}:
B_{\widetilde{W}_{K_v}^1}\simeq
B_{\widetilde{W}_{K_v}}/_{(\widetilde{W}_{K_v}/\widetilde{W}_{K_v}^1)}\longrightarrow
B_{W_{k(v)}}/_{EW_{k(v)}}\simeq \mathcal{T}.$$ The morphism $a$ and
$b$ are the localization morphisms and this square is a pull-back
(see \cite{SGA4} IV.5.8). It follows that the natural transformation
$$b^*\circ q_{v*}\rightarrow e_{\widetilde{W}_{K_v}^1*}\circ a^*$$
is an isomorphism. The localization functors $a^*$ and $b^*$ are
both exact and they preserve injective objets. We get
$$b^*\circ R^n(q_{v*})\simeq R^n(e_{\widetilde{W}_{K_v}^1*})\circ a^*.$$
In other words, $R^n(q_{v*})\mathcal{A}$ is the group object
$R^n(e_{\widetilde{W}_{K_v}^1*})$ of $\mathcal{T}$ endowed with its
natural action of $y({W}_{k(v)})$, since the functor
$b^*:B_{W_{k(v)}}\rightarrow\mathcal{T}$ is the forgetful functor
(sending an object $\mathcal{F}$ endowed with an action of
$y(W_{k(v)})$ to $\mathcal{F}$). Following the notations of
\cite{MatFlach}, we denote by
$\underline{H}^n(\widetilde{W}_{K_v}^1,\mathcal{A})$ the object
$R^n(q_{v*})\mathcal{A}$ of $B_{W_{k(v)}}$. We have
$$R^q(j_{_{L,S,*}})\mathcal{A}=(i_{v*})_{v\neq v_0}(\underline{H}^q(\widetilde{W}_{K_v}^1,\mathcal{A})_{v\neq
v_0}),$$ for any abelian object $\mathcal{A}$ of $B_{W_{L/K,S}}$ and any $q\geq1$, as
it follows from Proposition \ref{prop-higher-direct-im-generic} and
from the discussion above. But the direct image functor
$$(i_{v*})_{v\neq v_0}:\coprod_{v\neq v_0}B_{W_{k(v)}}\longrightarrow\mathfrak{F}_{_{L/K,S}}$$
is exact (as it follows from its explicit description) and preserves
injective objects. Therefore one has
$$H^p(\mathfrak{F}_{_{L/K,S}},R^q(j_{_{L,S,*}})\mathcal{A})
=H^p(\coprod_{v\neq
v_0}B_{W_{k(v)}},\,\underline{H}^q(\widetilde{W}_{K_v}^1,\mathcal{A})_{v\neq
v_0})=\prod_{v\neq
v_0}H^p(B_{W_{k(v)}},\underline{H}^q(\widetilde{W}_{K_v}^1,\mathcal{A})),$$
If $v$ is not in $S$ then $\widetilde{W}_{K_v}^1=0$ (see
\cite{Lichtenbaum} Lemma 3.7). We get
$\underline{H}^q(\widetilde{W}_{K_v}^1,\mathcal{A}))=0$ for $v$ not
in $S$. The next result follows.
\begin{prop}For any abelian object $\mathcal{A}$ of
$B_{W_{L/K,S}}$ and any $q\geq1$, we have
$$H^p(\mathfrak{F}_{_{L/K,S}},R^q(j_{_{L,S,*}})\mathcal{A})=\prod_{v\in S}H^p(B_{W_{k(v)}},\underline{H}^q(\widetilde{W}_{K_v}^1,\mathcal{A})).$$
\end{prop}
By \cite{MatFlach} Proposition 9.2, the sheaf
$\underline{H}^q(\widetilde{W}_{K_v}^1,\mathbb{Z})$ is represented
by the discrete $W_{k(v)}$-module
$H^q(\widetilde{W}_{K_v}^1,\mathbb{Z})$. Using \cite{MatFlach}
Proposition 8.1, we get
$$\underrightarrow{H}^p(\mathfrak{F}_{_{L/K,S}},\underrightarrow{R}^qj_*\mathbb{Z})=
{\underrightarrow{lim}}_{L,S}\sum_{v\in
S}H^p(W_{k(v)},H^q(\widetilde{W}_{K_v}^1,\mathbb{Z}))=\sum_{v\neq
v_0}H^p(W_{k(v)},H^q(W_{K_v}^1,\mathbb{Z})),$$ where the $(L/K,S)$
runs over the set of finite Galois extensions and finite $S$. We
have obtained the next corollary.
\begin{e-cor}\label{prop-pour-Lichten-cohomology} For any $q\geq1$, we have the
following identifications.
$$H^p(\mathfrak{F}_{_{W;\bar{Y}}},R^q(j_*)\mathbb{Z})=\prod_{v\neq
v_0}H^p(W_{k(v)},H^q(W_{K_v}^1,\mathbb{Z})).$$
$$\underrightarrow{H}^p(\mathfrak{F}_{_{L/K,S}},\underrightarrow{R}^qj_*\mathbb{Z})=
\sum_{v\neq v_0}H^p(W_{k(v)},H^q(W_{K_v}^1,\mathbb{Z})).$$
\end{e-cor}

Let $Pic(\bar{Y})$ be the Arakelov Picard group of $K$. This is the
group obtained by taking the idèle group of $K$ and dividing by the
principal idèles and the unit idèles. We denote by $Pic(Y)$ the
class group of $K$. Let $Pic^1(\bar{Y})$ be the kernel of the
absolute value map from $Pic(\bar{Y})$ to $\mathbb{R}^{>0}$. One has
an exact sequence of abelian compact groups
$$0\rightarrow \mathbb{R}^{r_1+r_2-1}/log(U_K/\mu_K)\rightarrow Pic^1(\bar{Y})\rightarrow Pic(Y)\rightarrow0$$
where $log(U_K/\mu_K)$ denotes the image of the logarithmic
embedding of the units modulo torsion $U_K/\mu_K$ in the kernel
$\mathbb{R}^{r_1+r_2-1}$ of the sum map
$\Sigma:\mathbb{R}^{r_1+r_2}\rightarrow\mathbb{R}$. By Pontryagin
duality, we obtain an exact sequence of discrete abelian groups (see
also \cite{Lichtenbaum} Proposition 6.4) :
\begin{equation}\label{exact-sequence-Pic1-D}
0\rightarrow Pic(Y)^{\mathcal{D}}\rightarrow
Pic^1(\bar{Y})^{\mathcal{D}}\rightarrow
Hom(U_K,\mathbb{Z})\rightarrow 0.
\end{equation}

\begin{e-thm}\label{thm-lichten-weil-etale-Zcoef}
One has
\begin{eqnarray*}
\underrightarrow{H}^n(\mathfrak{F}_{_{L/K,S}},\mathbb{Z})&=&\mathbb{Z}\mbox{ for $n=0$,}\\
&=&0\mbox{ for $n=1$,}\\
&=& Pic^1(\bar{Y})^{\mathcal{D}}\mbox{ for $n=2$,}\\
&=&\mu_K^{\mathcal{D}} \mbox{ for $n=3$.}
\end{eqnarray*}
The group
$\underrightarrow{H}^n(\mathfrak{F}_{_{L/K,S}},\mathbb{Z})$ is of
infinite rank for even $n\geq4$ and vanishes for odd $n\geq5$.
\end{e-thm}

\begin{e-proof}

The Cohomology of $W_{K_v}^1$ is given by class field theory for $v$
ultrametric. For a complex valuation $v$, we have
$H^q(W_{K_v}^1,\mathbb{Z})=H^q(\mathbb{S}^1,\mathbb{Z})$, which is
$\mathbb{Z}$ for $q$ even and $0$ for $q$ odd (this follows from
\cite{MatFlach} Prop. 5.2 and from the fact that the classifying
space of $\mathbb{S}^1$ is $\mathbb{C}P^{\infty}$). Moreover, one
has $H^{p}(B_{\mathbb{R}},\mathbb{Z})=0$ for any $q\geq1$ (see
\cite{MatFlach} Proposition 9.6). Therefore, Corollary
\ref{prop-pour-Lichten-cohomology} yields
\begin{eqnarray*}
\underrightarrow{H}^p(\mathfrak{F}_{_{L/K,S}},R^q(j_*)\mathbb{Z})&=&
\sum_{v\neq
v_0,\,v\nmid\infty}(\mathcal{O}_{K_v}^{\times})^{\mathcal{D}}\oplus\sum_{v\mid\infty}\mathbb{Z}\mbox{
for $p=0$ and $q=2$},\\
&=&\sum_{v\mid\infty}\mathbb{Z}\mbox{ for $p=0$ and
$q\geq4$ even},\\
 &=&0 \mbox{ otherwise.}
\end{eqnarray*}
Now the second spectral sequence of Proposition
\ref{prop-spectral-sequence-Lichtenbaum} for $A=\mathbb{Z}$ gives
$$\underrightarrow{H}^1(\mathfrak{F}_{_{L/K,S}},\mathbb{Z})=H^1(W_K,\mathbb{Z})=0.$$
Next, we obtain the exact sequence
$$0\rightarrow\underrightarrow{H}^2(\mathfrak{F}_{_{L/K,S}},\mathbb{Z})\rightarrow H^2(W_K,\mathbb{Z})
\rightarrow\sum_{v\nmid\infty}(\mathcal{O}_{K_v}^{\times})^{\mathcal{D}}\oplus\sum_{v\mid\infty}\mathbb{Z}\rightarrow
\underrightarrow{H}^3(\mathfrak{F}_{_{L/K,S}},\mathbb{Z})\rightarrow
0.
$$
This is the Pontryagin dual of
$$0\rightarrow\underrightarrow{H}^3(\mathfrak{F}_{_{L/K,S}},\mathbb{Z})^{\mathcal{D}}\rightarrow
\prod_{v\nmid\infty}\mathcal{O}_{K_v}^{\times}\times\prod_{v\mid\infty}\mathbb{S}^1\rightarrow
C_K^1\rightarrow
\underrightarrow{H}^2(\mathfrak{F}_{_{L/K,S}},\mathbb{Z})^{\mathcal{D}}\rightarrow
0.
$$
The result for $n\leq3$ follows. Furthermore, the spectral sequence
provides us with the exact sequence
$$0\rightarrow\underrightarrow{H}^n(\mathfrak{F}_{_{L/K,S}},\mathbb{Z})^{\mathcal{D}}\rightarrow
H^n(W_K,\mathbb{Z})\rightarrow\oplus_{v\mid\infty}\mathbb{Z}
\rightarrow\underrightarrow{H}^{n+1}(\mathfrak{F}_{_{L/K,S}},\mathbb{Z})\rightarrow
0,$$ for even $n\geq4$. Therefore, the result for $n\geq4$ follows
from the fact that the map
$H^n(W_K,\mathbb{Z})\rightarrow\oplus_{v\mid\infty}\mathbb{Z}$ is
surjective (see \cite{MatFlach} proof of Corollary 9).
\end{e-proof}

\begin{e-thm}\label{comparaison-limitcohomology-flask}
For $n\leq 1$ and $n\geq 4$, the canonical map
$$\underrightarrow{H}^n(\mathfrak{F}_{_{L/K,S}},\mathbb{Z})\longrightarrow
H^n(\mathfrak{F}_{_{W,\bar{Y}}},\mathbb{Z})$$ is an isomorphism.
Futhermore, there is an exact sequence
$$0\rightarrow{H}^2(\mathfrak{F}_{_{W,\bar{Y}}},\mathbb{Z})\rightarrow
(C_K^1)^{\mathcal{D}}
\rightarrow\prod_{v\nmid\infty}(\mathcal{O}_{K_v}^{\times})^{\mathcal{D}}\times\prod_{v\mid\infty}\mathbb{Z}\rightarrow
{H}^3(\mathfrak{F}_{_{W,\bar{Y}}},\mathbb{Z})\rightarrow 0.
$$
In particular, the canonical map
$$\underrightarrow{H}^n(\mathfrak{F}_{_{L/K,S}},\mathbb{Z})\longrightarrow
H^n(\mathfrak{F}_{_{W,\bar{Y}}},\mathbb{Z})$$ is not an isomorphism
for $n=2,3$.
\end{e-thm}
\begin{e-proof}
The morphism of topoi
$\mathfrak{F}_{_{W,\bar{Y}}}\rightarrow\mathfrak{F}_{_{L/K,S}}$
yields a map ${H}^n(\mathfrak{F}_{_{L/K,S}},\mathbb{Z})\rightarrow
H^n(\mathfrak{F}_{_{W,\bar{Y}}},\mathbb{Z})$. By the universal
property of the inductive limit we get a morphism
$$\underrightarrow{H}^n(\mathfrak{F}_{_{L/K,S}},\mathbb{Z})\longrightarrow
H^n(\mathfrak{F}_{_{W,\bar{Y}}},\mathbb{Z}).$$ Using Proposition
\ref{prop-transitivity-incl-generic} (with $L'=\bar{K}$), and
passing to the limit, we obtain a morphism of spectral sequences
$$[\underrightarrow{H}^p(\mathfrak{F}_{_{L/K,S}},\underrightarrow{R}^qj_*\mathbb{Z})\Rightarrow
H^p(B_{W_{K}},\mathbb{Z})]\longrightarrow[{H}^p(\mathfrak{F}_{_{W,\bar{Y}}},R^q(j_*)\mathbb{Z})\Rightarrow
H^p(B_{W_{K}},\mathbb{Z})].$$ Comparing these spectral sequences and
using the previous proof, we deduce the result.
\end{e-proof}

Let $\bar{V}=(V,V_{\infty})$ be an open sub-scheme of $\bar{Y}$. It
defines a sub-object of the final object of
$\mathfrak{F}_{_{L/K,S}}$. The open sub-topos
$$\mathfrak{F}_{_{L/K,S}}/_{\bar{V}}\longrightarrow\mathfrak{F}_{_{L/K,S}}$$ is
the full sub-category of $\mathfrak{F}_{_{L/K,S}}$ whose objects are
of the form $\mathcal{F}=(F_v;f_v)_{v\in \bar{V}}$ (i.e.
$F_v=\emptyset$ for $v\in\bar{Y}-\bar{V}$).
\begin{e-defn}\label{Not-Licht-cohomology-open}
We denote by
$H^n(\mathfrak{F}_{_{L/K,S}},\bar{V},-):=H^n(\mathfrak{F}_{_{L/K,S}}/_{\bar{V}},-)$
the cohomology of the open sub-topos
$\mathfrak{F}_{_{L/K,S}}/_{\bar{V}}$. Then we set
$$\underrightarrow{H}^n(\mathfrak{F}_{_{L/K,S}},\bar{V},\mathbb{Z})=\underrightarrow{lim}\,{H}^n(\mathfrak{F}_{_{L/K,S}},\bar{V},\mathbb{Z}).$$
\end{e-defn}

Let $L/K$ be finite extension of $K$. For any place $v$ of $L$, we
denote by $U_{L_v}$ the maximal compact sub-group of $L_v^{\times}$.
Hence $U_{L_v}=\mathcal{O}_{L_v}^{\times}$ for a finite place $v$,
$U_{L_v}=\mathbb{S}^1$ for $v$ complex and
$U_{L_v}=\mathbb{Z}/2\mathbb{Z}$ for a real place $v$.

\begin{e-defn}\label{defn-class-group-barU}
Let $\bar{V}=(V,V_{\infty})$ be a connected étale $\bar{Y}$-scheme
and let $L=K(\bar{V})$ be the number field corresponding to the
generic point of $\bar{V}$. We define the \emph{class group}
$C_{\bar{V}}$ associated to $\bar{V}$ by the exact sequence of
topological groups
$$0\rightarrow\prod_{v\in
\bar{V}}U_{L_v} \longrightarrow C_L\rightarrow
C_{\bar{V}}\rightarrow0 .$$ We denote by $C^1_{\bar{V}}$ the
topological sub-group of $C_{\bar{V}}$ defined by the kernel of the
canonical continuous morphism
$C_{\bar{V}}\rightarrow\mathbb{R}_{>0}$.
\end{e-defn}
Note that $C_{\bar{V}}$ is just the $S$-idèle class group of the
number field $K(\bar{V})$, where $S$ is the set of valuations of
$K(\bar{V})$ not corresponding to a point of $\bar{V}$. The group
$C^1_{\bar{V}}$ is compact.

\begin{prop}\label{Licht-cohomology-open-Z}
Let $\bar{V}\varsubsetneq\bar{Y}$ be a proper open sub-scheme of
$\bar{Y}$. Then one has
\begin{eqnarray*}
\underrightarrow{H}^n(\mathfrak{F}_{_{L/K,S}},\bar{V},\mathbb{Z})&=&\mathbb{Z} \mbox{ for }n=0,\\
                           &=&0 \mbox{ for $n$ odd},         \\
                           &=&(C^1_{\bar{V}})^{\mathcal{D}}\mbox{ for
                           }n=2.
\end{eqnarray*}
For even $n\geq 4$, we have an exact sequence
$$
0\longrightarrow
\underrightarrow{H}^n(\mathfrak{F}_{_{L/K,S}},\bar{V},\mathbb{Z})\longrightarrow
H^n(W_K,\mathbb{Z})\longrightarrow\sum_{v\in
V_{\infty}}\mathbb{Z}\longrightarrow 0.$$
\end{prop}
\begin{e-proof}
We use again the Leray spectral sequence induced by the inclusion of
the generic point. We get
$\underrightarrow{H}^1(\mathfrak{F}_{_{L/K,S}},\bar{V},\mathbb{Z})=H^1(W_K,\mathbb{Z})=0$
and the exact sequence
$$0\rightarrow\underrightarrow{H}^2(\mathfrak{F}_{_{L/K,S}},\bar{V},\mathbb{Z})\rightarrow
(C_K^1)^{\mathcal{D}}\rightarrow\coprod_{v\in
V}(\mathcal{O}_{K_v}^{\times})^{\mathcal{D}}\oplus\coprod_{v\in
V_{\infty}}\mathbb{Z} \rightarrow
\underrightarrow{H}^3(\mathfrak{F}_{_{L/K,S}},\bar{V},\mathbb{Z})\rightarrow
0.$$ Moreover the map
$$\prod_{v\in
\bar{V}}U_{K_v}=\prod_{v\in
V}\mathcal{O}_{K_v}^{\times}\times\prod_{v\in
V_{\infty}}\mathbb{S}^1 \longrightarrow C_K^1$$ is injective, since
$\bar{V}\varsubsetneq\bar{Y}$. By Pontryagin duality, we obtain
$\underrightarrow{H}^3(\mathfrak{F}_{_{L/K,S}},\bar{V},\mathbb{Z})=0$
and
$$\underrightarrow{H}^2(\mathfrak{F}_{_{L/K,S}},\bar{V},\mathbb{Z})=Ker[(C_K^1)^{\mathcal{D}}\rightarrow\coprod_{v\in
\bar{V}}U_{K_v}^{\mathcal{D}}]=(C_{\bar{V}}^1)^{\mathcal{D}}$$ Next
we obtain an exact sequence
$$0\rightarrow
\underrightarrow{H}^n(\mathfrak{F}_{_{L/K,S}},\bar{V},\mathbb{Z})\rightarrow
H^n(W_K,\mathbb{Z})\rightarrow\sum_{v\in
V_{\infty}}\mathbb{Z}\rightarrow
\underrightarrow{H}^{n+1}(\mathfrak{F}_{_{L/K,S}},\bar{V},\mathbb{Z})\rightarrow
H^{n+1}(W_K,\mathbb{Z})=0$$ for even $n\geq 4$. This ends the proof
since the map $H^n(W_K,\mathbb{Z})\rightarrow\sum_{v\in
V_{\infty}}\mathbb{Z}$ is surjective for even $n\geq 4$ (see
\cite{MatFlach} proof of Corollary 9).
\end{e-proof}

\begin{prop}\label{Lichtenbaum-cohomology-open-R}
For any open sub-scheme $\bar{V}\subseteq\bar{Y}$, one has
$$\underrightarrow{H}^n(\mathfrak{F}_{_{L/K,S}},\bar{V},\widetilde{\mathbb{R}})=\mathbb{R}
\mbox{ for $n=0,1$ and
}\underrightarrow{H}^n(\mathfrak{F}_{_{L/K,S}},\bar{V},\widetilde{\mathbb{R}})=0
\mbox{ for $n\geq2$}.$$
\end{prop}
\begin{e-proof}
Arguing as above and using the fact that
$H^n(W^1_{K_v},\widetilde{\mathbb{R}})=0$ for $n\geq1$ and any
non-trivial valuation $v$ since $W^1_{K_v}$ is compact (see
\cite{MatFlach} Corollary 8), we obtain
${H}^n(\mathfrak{F}_{_{L/K,S}},\bar{V},\widetilde{\mathbb{R}})=H^n(W_{L/K,S},\widetilde{\mathbb{R}})$
for any $\bar{V}\subseteq\bar{Y}$ and any pair $(L,S)$. Now the
product decomposition $W_{L/K,S}=W^1_{L/K,S}\times\mathbb{R}$ and
the compactness of $W^1_{L/K,S}$ show (see \cite{MatFlach}
Proposition 9.6):
$$H^n(W_{L/K,S},\widetilde{\mathbb{R}})=H^n(\mathbb{R},\widetilde{\mathbb{R}})=\mathbb{R}
\mbox{ for $n=0,1$ and 0 for $n\geq2$. }$$ Passing to the limit over
$(L,S)$, we obtain the result.
\end{e-proof}
\subsection{Cohomology with compact
support}\label{subsect-cpct-support}

The open sub-topos
$\phi:\mathfrak{F}_{_{L/K,S}}/_{Y}\rightarrow\mathfrak{F}_{_{L/K,S}}$
associated to the inclusion $Y\rightarrow\bar{Y}$ gives rise to
three adjoint functors $\phi_{L,S,!},\,\phi_{L,S}^*,\phi_{L,S,*}$.
The functor $\phi_{L,S,!}$ is the usual extension by 0. If
$\mathcal{A}=(A_v;f_v)_{v\in Y}$ is an abelian object of
$\mathfrak{F}_{_{L/K,S}}/_{Y}$, then $\phi_{L,S,!}\,\mathcal{A}$ is
the abelian object of $\mathfrak{F}_{_{L/K,S}}$ whose $v$-component
is $A_v$ for $v\in Y$, and $0$ for $v\in Y_{\infty}$. If there is no
risk of ambiguity, we denote by $\phi_{!},\,\phi^*,\phi_*$ the
functors defined above.

For any abelian object $\mathcal{A}=(A_v;f_v)_{v\in \bar{Y}}$ of
$\mathfrak{F}_{_{L/K,S}}$ we have an exact sequence
\begin{equation}\label{exact-sequence-flask-sheaves}
0\rightarrow
\phi_{!}\phi^*\mathcal{A}\rightarrow\mathcal{A}\rightarrow\prod_{Y_{\infty}}i_{v,*}A_v\rightarrow0.
\end{equation}
Moreover, $i_{v,*}$ is exact and preserves injectives, hence one has
\begin{equation}\label{coh-infty-flask}
H^n(\mathfrak{F}_{_{L/K,S}},\prod_{Y_{\infty}}i_{v,*}A_v)=H^n(\Coprod_{Y_{\infty}}B_{\mathbb{R}},\prod_{Y_{\infty}}i_{v,*}A_v)
=\prod_{Y_{\infty}}H^n(B_{\mathbb{R}},A_v).
\end{equation}

\begin{notation}
We denote by $\phi_!\mathbb{Z}$ (respectively
$\phi_!\widetilde{\mathbb{R}}$) the abelian object of
$Top\,(\mathfrak{F}_{\bullet})$ defined by the family of sheaves
$\phi_{L,S,!}\mathbb{Z}$ (respectively
$\phi_{L,S,!}\widetilde{\mathbb{R}}$) with the obvious transition
maps.
\end{notation}

Using the exact sequence of sheaves
(\ref{exact-sequence-flask-sheaves}), equation
(\ref{coh-infty-flask}), and passing to the limit we get the
following long exact sequences, for any open sub-scheme
$\bar{V}\subseteq\bar{Y}$ :
\begin{equation}\label{exact-seq-coh-support-flask-Z}
...\rightarrow
\underrightarrow{H}^n(\mathfrak{F}_{_{L/K,S}},\bar{V},\phi_!\mathbb{Z})\rightarrow
\underrightarrow{H}^n(\mathfrak{F}_{_{L/K,S}},\bar{V},\mathbb{Z})\rightarrow
\Prod_{V_{\infty}}H^n(B_{\mathbb{R}},\mathbb{Z})\rightarrow...
\end{equation}
\begin{equation}\label{exact-seq-coh-support-flask-R}
...\rightarrow
\underrightarrow{H}^n(\mathfrak{F}_{_{L/K,S}},\bar{V},\phi_!\widetilde{\mathbb{R}})\rightarrow
\underrightarrow{H}^n(\mathfrak{F}_{_{L/K,S}},\bar{V},\widetilde{\mathbb{R}})\rightarrow
\Prod_{V_{\infty}}H^n(B_{\mathbb{R}},\widetilde{\mathbb{R}})\rightarrow...
\end{equation}
By (\cite{MatFlach} Proposition 9.6), we have
$H^n(B_{\mathbb{R}},\mathbb{Z})=0$ for any $n\geq1$,
$H^n(B_{\mathbb{R}},\widetilde{\mathbb{R}})=\mathbb{R}$ for $n=0,1$
and $H^n(B_{\mathbb{R}},\widetilde{\mathbb{R}})=0$ for $n\geq2$.
This is enough to compute the groups
$\underrightarrow{H}^n(\mathfrak{F}_{_{L/K,S}},\bar{V},\phi_!\mathbb{Z})$
and
$\underrightarrow{H}^n(\mathfrak{F}_{_{L/K,S}},\bar{V},\phi_!\widetilde{\mathbb{R}})$
for any open $\bar{V}\subseteq\bar{Y}$. In particular, we recover
the result (\cite{Lichtenbaum} Theorem 6.3) for $\bar{V}=\bar{Y}$
and $n\leq3$.

\section{The category of sheaves on Lichtenbaum's Weil-étale site}

In this section we show that the topos $\mathfrak{F}_{_{L/K,S}}$ is
equivalent to the category of sheaves on Lichtenbaum's Weil-étale
site $T_{L/K,S}$.

\subsection{The local section site.}

The Weil-étale site $(T_{L/K,S},\mathcal{J}_{ls})$ is defined in
\cite{Lichtenbaum} using the groups $W_{L/K;S}$, where $L/K$ is a
finite Galois extension and $S$ a finite set of primes of $K$
containing the archimedean ones and the primes ramified in $L/K$.

\begin{e-defn}
An object of the category $T_{L/K,S}$ is a collection
$\mathcal{X}=(X_v;f_v)_{v\in\bar{Y}}$, where $X_{v}$ is a
$W_{k(v)}$-topological space and $f_v:X_{v}\rightarrow X_{v_0}$ is a
morphism of $W_{K_v}$-spaces for any $v\neq v_0$ (the topological
group $W_{K_v}$ acts continuously on $X_v$ and $X_{v_0}$ via the
morphisms $\theta_v:W_{K_v}\rightarrow W_K$ and
$q_v:W_{K_v}\rightarrow W_{k(v)}$ respectively). If $v=v_0$ we
require that the action of $W_K$ on $X_{v_0}$ factors through
$W_{L/K,S}$.

The morphisms in this category are defined in the obvious way. The
topology $\mathcal{J}_{ls}$ on the category $T_{L/K,S}$ is generated
by the pre-topology for which a cover is a family of morphisms
$\{\mathcal{X}_i\rightarrow\mathcal{X}\}$ such that $\{X_{i;v}
\rightarrow X_v\}$ is a local section cover, for any valuation $v$.
\end{e-defn}

M. Flach has shown in \cite{MatFlach} that the definition of
$H^i(W_K;A)$ as the direct limit
$\underrightarrow{lim}\,\,\,H^i(W_{L/K;S};A)$ coincides with the
topological group cohomology of $W_K$. Here, $A$ is a discrete
abelian group or the usual topological group $\mathbb{R}$ whith
trivial $W_K$-action (see \cite{MatFlach} Lemma 10). This suggests
the following definition.
\begin{e-defn}The \emph{local section site} $(T_{W;\bar{Y}};\mathcal{J}_{ls})$ is defined as
above, but the action of $W_K$ on the generic component $X_{v_0}$ of
an object $\mathcal{X}$ of $T_{W;\bar{Y}}$ is not supposed to factor
through $W_{L/K,S}$.
\end{e-defn}

For any topological group $G$, we denote by $B_{Top}G$ the category
of topological spaces (in a given universe) on which $G$ acts
continuously. The functor $B_{Top}\widetilde{W}_{K_v}\rightarrow
B_{Top}{W}_{K_v}$ induced by the surjective map
$W_{K_v}\rightarrow\widetilde{W}_{K_v}$ is fully faithful.
Therefore, an object of the category $T_{L/K,S}$ is given by a
collection $\mathcal{X}=(X_v;f_v)_{v\in\bar{Y}}$ where $f_v$ is a
map of $\widetilde{W}_{K_v}$-topological spaces.

The category $T_{W;\bar{Y}}$ (respectively $T_{L/K,S}$) has finite
projective limits. Indeed, the final object is given by the trivial
action of $W_{k(v)}$ on the one point space $X_v:=\{*\}$ for any
$v$, and by the trivial map $f_v:X_v\rightarrow X_{v_0}$. Let
$$\phi:\mathcal{U}=(U_v;f_v)\longrightarrow\mathcal{X}=(X_v;\xi_v)
\mbox{ and
}\phi':\mathcal{U'}=(U'_v;f'_v)\rightarrow\mathcal{X}=(X_v;\xi_v)$$
be two morphisms in $T_{W;\bar{Y}}$  (respectively in $T_{L/K,S}$).
The object $(U_v\times_{X_v}U'_v;\,f_v\times_{\xi_v}f'_v)_{v\in Y}$
does define a fiber product
$\mathcal{U}\times_{\mathcal{X}}\mathcal{U'}$ in the category
$T_{W;\bar{Y}}$ (respectively $T_{L/K,S}$).

\subsection{The local section site is a site for the flask topos.}\label{sect-Comparais-Lichten}

\begin{notation}
In order to simplify the notations, we assume in  this section
\ref{sect-Comparais-Lichten} that $L=\bar{K}/K$ is an algebraic
closure of $K$ and $S$ is the set of all places of $K$. However,
everything here remains valid for any suitable pair $(L/K,S)$.
\end{notation}
Let $v$ be a valuation of $K$. The Yoneda embedding yields fully
faithful functors
$$\varepsilon_v:B_{Top}W_{k(v)}\rightarrow B_{W_{k(v)}}\mbox{ and }
\varepsilon_{K_v}:B_{Top}W_{K_v}\rightarrow B_{W_{K_v}}.$$ Recall
that $\theta_v^*:B_{W_K}\rightarrow B_{W_{K_v}}$ and
$q_v^*:B_{W_{k(v)}}\rightarrow B_{W_{K_v}}$ denote the pull-backs of
the morphisms of classifying topoi induced by the Weil map
$\theta_v:W_{K_v}\rightarrow W_K$ and by the projection
$q_v:W_{K_v}\rightarrow W_{k(v)}$. In the following proof, we denote
also by
$$^{t}{\theta_v^*}:B_{Top}{W_K}\rightarrow B_{Top}{W_{K_v}} \mbox{
and }^{t}{q_v^*}:B_{Top}{W_{k(v)}}\rightarrow B_{Top}{W_{K_v}}$$ the
functors induced by $\theta_v$ and $q_v$. One has
\begin{equation}\label{identifications-pull-back-topol-yoneda}
q_v^*\circ\varepsilon_v=\varepsilon_{K_v}\circ {^t{q_v^*}} \mbox{
and
}\theta_v^*\circ\varepsilon_{v_0}=\varepsilon_{K_v}\circ{^t{\theta_v^*}}.
\end{equation}

\begin{prop}\label{prop-functor-loc-sect-site-flask-topos}
There is a fully faithful functor
$$\fonc{\mathrm{y}}{T_{W;\bar{Y}}}{\mathfrak{F}_{_{W;\bar{Y}}}}{\mathcal{X}=(X_v;f_v)}{\mathrm{y}(\mathcal{X})=(\varepsilon_v(X_v);\varepsilon_{K_v}(f_v))}.$$
\end{prop}

\begin{e-proof}

If $\mathcal{X}=(X_v;f_v)$ is an object of $T_{W;\bar{Y}}$, then
$f_v:{^{t}{q_v^*}}(X_v)\rightarrow{^{t}{\theta_v^*}}(X_{v_0})$ is a
map of $B_{Top}{W_{K_v}}$, for any valuation $v$. By
(\ref{identifications-pull-back-topol-yoneda}), the map
$$\varepsilon_{K_v}(f_v):q_v^*\circ\varepsilon_{v}(X_v)=\varepsilon_{K_v}\circ{^{t}{q_v^*}}(X_v)
\longrightarrow
\varepsilon_{K_v}\circ{^{t}{\theta_v^*}}(X_{v_0})=\theta_v^*\circ\varepsilon_{v}(X_{v_0})$$
is a morphism of $B_{W_{K_v}}$. Hence
$\mathrm{y}(\mathcal{X})=(\varepsilon_v(X_v);\varepsilon_{K_v}(f_v))$
is an object of $\mathfrak{F}_{_{W;\bar{Y}}}$ and $\mathrm{y}$ is a
functor. Let $\mathcal{X}=(X_v\,;\,\,\,f_v)_v$ and
$\mathcal{X'}=(X'_v\,;\,\,\,f'_v)_v$ be two objects of
$T_{W;\bar{Y}}$. Denote by
$${\mathrm{y}}_{(\mathcal{X};\mathcal{X'})}:Hom_{T_{W;\bar{Y}}}((X_v\,;\,\,\,f_v);(X'_v\,;\,\,\,f'_v))\longrightarrow
Hom_{\mathfrak{F}_{_{W;\bar{Y}}}}((\varepsilon_vX_v\,;\,\,\,\varepsilon_{K_v}f_v);(\varepsilon_vX_v\,;\,\,\,\varepsilon_{K_v}f_v))$$
the map defined by the functor $\mathrm{y}$. One has to show that
${\mathrm{y}}_{(\mathcal{X};\mathcal{X'})}$ is a bijection, for any
objects $\mathcal{X}$ and $\mathcal{X'}$. Let
$$\phi'=(\phi'_v)_v,\,\,\phi=(\phi_v)_v\,\,:(X_v\,;\,\,\,f_v)\rightrightarrows(X'_v\,;\,\,\,f'_v)$$
be two morphisms in $T_{W;\bar{Y}}$ such that
${\mathrm{y}}_{(\mathcal{X};\mathcal{X'})}(\phi')={\mathrm{y}}_{(\mathcal{X};\mathcal{X'})}(\phi)$.
One has $\varepsilon_v(\phi'_v)=\varepsilon_v(\phi_v)$ for any
$v\in\bar{Y}$. Since $\varepsilon_v$ is fully faithful, we get
$\phi'_v=\phi_v$ hence the map
${\mathrm{y}}_{(\mathcal{X};\mathcal{X'})}$ is injective. Let
$$\varphi=(\varphi_v)_v\,\,:(\varepsilon_vX_v\,;\,\,\,\varepsilon_{K_v}f_v)\rightarrow(\varepsilon_vX'_v\,;\,\,\,\varepsilon_{K_v}f'_v)$$
be a morphism in $\mathfrak{F}_{_{W;\bar{Y}}}$. For any
$v\in\bar{Y}$, there exists a unique morphism $\phi_v:X_v\rightarrow
X'_v$ such that $\varepsilon_v(\phi_v)=\varphi_v$ (since
$\varepsilon_v$ is fully faithful). The square
 \[ \xymatrix{
  q_v^*(\varepsilon_vX_v)\,\,\,\,\,\ar[d]_{\varepsilon_{K_v}(f_v)} \ar[r]^{q_v^*(\varepsilon_v(\phi_v))} & \,\,\,\,\,q_v^*(\varepsilon_vX'_v)\ar[d]^{\varepsilon_{K_v}(f'_v)}   \\
  \theta_v^*(\varepsilon_{v_0}X_{v_0})\,\,\,\,\,  \ar[r]_{\theta_v^*(\varepsilon_{v_0}(\phi_{v_0}))}& \,\,\,\,\,\theta_v^*(\varepsilon_{v_0}X'_{v_0})
} \] is commutative. By
(\ref{identifications-pull-back-topol-yoneda}), the following
diagram commutes as well
\[ \xymatrix{
  \varepsilon_{K_v}({^t{q_v^*}}(X_v))\,\,\,\,\,\ar[d]_{\varepsilon_{K_v}(f_v)} \ar[r]^{\varepsilon_{K_v}({^t{q_v^*}}(\phi_v))} &
  \,\,\,\,\,\varepsilon_{K_v}({^t{q_v^*}}(X'_v))\ar[d]^{\varepsilon_{K_v}(f'_v)}   \\
 \varepsilon_{K_v}({^t{\theta_v^*}}(X_{v_0}))\,\,\,\,\,  \ar[r]_{\varepsilon_{K_v}({^t{\theta_v^*}}(\phi_{v_0}))}&
 \,\,\,\,\,\varepsilon_{K_v}({^t{\theta_v^*}}(X'_{v_0}))
} \] Finally, the square
\[ \xymatrix{
{^t{q_v^*}}(X_v)\,\,\,\,\,\ar[d]_{f_v} \ar[r]^{{^t{q_v^*}}(\phi_v)} & \,\,\,\,\,{^t{q_v^*}}(X'_v)\ar[d]^{f'_v}   \\
{^t{\theta_v^*}}(X_{v_0})\,\,\,\,\,
\ar[r]_{{^t{\theta_v^*}}(\phi_{v_0})}&
\,\,\,\,\,{^t{\theta_v^*}}(X'_{v_0}) } \] is commutative, since
$\varepsilon_{K_v}$ is fully faithful. Hence
$\phi=(\phi_v)_v:\mathcal{X}\rightarrow\mathcal{X'}$ is a morphism
of $T_{W;\bar{Y}}$ such that
$$\mathrm{y}_{(\mathcal{X};\mathcal{X'})}(\phi)={\mathrm{y}}_{(\mathcal{X};\mathcal{X'})}((\phi_v)_v)=(\varepsilon_v\phi_v)_v=(\varphi_v)_v=\varphi.$$
The functor $\mathrm{y}$ is fully faithful, since the map
${\mathrm{y}}_{(\mathcal{X};\mathcal{X'})}$ is bijective for any
objects $\mathcal{X}$ and $\mathcal{X'}$ of $T_{W;\bar{Y}}$.
\end{e-proof}

For the notion of induced topology we refer to (\cite{SGA4}
III.3.1).

\begin{prop}\label{localsection-induced-by-canonical}
The local section topology $\mathcal{J}_{ls}$ on $T_{W;\bar{Y}}$ is
the topology induced by the canonical topology of
$\mathfrak{F}_{_{W;\bar{Y}}}$ via the functor $\mathrm{y}$.
\end{prop}

\begin{e-proof}
Recall that the coproduct of a family of topoi is, as a category,
the product of the underlying categories. Then consider the
following commutative diagram
 \[ \xymatrix{
\mathfrak{F}_{_{W;\bar{Y}}} \ar[r]^{{(i_v^*)}_v\,\,\,\,}  & \coprod_{v\in \bar{Y}} B_{W_{k(v)}}                         \\
T_{W;\bar{Y}}   \ar[r]^{}\ar[u]^{\mathrm{y}}          & \prod_{v\in
\bar{Y}} B_{Top}W_{k(v)}
  \ar[u]^{}  }
\]
The local section topology on $T_{W;\bar{Y}}$ is (by definition) the
topology induced by the local section topology on $\prod_{v\in
\bar{Y}} B_{Top}W_{k(v)}$ (see \cite{SGA4} III.3.4). By
\cite{MatFlach} Proposition 4.1, the local section topology on
$B_{Top}W_{k(v)}$ is the topology induced by the canonical topology
of $B_{W_{k(v)}}$. Hence the local section topology on $\prod
B_{Top}W_{k(v)}$ is the topology induced by the canonical topology
of $\coprod B_{W_{k(v)}}$. Since the previous diagram is
commutative, the local section topology on $T_{W;\bar{Y}}$ is the
topology induced by the canonical topology of $\coprod B_{W_{k(v)}}$
via the functor
$${(i^*_v)}_v\circ \mathrm{y}\,\,:T_{W;\bar{Y}}\longrightarrow \mathfrak{F}_{_{W;\bar{Y}}}
\longrightarrow\coprod_{v\in\bar{Y}} B_{W_{k(v)}}.$$ Hence, it
remains to show that the canonical topology of
$\mathfrak{F}_{_{W;\bar{Y}}}$ is induced by the canonical topology
of $\coprod B_{W_{k(v)}}$.

The functor
${(i^*_v)}_v:\mathfrak{F}_{_{W;\bar{Y}}}\rightarrow\coprod
B_{W_{k(v)}}$ is the pull-back of the morphism of topoi
${(i_v)}_v:\coprod B_{W_{k(v)}}\rightarrow
\mathfrak{F}_{_{W;\bar{Y}}}$. Then ${(i^*_v)}_v$ is a continuous
morphism of left exact sites, when $\mathfrak{F}_{_{W;\bar{Y}}}$ and
$\coprod B_{W_{k(v)}}$ are viewed as categories endowed with their
canonical topologies. This shows that the topology
$\mathcal{J}_{ind}$ on $\mathfrak{F}_{_{W;\bar{Y}}}$ induced by the
canonical topology of $\coprod B_{W_{k(v)}}$ is finer than the
canonical topology of $\mathfrak{F}_{_{W;\bar{Y}}}$, by definition
of the induced topology.

We need to show that any representable presheaf is a sheaf on the
site $(\mathfrak{F}_{_{W;\bar{Y}}};\mathcal{J}_{ind})$. Let
$\mathcal{F}={(F_v;f_v)}_v$ be an object of
$\mathfrak{F}_{_{W;\bar{Y}}}$, and let
$$\{u_i:\,\,\mathcal{X}_i=(X_{i;v};\xi_{i;v})\longrightarrow\mathcal{X}=(X_{v};\xi_{v})\}_{i\in I}$$
be a covering family of the site
$(\mathfrak{F}_{_{W;\bar{Y}}};\mathcal{J}_{ind})$. The functor
$$(\mathfrak{F}_{_{W;\bar{Y}}};\mathcal{J}_{ind})\longrightarrow(\coprod_{v\in\bar{Y}}
B_{W_{k(v)}};\mathcal{J}_{can})\longrightarrow(
B_{W_{k(v)}};\mathcal{J}_{can})$$ is continuous. Therefore (see
\cite{SGA4} III.1.6) the family
$$\{u_{i;v}:X_{i;v}\longrightarrow X_v\}_{i\in I}$$ is a covering
family of $( B_{W_{k(v)}};\mathcal{J}_{can})$ for any valuation $v$.
Since the covering families for the canonical topology of a topos
are precisely the epimorphic families, the family
$\{X_{i;v}\rightarrow X_v\}$ is epimorphic. Moreover, the pull-back
$q_v^*$ of the morphism $q_v:B_{W_{K_v}}\rightarrow B_{W_{k(v)}}$
preserves (as any pull-back) epimorphic families. Hence the family
$$\{ q_v^*(u_{i;v}):\,\,q_v^*(X_{i;v})\longrightarrow
q_v^*(X_v)\}_{i\in I}$$ is epimorphic in the category $B_{W_{K_v}}$.
Consider the diagram $\mathbb{D}$:
\[ \xymatrix{
  Hom({(X_v)}_v;{(F_v)}_v)  \ar[r]^{b} & \prod Hom({(X_{i;v})}_v;{(F_v)}_v)  & \rightrightarrows & \prod
Hom({(X_{i;v}\times_{X_v}X_{j;v})}_v;{(F_v)}_v) \\
  Hom(\mathcal{X};\mathcal{F}) \ar[u]_{a}  \ar[r]^{c}& \prod
  Hom(\mathcal{X}_i;\mathcal{F})\ar[u]_{d}  &\rightrightarrows &
  \prod
  Hom(\mathcal{X}_i\times_{\mathcal{X}}\mathcal{X}_j;\mathcal{F})\ar[u]
} \] The sets of homomorphisms in the first line correspond to the
category $\coprod B_{W_{k(v)}}$, and the set of homomorphisms of the
second line correspond to the category
$\mathfrak{F}_{_{W;\bar{Y}}}$. The vertical arrows are given by the
faithful functor ${(i_v^*)}_v$. Hence those vertical maps are all
injective. In particular, $a$ and $d$ are both injective.

The functor ${(i_v^*)}_v$ sends covering families of
$\mathfrak{F}_{_{W;\bar{Y}}}$ to covering families of $\coprod
B_{W_{k(v)}}$, since ${(i_v^*)}_v$ is continuous. Moreover, any
representable presheaf of $\coprod B_{W_{k(v)}}$ is a sheaf for the
canonical topology. This shows that the first line of $\mathbb{D}$
is exact. Hence, the maps $a$ and $b$ are both injective, which
shows that $c$ is injective.

Now, let ${(\varphi_i)}_i$ be an element of the kernel of
$$\prod Hom(\mathcal{X}_i;\mathcal{F})\rightrightarrows \prod
Hom(\mathcal{X}_i\times_{\mathcal{X}}\mathcal{X}_j;\mathcal{F}).$$
The square on the right hand side is commutative hence
$d({(\varphi_i)}_i)$ is in the kernel of
$$\prod Hom({(X_{i;v})}_v;{(F_v)}_v)   \rightrightarrows  \prod
Hom({(X_{i;v}\times_{X_v}X_{j;v})}_v;{(F_v)}_v).$$ Then we get an
element $\phi\in Hom({(X_v)}_v;{(F_v)}_v)$ which goes to
${d((\varphi_i)}_i)$, since the first line is exact. More precisely,
one has $b(\phi)={d((\varphi_i)}_i)$. Let $v$ be a non-trivial
valuation. For any $i\in I$, consider the diagram
\[ \xymatrix{
q_v^*(X_{i;v})\,\,\,\,\,\,\,\, \ar[d]_{\xi_{i;v}}
\ar[r]^{q_v^*(u_{i;v})}
& \,\,\,\,q_v^*(X_v) \ar[d]^{\xi_v} \ar[r]^{q_v^*(\phi_v)}\,\,\,\,  & \,\,\,\,q_v^*(F_v) \ar[d]^{f_v}\\
 \theta_v^*(X_{i;v_0})\,\,\,\,\,\,\,\, \ar[r]^{u_{i;v}}&
\,\,\,\,\theta_v^*(X_{v_0}) \ar[r]_{\theta_v^*(\phi_{v_0})} \,\,\,\,
&\,\,\,\,\theta_v^* (F_{v_0}) }
\]
Here, the total square and the left hand side square are both
commutative. Indeed, $$u_i\in
Hom_{\mathfrak{F}_{_{W;\bar{Y}}}}(\mathcal{X}_i;\mathcal{X}) \mbox{
and } \phi\circ u_i=\varphi_i\in
Hom_{\mathfrak{F}_{_{W;\bar{Y}}}}(\mathcal{X}_i;\mathcal{F}).$$ It
follows that the elements
$$\theta_v^*(\phi_{v_0})\circ\xi_v\mbox{ and }f_v\circ q_v^*(\phi_v)$$
of the set $Hom_{B_{W_{K_v}}}(q_v^*(X_v);\theta_v^*(F_{v_0}))$ have
the same image under the morphism
$$Hom_{B_{W_{K_v}}}(q_v^*(X_v);\theta_v^*(F_{v_0}))\longrightarrow
Hom_{B_{W_{K_v}}}(q_v^*(X_{i;v});\theta_v^*(F_{v_0})),$$ for any
$i\in I$. Hence, $\theta_v^*(\phi_{v_0})\circ\xi_v$ and $f_v\circ
q_v^*(\phi_v)$ have the same image
under the morphism
\begin{equation}\label{epimorphic}
Hom_{B_{W_{K_v}}}(q_v^*(X_{v});\theta_v^*(F_{v_0}))\longrightarrow
Hom_{B_{W_{K_v}}}(\coprod_{i\in
I}q_v^*(X_{i;v});\theta_v^*(F_{v_0})).
\end{equation}
Furthermore, the morphism (\ref{epimorphic}) is injective, since the
family $\{q_v^*(X_{i;v})\rightarrow q_v^*(X_v)\}$ is epimorphic. The
equality
$$\theta_v^*(\phi_{v_0})\circ\xi_v=f_v\circ
q_v^*(\phi_v)$$ follows. For any valuation $v\in \bar{Y}$, the
square
\[ \xymatrix{
q_v^*(X_v) \ar[d]^{\xi_v} \ar[r]^{q_v^*(\phi_v)}\,\,\,\,  & \,\,\,\,q_v^*(F_v) \ar[d]^{f_v}\\
\theta_v^*(X_{v_0}) \ar[r]_{\theta_v^*(\phi_{v_0})} \,\,\,\,
&\,\,\,\,\theta_v^* (F_{v_0}) }
\]
is commutative. In other words, the element $\phi\in Hom_{\coprod
B_{W_{k(v)}}}({(X_v)}_v;{(F_v)}_v)$ lies in
$$Hom_{\mathfrak{F}_{_{W;\bar{Y}}}}(\mathcal{X};\mathcal{F})\,\,\subseteq
\,\,Hom_{\coprod B_{W_{k(v)}}}({(X_v)}_v;{(F_v)}_v).$$ Hence there
exists a unique $\varphi\in
Hom_{\mathfrak{F}_{_{W;\bar{Y}}}}(\mathcal{X};\mathcal{F})$ such
that $a(\varphi)=\phi$. We get $$b\circ
a(\varphi)=b(\phi)=d({(\varphi_i)}_i)=d\circ c(\varphi)$$ and
$$c(\varphi)={(\varphi_i)}_i,$$ since $d$ is injective. This shows
that the second line of $\mathbb{D}$ is exact.

We have shown that the sequence
$$Hom_{\mathfrak{F}_{_{W;\bar{Y}}}}(\mathcal{X};\mathcal{F})\rightarrow
\prod_{i}
Hom_{\mathfrak{F}_{_{W;\bar{Y}}}}(\mathcal{X}_i;\mathcal{F})\rightrightarrows
\prod_{i;j}
Hom_{\mathfrak{F}_{_{W;\bar{Y}}}}(\mathcal{X}_i\times_{\mathcal{X}}\mathcal{X}_j;\mathcal{F})$$
is exact and that the first arrow is injective, for any object
$\mathcal{F}$ of $\mathfrak{F}_{_{W;\bar{Y}}}$ and any covering
family $\{\mathcal{X}_i\rightarrow\mathcal{X}\}_{i\in I}$ of the
site $(\mathfrak{F}_{_{W;\bar{Y}}};\mathcal{J}_{ind})$. Hence any
representable presheaf of the category $\mathfrak{F}_{_{W;\bar{Y}}}$
is a sheaf for the topology $\mathcal{J}_{ind}$. In other words the
topology $\mathcal{J}_{ind}$ is sub-canonical, that is, coarser than
the canonical topology. Since $\mathcal{J}_{ind}$ is also finer than
the canonical topology, it is the canonical topology.
\end{e-proof}

\begin{e-cor}
The functor $\mathrm{y}$ is continuous.
\end{e-cor}
\begin{e-proof}
By definition of the induced topology, $\mathcal{J}_{ls}$ is the
finest topology on $T_{W;\bar{Y}}$ such that $\mathrm{y}$ is
continuous (see \cite{SGA4} III.3.1).
\end{e-proof}

\begin{e-cor}\label{cor-loc-section-sub-canon}
The local section topology $\mathcal{J}_{ls}$ on $T_{W;\bar{Y}}$ is
sub-canonical.
\end{e-cor}
\begin{e-proof}
Let $\mathcal{X}$ be an object of $T_{W;\bar{Y}}$. The presheaf
$\widetilde{\mathrm{y}(\mathcal{X})}$ of
$\mathfrak{F}_{_{W;\bar{Y}}}$ represented by
$\mathrm{y}(\mathcal{X})$ is a sheaf, since
$\mathfrak{F}_{_{W;\bar{Y}}}$ is endowed with the canonical
topology. The restriction of $\widetilde{\mathrm{y}(\mathcal{X})}$
to the sub-category $T_{W;\bar{Y}}$ via the functor $\mathrm{y}$ is
a sheaf for the local section topology $\mathcal{J}_{ls}$, since
$\mathrm{y}$ is continuous. But this sheaf is canonically isomorphic
to the presheaf $\widetilde{\mathcal{X}}$ of $T_{W;\bar{Y}}$
represented by $\mathcal{X}$, since $\mathrm{y}$ is fully faithful.
Hence $\widetilde{\mathcal{X}}$ is a sheaf.
\end{e-proof}

We denote by $y:Top\rightarrow\mathcal{T}$ the Yoneda embedding. In
order to simplify the notations in the following proof, we also
denote by $y:B_{Top}{G}\rightarrow B_{G}$ the induced functor, for
any topological group $G$. By \cite{MatFlach} Corollary 3, the full
sub-category of $B_{W_{k(v)}}$ defined by the family of objects
$$\{y(W_{k(v)}\times Z); Z\in Ob(Top)\}$$
is a generating sub-category, for any valuation $v$. Here
$y(W_{k(v)})$ acts on $y(W_{k(v)}\times Z)=y(W_{k(v)})\times y(Z)$
on the first factor. Consider the sequence adjoint functors between
$B_{W_{K_v}}$ and $B_{W_K}$
$${\theta_v}_!\,\,\,;\,\,\,\,\,\,\theta_v^*\,\,\,;\,\,\,\,\,\,{\theta_v}_*$$
induced by the morphism of topological groups
$\theta_v:W_{K_v}\rightarrow W_K$. Recall that the functor
${\theta_v}_!$ is defined by
$${\theta_v}_!(\mathcal{F})=y(W_K)\times^{y(W_{K_v})}\mathcal{F}:=(y(W_K)\times \mathcal{F})/{y(W_{K_v})},$$
where $y(W_{K_v})$ acts on the left on $\mathcal{F}$ and by
right-translations on $y(W_K)$. For any $v\in \bar{Y}^0$ and any
$Z\in Ob(Top)$, we define
$$\mathcal{G}_{Z;v}=(({\theta_v}_!(q_v^*(y(W_{k(v)}\times
Z)))\,\,\,;\,\,\,y(W_{k(v)}\times Z)\,\,\, ;\,\,\,
(\emptyset_{B_{W_{k(w)}}})_{w\neq{v_0;v}})\,\,\,;\,\,\,\,(g_{Z;v})),$$
where the morphism
$$g_{Z;v}:q_v^*(y(W_{k(v)}\times Z))\longrightarrow\theta_v^*\circ{\theta_v}_!(q_v^*(y(W_{k(v)}\times Z)))$$
is given by adjunction. For the trivial valuation $v_0$ and for any
$Z\in Ob(Top)$, we define
$$\mathcal{G}_{Z;v_0}:=(y(W_{K}\times Z)\,\,\,;\,\,\,(\emptyset_{B_{W_{k(v)}}})_{v\neq v_0}).$$
Note that $T_{W;\bar{Y}}$ is equivalent to a full subcategory of
$\mathfrak{F}_{_{W;\bar{Y}}}$, by Proposition
\ref{prop-functor-loc-sect-site-flask-topos}.

\begin{prop}\label{prop-T_WY-generates-F_WY}
The category $T_{W;\bar{Y}}$ is a generating sub-category of
$\mathfrak{F}_{_{W;\bar{Y}}}$.
\end{prop}

\begin{e-proof}

It is shown in the proof of Proposition \ref{flasktopos-is-topos}
that the family
$$\{\mathcal{G}_{Z;v};\,\,\,;Z\in Ob(Top);\,\,\,v\in Y\}$$ is a
generating family of $\mathfrak{F}_{_{W;\bar{Y}}}$. Hence it remains
to show that $\mathcal{G}_{Z;v}$ lies in $T_{W;\bar{Y}}$, for any
$Z\in Ob(Top)$ and any $v\in \bar{Y}$. It is obvious for the trivial
valuation $v=v_0$. Take a non-trivial valuation $v\neq v_0$. One has
\begin{equation}
{\theta_v}_!(q_v^*(y(W_{k(v)}\times~
Z))):=y(W_K)\times^{y(W_{K_v})}y(W_{k(v)}\times~Z)=y(W_K\times^{W_{K_v}}(W_{k(v)}\times~Z)).
\end{equation}
as it is shown in the proof of (\cite{MatFlach} Lemma 13). Note that
$W_{K_v}$ acts on the right on $W_K$ and by left translation on the
first factor on $(W_{k(v)}\times~Z)$. This defines the quotient
space $(W_K\times^{W_{K_v}}(W_{k(v)}\times~Z))$. Then $W_{K}$ acts
on $(W_K\times^{W_{K_v}}(W_{k(v)}\times~Z))$ by left translations on
the first factor. We obtain
\begin{eqnarray}
\mathcal{G}_{Z;v}&=&({\theta_v}_!(q_v^*(y(W_{k(v)}\times
Z)));\,\,\,y(W_{k(v)}\times Z) ;\,\,\,
(\emptyset_{B_{W_{k(w)}}})_{w\neq{v_0;v}};\,\,\,g_{Z;v})  \\
  &=&\mathrm{y}\,\,(W_K\times^{W_{K_v}}(W_{k(v)}\times Z);\,\,\,(W_{k(v)}\times
  Z);\,\,\,
  (\emptyset_{Top})_{w\neq{v_0;v}};\,\,\,\widetilde{g_{Z;v}}\,)\,,
\end{eqnarray}
where $$\widetilde{g_{Z;v}}:(W_{k(v)}\times
Z)\longrightarrow~W_K\times^{W_{K_v}}(W_{k(v)}\times Z)$$ is the
unique map of $W_{K_v}$-topological spaces such that
$y(\widetilde{g_{Z;v}})=g_{Z;v}$ (recall that the Yoneda embedding
$y:B_{Top}W_{K_v}\rightarrow B_{W_{K_v}}$ is fully faithful). Hence
$\mathcal{G}_{Z;v}$ is an object of $T_{W;\bar{Y}}$ for any $Z\in
Ob(Top)$ and any $v\in \bar{Y}$.

\end{e-proof}

\begin{e-thm}\label{thm-principal-flasktopos}
The canonical morphism
$$\mathfrak{F}_{_{W;\bar{Y}}}\longrightarrow\widetilde{(T_{W;\bar{Y}};\mathcal{J}_{ls})}$$
is an equivalence of topoi, where
$\widetilde{(T_{W;\bar{Y}};\mathcal{J}_{ls})}$ is the category of
sheaves on the local section site. More generally, the canonical map
$$\mathfrak{F}_{_{L/K,S}}\longrightarrow\widetilde{(T_{L/K,S};\mathcal{J}_{ls})}$$
is an equivalence.
\end{e-thm}

\begin{e-proof}
The functor
$\mathrm{y}:T_{W;\bar{Y}}\rightarrow\mathfrak{F}_{_{W;\bar{Y}}}$ is
fully faithful, the topology $\mathcal{J}_{ls}$ is induced by the
canonical topology of $\mathfrak{F}_{_{W;\bar{Y}}}$ and
$T_{W;\bar{Y}}$ is a generating sub-category of
$\mathfrak{F}_{_{W;\bar{Y}}}$. The first claim of the theorem
follows from (\cite{SGA4} IV.1.2.1). More generally, the proofs of
(\ref{prop-functor-loc-sect-site-flask-topos}),
(\ref{localsection-induced-by-canonical}) and
(\ref{prop-T_WY-generates-F_WY}) remain valid for
$\mathfrak{F}_{_{L/K,S}}$, where $L/K$ and $S$ are both finite, by
replacing $W_{K}$ and $W_{K_v}$ with $W_{L/K,S}$ and
$\widetilde{W}_{K_v}$ respectively.
\end{e-proof}

\begin{e-cor}
The canonical map
$$\underrightarrow{lim}_{_{L/K,S}}H^n(T_{L/K,S},\mathbb{Z})\rightarrow
H^n(T_{W,\bar{Y}},\mathbb{Z})$$ is not an isomorphism for $n=2,3$.
\end{e-cor}
\begin{e-proof}
The cohomology of the site $T_{L/K,S}$ (respectively
$T_{W;\bar{Y}}$) is by definition the cohomology of the topos
$\widetilde{(T_{L/K,S};\mathcal{J}_{ls})}$ (respectively
$\widetilde{(T_{W;\bar{Y}};\mathcal{J}_{ls})}$) which is in turn
equivalent to $\mathfrak{F}_{_{L/K,S}}$ (respectively
$\mathfrak{F}_{_{W;\bar{Y}}}$) by Theorem
\ref{thm-principal-flasktopos}. Therefore this corollary is just a
reformulation of Theorem \ref{comparaison-limitcohomology-flask}.
\end{e-proof}

\begin{e-rem}
This corollary points out that Lichtenbaum's Weil-étale cohomology
is not defined as the cohomology of a site (i.e. of a topos).
\end{e-rem}

\section{The Artin-Verdier étale
topos}\label{section-Artin-Verdier-etaletopos}

The Artin-Verdier étale site associated to a number field takes the
(ramification at the) archimedean primes into account (see
\cite{Zink} and \cite{Bienenfeld}). This refinement of the étale
site is necessary if one wants to obtain naturally the vanishing of
the cohomology in degrees greater than three. Recall that $\bar{Y}$
is the set of valuations of a number field $K$.

\subsection{The Artin-Verdier étale site of
$\bar{Y}$.}\label{subsect-AV-etale-site}
Here, all schemes are separated and of finite type over
$Spec(\mathbb{Z})$. A \emph{connected $\bar{Y}$-scheme} is a pair
$\bar{X}=(X;X_{\infty})$, where $X$ is a connected $Y$-scheme in the
usual sense. When $X$ is empty, $X_{\infty}$ has to be (empty or) a
single point over $Y_{\infty}$. If $X$ is not empty, $X_{\infty}$ is
an connected open subset of $X(\mathbb{C})/\sim$. Here,
$X(\mathbb{C})/\sim$ is the quotient of the set of complex valued
points of $X$ under the equivalence relation defined by complex
conjugation, endowed with the quotient topology. A
\emph{$\bar{Y}$-scheme} is a finite sum of connected
$\bar{Y}$-schemes.

A \emph{connected étale $\bar{Y}$-scheme} is a connected
$\bar{Y}$-scheme $(X;X_{\infty})$, where $X/Y$ is étale of finite
presentation and $X_{\infty}/Y_{\infty}$ is unramified in the sense
that if $y\in Y_{\infty}$ is real, so is any point $x$ of
$X_{\infty}$ lying over $y$. An \emph{étale $\bar{Y}$-scheme}
$\bar{X}$ is a finite sum of connected étale $\bar{Y}$-schemes,
called the \emph{connected components} of $\bar{X}$. A morphism
$\bar{\phi}:(U;U_{\infty})\rightarrow (V;V_{\infty})$ of étale
$\bar{Y}$-schemes is given by a morphism $\phi:U\rightarrow V$ of
étale $Y$-schemes which induces a map
$\phi_{\infty}:U_{\infty}\rightarrow V_{\infty}$ over $Y_{\infty}$.
Fiber products
$\bar{U}\times_{\bar{X}}\bar{V}:=(U\times_{X}V;U_{\infty}\times_{X_{\infty}}V_{\infty})$
exist in the category $Et_{\bar{Y}}$ of étale $\bar{Y}$-schemes.

\begin{e-defn}
The \emph{Artin-Verdier étale site of $\bar{Y}$} is defined by the
category $Et_{\bar{Y}}$ of étale $\bar{Y}$-schemes endowed with the
topology $\mathcal{J}_{et}$ generated by the pre-topology for which
the coverings are the surjective families.
\end{e-defn}

\subsection{The specialization maps.}

Let $v$ be a valuation of $K$ corresponding to a point of $Y$. We
denote by $k(v)$ and $\overline{k(v)}$ the residue field at $v$ and
its algebraic closure. The henselization and the strict
henselization of $\bar{Y}$ at $v$ are defined as the projective
limits
$$Spec(\mathcal{O}^h_{\bar{Y};v}):=\underleftarrow{lim}\,\,\bar{U}\,\,\,\,\mbox{and}
\,\,\,\,Spec(\mathcal{O}^{sh}_{\bar{Y};v}):=\underleftarrow{lim}\,\,\bar{U},$$
where $\bar{U}$ runs over the filtered system of étale neighborhoods
of $v$ in $\bar{Y}$ and the filtered system of étale neighborhoods
of $\bar{v}$ in $\bar{Y}$ respectively. Here an étale neighborhood
of $v$ in $\bar{Y}$ (respectively of $\bar{v}$ in $\bar{Y}$) is
given by an étale $\bar{Y}$-scheme $\bar{U}$ endowed with a morphism
$(Spec(k(v));\emptyset)\rightarrow \bar{U}$ over $\bar{Y}$
(respectively endowed with a morphism
$(Spec(\overline{k(v)});\emptyset)\rightarrow \bar{U}$ over
$\bar{Y}$). Then for $v$ ultrametric, the local ring
$\mathcal{O}_{\bar{Y};v}^h:=\mathcal{O}_{\bar{K}}^{D_{v}}$ is
henselian with fraction field $K_v^h$ and with residue field $k(v)$.
Respectively, the local ring
$\mathcal{O}_{\bar{Y};v}^{sh}:=\mathcal{O}_{\bar{K}}^{I_{v}}$ is
strictly henselian with fraction field $K_v^{sh}$ and with residue
field $\overline{k(v)}$. For an archimedean valuation $v$, one has
$$(Spec(K_v^{sh});v)=(Spec(K_v^{h});v)=\underleftarrow{lim}\,\,\bar{U},$$
where $\bar{U}$ runs over the filtered system of $\bar{Y}$-morphisms
$(\emptyset;v)\rightarrow \bar{U}$. The choice of the valuation
$\bar{v}$ of $\bar{K}$ lying over $v$ induces an embedding
$$K_v^{sh}=\bar{K}^{I_v}\longrightarrow\bar{K}.$$
For any ultrametric valuation $v$, we get a \emph{specialization
map} over $\bar{Y}$
\begin{equation}\label{specialisationulrametric}
Spec(\bar{K})=(Spec(\bar{K});\emptyset)\longrightarrow
(Spec(\mathcal{O}^{sh}_{\bar{Y};v});\emptyset)=:\bar{Y}_{v}^{\,sh}.
\end{equation}
Such a specialization map over $\bar{Y}$ is also defined for a
archimedean valuation $v$ :
\begin{equation}\label{specializationarchimedean}
Spec(\bar{K})=(Spec(\bar{K});\emptyset)\longrightarrow
(Spec(K_v^{sh});v)=:\bar{Y}_{v}^{\,sh}.
\end{equation}
In what follows, $\bar{Y}_{v}^{\,sh}$ denotes the $\bar{Y}$-schemes
$(Spec(\mathcal{O}_{\bar{Y};v});\emptyset)$, $(Spec(K_v^{sh});v)$
and $Spec(\bar{K})=(Spec(\bar{K});\emptyset)$ for $v$ ultrametric,
archimedean and the trivial valuation respectively.

\subsection{The étale topos of $\bar{Y}$.}
The Artin-Verdier étale topos $\bar{Y}_{et}$ associated to the
Arakelov compactification of $Spec(\mathcal{O}_K)$ is the category
of sheaves of sets on the site $(Et_{\bar{Y}};\mathcal{J}_{et})$. We
denote by $Y_{et}$ the usual étale topos of the scheme $Y$.

\begin{prop}\label{YdansYbar-etale}
There is an open embedding
$$\varphi:Y_{et}\longrightarrow\bar{Y}_{et}$$
corresponding to the open inclusion
$Y:=(Y;\emptyset)\rightarrow\bar{Y}$. For any closed point $v$ of
$\bar{Y}$, there is a closed embedding (see \cite{SGA4} IV
Proposition 9.3.4)
$$u_v: B^{sm}_{G_{k(v)}}\longrightarrow\bar{Y}_{et}.$$
The closed complement of $Y_{et}$ in $\bar{Y}_{et}$ is the image of
the closed embedding
$$u:=(u_v)_{v\in Y_{\infty}}:\coprod_{v\in Y_{\infty}}\underline{Set}\longrightarrow\bar{Y}_{et}.$$
\end{prop}
\begin{e-proof}
The map $Y:=(Y;\emptyset)\rightarrow\bar{Y}$ is a monomorphism in
$Et_{\bar{Y}}$, hence the Yoneda embedding defines a sub-object
$\varepsilon(Y)$ of the final object of $\bar{Y}_{et}$. Thus the
localization morphism
\begin{equation}\label{YdansYbar}
{\bar{Y}_{et}}{_{/\varepsilon(Y)}}\longrightarrow\bar{Y}_{et}
\end{equation}
is an open embedding. The category $(Et_{\bar{Y}})_{/Y}$ is
isomorphic to the usual category $Et_Y$ of étale $Y$-schemes. Under
this identification, the usual étale topology $\mathcal{J}_{et}$ on
$Et_Y$ is the topology $\mathcal{J}_{ind}$ induced by the forgetful
functor
$$(Et_{\bar{Y}})_{_{/Y}}\rightarrow Et_{\bar{Y}},$$ where
$Et_{\bar{Y}}$ is endowed with the Artin-Verdier étale topology.
Moreover, one has an equivalence (see \cite{SGA4} III.5.4)
\begin{equation}\label{uneequivalence}
\widetilde{(Et_{Y};\mathcal{J}_{et})}\simeq\widetilde{((Et_{\bar{Y}})_{/Y};\mathcal{J}_{ind})}\simeq{\bar{Y}_{et}}{_{/\varepsilon(Y)}}.
\end{equation}
The first claim of the proposition follows from (\ref{YdansYbar})
and (\ref{uneequivalence}).

Assume that $v$ corresponds to an ultrametric valuation and denote
by $v\rightarrow\bar{Y}$ the morphism
$(Spec(k(v));\emptyset)\rightarrow\bar{Y}$. The functor
$$\fonc{u_v^*}{Et_{\bar{Y}}}{Et_{k(v)}\simeq T_{G_{k(v)}}^f}{(\bar{X}\rightarrow \bar{Y})}
{(\bar{X}\times_{\bar{Y}}v\rightarrow Spec(k(v)))}$$ is a morphism
of left exact sites, where $T_{G_{k(v)}}^f$ denotes the category of
finite $G_{k(v)}$-sets endowed with the canonical topology. We
denote by
$$u_v: B^{sm}_{G_{k(v)}}\longrightarrow\bar{Y}_{et}$$
the induced morphism of topoi. The adjunction transformation
$u_v^*\circ u_{v*}\rightarrow Id$ is an isomorphism (i.e. $u_v$ is
an embedding).

Assume now that $v$ is an archimedean valuation and denote by
$v\rightarrow\bar{Y}$ the morphism
$(\emptyset;v)\rightarrow\bar{Y}$. Again, the functor
$$\fonc{u_v^*}{Et_{\bar{Y}}}{\underline{Set}^f= T_{G_{k(v)}}^f}{\bar{X}\rightarrow \bar{Y}}
{\bar{X}\times_{\bar{Y}}v\rightarrow (\emptyset;v)}$$ is a morphism
of left exact sites, where $\underline{Set}^f$ is the category of
finite sets, endowed with the canonical topology. We get a embedding
of topoi
$$u_v: \underline{Set}\longrightarrow\bar{Y}_{et}.$$
For any $v\in\bar{Y}^0$, let $(\bar{Y}-v)\rightarrow\bar{Y}$ be the
open complement of the closed point $v$. Again,
$\varepsilon(\bar{Y}-v)$ is a sub-object of the final object of
$\bar{Y}_{et}$, which yields an open embedding
$j:{\bar{Y}_{et}}{_{/\varepsilon(\bar{Y}-v)}}\longrightarrow\bar{Y}_{et}$.
The strictly full subcategory of $\bar{Y}_{et}$ defined by the
objects $X$ such that $j^*(X)$ is the final object of
${\bar{Y}_{et}}{_{/\varepsilon(\bar{Y}_v)}}$ is exactly the
essential image of $u_{v*}$. In other words, the image of $u_{v}$ is
the closed complement of $j$. Hence, $u_v$ is a closed embedding.
The last claim of the proposition follows from (\cite{SGA4}
IV.9.4.6).
\end{e-proof}

\begin{e-cor}\label{conservative-family-etale-topos}
The family of functors
$$\{u_v^*:\bar{Y}_{et}\rightarrow
B^{sm}_{G_{k(v)}};\,\,v\in\bar{Y}^0\}$$ is conservative.
\end{e-cor}
\begin{e-proof}
By (\cite{SGA4} VIII.3.13), the family of functors
$$\{u_v^*:Y_{et}\rightarrow
B^{sm}_{G_{k(v)}};\,\,v\in Y^0\}$$ is conservative. By (\cite{SGA4}
IV 9.4.1.c), the result follows from Proposition
\ref{YdansYbar-etale}.
\end{e-proof}

Let $\mathcal{F}$ be an object of $Y_{et}$ and let
$\mathcal{F}_{v_0}\in Ob(B_{G_K}^{sm})$ be its generic stalk. For
any archimedean valuation $v$ one has
\begin{equation}\label{infinite-stalk-étale-sheaf}
u_v^*\circ\varphi_*(\mathcal{F})\simeq\mathcal{F}_{v_0}^{I_v}.
\end{equation}
Let $u:\coprod_{v\in
Y_{\infty}}\underline{Set}\rightarrow\bar{Y}_{et}$ be the morphism
given by the family $(u_v)_{v\in Y_{\infty}}$. Consider the functor
$$
\fonc{\rho:=u^*\varphi_*}{Y_{et}}{\coprod_{v\in Y_{\infty}}
\underline{Set}}{\mathcal{F}}{(\mathcal{F}_{v_0}^{I_v})_{v\in
Y_{\infty}}}
$$
Let us consider the category $(\coprod_{v\in Y_{\infty}}
\underline{Set}\,,Y_{et},\rho)$ defined in (\cite{SGA4} IV.9.5.1).

\begin{e-cor}\label{decomposition-topos-etale-Ybar}
The category $\bar{Y}_{et}$ is equivalent to $(\coprod_{v\in
Y_{\infty}} \underline{Set}\,,Y_{et},\rho)$.
\end{e-cor}
\begin{e-proof}
There is a functor
$$
\fonc{\Phi}{\bar{Y}_{et}}{(\coprod_{v\in Y_{\infty}}
\underline{Set}\,,Y_{et},\rho)}{\mathcal{F}}{(\varphi^*\mathcal{F},u^*\mathcal{F},f)}
$$
where $f:u^*\mathcal{F}\rightarrow u^*\varphi_*\varphi^*\mathcal{F}$
is given by adjunction. By (\cite{SGA4} IV.9.5.4.a) and Proposition
\ref{YdansYbar-etale}, the functor $\Phi$ is an equivalence of
categories.
\end{e-proof}

In particular, we have the usual sequences of adjoint functors
$$\varphi_!,\,\varphi^*,\,\varphi_* \mbox{ and } u^*,\,u_*,\,u^!.$$
between the categories of abelian sheaves on $\bar{Y}_{et}$,
$Y_{et}$ and $Y_{\infty}$ respectively. It follows that $u_*$ is
exact and $\varphi^*$ preserves injective objects since $\varphi_!$
is exact. For any abelian sheaf $\mathcal{A}$ on $\bar{Y}_{et}$, one
has the exact sequence
\begin{equation}\label{exact-sequ-etale-supp-cpct}
0\rightarrow\varphi_!\varphi^*\mathcal{A}\rightarrow\mathcal{A}\rightarrow
u_*u^*\mathcal{A}\rightarrow0,
\end{equation}
where the morphisms are given by adjunction. If $\mathcal{A}$ is a
sheaf on $Y_{et}$, the étale cohomology with compact support is
defined by
$H_c^n(Y_{et},\mathcal{A}):=H^n(\bar{Y}_{et},\varphi_!\mathcal{A})$.
To compute the étale cohomology with compact support, we use
(\ref{exact-sequ-etale-supp-cpct}) and we observe that the
cohomology of the sheaf $u_*u^*\mathcal{A}$ is trivial in degrees
$n\geq1$ since $u_*$ is exact.
For example, one has
\begin{equation}\label{cohom-compact-etale-Z}
H_c^n(Y_{et},\mathbb{Z})=0,(\prod_{Y_{\infty}}\mathbb{Z})/\mathbb{Z}
\mbox{ for $n=0,1$ and }
H_c^n(Y_{et},\mathbb{Z})=H^n(\bar{Y}_{et},\mathbb{Z}) \mbox{ for
$n\geq2$}.
\end{equation}
Consider now the constant étale sheaf associated to the discrete
abelian group $\mathbb{R}$. By Proposition
\ref{cohomology-artin-verdier-divisible} below, one has
\begin{equation}\label{cohom-compact-etale-R}
H_c^1(Y_{et},\mathbb{R})=(\prod_{Y_{\infty}}\mathbb{R})/\mathbb{R}
\mbox{ and } H_c^n(Y_{et},\mathbb{R})=0 \mbox{ for $n\neq1$.}
\end{equation}
\subsection{Artin-Verdier étale cohomology.}
Here we compute the Artin-Verdier étale cohomology with
$\mathbb{Z}$-coefficients. Let $j:Spec(K)\rightarrow Y\rightarrow
\bar{Y}$ be the inclusion of the generic point of $\bar{Y}$.
\begin{prop}\label{cohomology-artin-verdier-divisible}
For any uniquely divisible $G_K$-module $Q$, the sheaf $j_*Q$ on
$\bar{Y}_{et}$ is acyclic for the global sections functor, i.e.
$H^q(\bar{Y}_{et},j_*Q)=0$ for any $q\geq 1$. More generally, if $Q$
is a $G_K$-module such that $H^n(G_K,Q)=H^n(I_v,Q)=0$ for any
$n\geq1$ and any valuation $v$ of $K$, then the sheaf $j_*Q$ on
$\bar{Y}_{et}$ is acyclic for the global sections functor.
\end{prop}
\begin{e-proof}
Assume that $Q$ is uniquely divible The more general case follows
from the same argument. For any $v\in\bar{Y}^0$, one has (see
(\ref{infinite-stalk-étale-sheaf})) :
$$u_v^*\,R^q(j_*){Q}=R^q(u_v^*j_*){Q}=H^q(I_v;{Q}).$$ The groups $I_v$
are all profinite (or finite) and $Q$ is uniquely divisible. We
obtain
$$u_v^*\,R^q(j_*)Q=0$$ for any $q\geq 1$. By Corollary
\ref{conservative-family-etale-topos}, we get $R^q(j_*)(Q)=0$ for
any $q\geq 1$. Then the Leray spectral sequence
$$H^p(\bar{Y}_{et};R^q(j_*)Q)\Longrightarrow H^{p+q}(G_K;Q)$$
yields $$H^n(\bar{Y}_{et};j_*Q)\simeq H^{n}(G_K;Q)=0$$ for any
$n\geq1$, since Galois cohomology is torsion.

\end{e-proof}
Let $Pic(Y)$ and $U_K$ be the class-group of and the unit group of
$K$ respectively. Let $r_1$ be number of real primes of $K$. We
denote by $A^D=Hom(A;\mathbb{Q}/\mathbb{Z})$ the dual of a finitely
generated abelian group (or a profinite group) $A$. We consider the
idèle class group $C_K$ of $K$ and the connected component $D_K$ of
$1\in C_K$. The cohomology of the global Galois group $G_K$ with
coefficients in $\mathbb{Z}$ is trivial in odd degrees and we have
$H^q(G_K;\mathbb{Z})=\mathbb{Z}, (C_K/D_K)^{D},
(\mathbb{Z}/2\mathbb{Z})^{r_1}$ for $r=0$, $r=2$ and $r\geq4$ even
respectively (see \cite{Milne2} I. Corollary 4.6).

\begin{prop}\label{art-Verd-cohomology-of-Z}
The cohomology of the Artin-Verdier étale topos with coefficients in
$\mathbb{Z}$ is given by
\begin{eqnarray*}
H^q(\bar{Y}_{et};\mathbb{Z})&=&\mathbb{Z} \mbox{ for }q=0,\\
                           &=&0 \mbox{ for }q=1,         \\
                           &=& Pic(Y)^D\mbox{ for }q=2,\\
                           &=& U_K^D\mbox{ for }q=3,\\
                           &=& 0\mbox{ for } q\geq 4.
\end{eqnarray*}
\end{prop}

\begin{e-proof}As in the proof of Proposition \ref{cohomology-artin-verdier-divisible}, we get
$$u_v^*(R^q(j_*)\mathcal{L})=R^q(u_v^*j_*)\mathcal{L}=H^q(I_v;\mathcal{L})\,\,\in Ob(B^{sm}_{G_{k(v)}})$$
for any $\mathcal{L}\in Ob(B^{sm}_{G_{K}})$ and any $v\in\bar{Y}$
(recall that $I_{v_0}$ is trivial). In particular, one has
$j_*\mathbb{Z}=\mathbb{Z}$ and
$$j^*R^q(j_*)=R^q(j^*j_*)=R^q(Id)=0$$
for any $q\geq1$, since $j$ is an embedding. Moreover the map
$R^q(j_*)\mathcal{L}\rightarrow\Prod_{v\in\bar{Y}^0}u_{v*}u_v^*R^q(j_*)\mathcal{L}$
given by adjunction factors through
$\sum_{v\in\bar{Y}^0}u_{v*}H^q(I_v;\mathcal{L})$, since a cohomology
class in $H^q(G_K,\mathcal{L})$ is unramified almost everywhere. The
induced map
$$R^q(j_*)\mathcal{L}\longrightarrow\sum_{v\in\bar{Y}^0}u_{v*}H^q(I_v;\mathcal{L})$$
is an isomorphism using Corollary
\ref{conservative-family-etale-topos} and the fact that $u_v^*$
commutes with sums. We obtain $R^q(j_*)\mathbb{Z}=0$ for $q$ odd. By local class field
theory, we have
$$R^2(j_*)\mathbb{Z}=\sum_{v\in{Y}^0}u_{v*}(\mathcal{O}^{\times}_{K_v})^{D}\sum_{v\in
K(\mathbb{R})}u_{v*}(\mathbb{Z}/2\mathbb{Z})^{D}$$ and
$R^q(j_*)\mathbb{Z}=\sum_{v\in
K(\mathbb{R})}u_{v*}(\mathbb{Z}/2\mathbb{Z})^{D}$ for $q\geq4$ even.
The Leray spectral sequence
$$H^p(\bar{Y}_{et},R^q(j_*)\mathbb{Z})\Longrightarrow H^{p+q}(G_K,\mathbb{Z})$$
yields the exact sequence
$$0\rightarrow H^2(\bar{Y}_{et},\mathbb{Z})\rightarrow (C_K/D_K)^D\rightarrow
\sum_{v\nmid\infty}(\mathcal{O}^{\times}_{K_v})^{D}\sum_{v\in
K(\mathbb{R})}(\mathbb{Z}/2\mathbb{Z})^{D} \rightarrow
H^3(\bar{Y}_{et},\mathbb{Z})\rightarrow H^3(G_K,\mathbb{Z})=0$$
where the central map is the Pontryagin dual of the canonical
morphism .
$$\prod_{v\nmid\infty}(\mathcal{O}^{\times}_{K_v})\prod_{v\in
K(\mathbb{R})}\mathbb{Z}/2\mathbb{Z} \longrightarrow(C_K/D_K).$$ The
result follows for $q\leq3$. Next the Leray spectral sequence yields
$$0\rightarrow H^q(\bar{Y}_{et},\mathbb{Z})\rightarrow H^q(G_K,\mathbb{Z})\rightarrow
\sum_{v\in K(\mathbb{R})}\mathbb{Z}/2\mathbb{Z}\rightarrow
H^{q+1}(\bar{Y}_{et},\mathbb{Z})\rightarrow
H^{q+1}(G_K,\mathbb{Z})=0$$ for any even $q\geq4$. This ends the
proof since the map $H^q(G_K,\mathbb{Z})\rightarrow \sum_{v\in
K(\mathbb{R})}\mathbb{Z}/2\mathbb{Z}$ is an isomorphism.
\end{e-proof}
The cohomology groups $H^n(\bar{Y}_{et},\mathbb{Z})$ for $n=0,1,2$
can also be deduced from unramified class field theory (i.e.
$\pi_1(\bar{Y}_{et})^{ab}\simeq Cl(K)$), using Proposition
\ref{cohomology-artin-verdier-divisible}.

\section{The morphism from the flask topos to the étale
topos}\label{section-morphism-zeta}

In this section we describe the relation between the flask topos and
the étale topos. There is a morphism of topoi from the full flask
topos $\mathfrak{F}_{_{W;\bar{Y}}}$ to $\bar{Y}_{et}$. However this
morphism does not factor through $\mathfrak{F}_{_{L/K,S}}$, and we
have to decompose the étale topos as a projective limit in order to
understand the relation between the projective system of topoi
$\mathfrak{F}_{_{\bullet}}$ and $\bar{Y}_{et}$.

\subsection{The morphism from the étale site to the local section site.}

By Corollary \ref{cor-loc-section-sub-canon}, the local section
topology $\mathcal{J}_{ls}$ on the category $T_{W;\bar{Y}}$ is
sub-canonical. Since $T_{W;\bar{Y}}$ has finite projective limits,
$(T_{W;\bar{Y}};\mathcal{J}_{ls})$ is what we call a left exact
site.

\begin{prop}\label{prop-morph-sites-etale-loc-sections}
There exists a morphism of left exact sites
$$\fonc{\zeta^*}{(Et_{\bar{Y}};\mathcal{J}_{et})}{(T_{W;\bar{Y}};;\mathcal{J}_{ls})}{\bar{X}}{\mathcal{X}}.$$
\end{prop}
\begin{e-proof}
Let $\bar{X}$ be an étale $\bar{Y}$-scheme. For any valuation $v$, we define 
$$X_v:=Hom_{\bar{Y}}(\bar{Y}_{v}^{\,sh};\bar{X}).$$
Note that, for any ultrametric valuation $v$, the set
$$X_v=Hom_{\bar{Y}}(\bar{Y}_{v}^{\,sh};\bar{X})=Hom_{\bar{Y}}(Spec(\overline{k(v)});\bar{X})$$
carries an action of $G_{k(v)}$. For any archimedean valuation $v$,
$$X_v=Hom_{\bar{Y}}(\bar{Y}_{v}^{\,sh};\bar{X})=Hom_{\bar{Y}}((\emptyset;\bar{v});\bar{X}),$$
is just a set. For the trivial valuation $v=v_0$,
$$X_{v_0}=Hom_{\bar{Y}}(\bar{Y}_{v_0}^{\,sh};\bar{X})=Hom_{\bar{Y}}(Spec(\bar{K});\bar{X})$$ is
a $G_K$-set. For any valuation $v$, $X_v$ is viewed as a
$W_{k(v)}$-topological space. The morphisms
(\ref{specialisationulrametric}) and
(\ref{specializationarchimedean}) yield maps of $W_{K_v}$-spaces
$f_v:X_v\rightarrow X_{v_0}$, for any $v$. So we get an object
$\zeta^*(\bar{X})=\mathcal{X}$ of $T_{W;\bar{Y}}$. Clearly,
$\zeta^*$ is a functor. It preserves final objects and fiber
products, by the universal property of fiber products in the
category $Et_{\bar{Y}}$. Hence $\zeta^*$ is left exact. Furthermore,
an étale cover $\{\bar{X_i}\rightarrow\bar{X};\,i\in I\}$ yields a
surjective family of finite $G_{k(v)}$-sets $\{X_{i,v}\rightarrow
X_v;\,i\in I\}$ for any valuation $v$, hence a local section cover.
It follows that $\zeta^*$ is continuous and left exact.
\end{e-proof}

This morphism of left exact sites induces a morphism of topoi. Hence
the next result follows from Theorem \ref{thm-principal-flasktopos}.
\begin{e-cor}
There is a morphism of topoi
$\zeta:\mathfrak{F}_{_{W;\bar{Y}}}\rightarrow\bar{Y}_{et}$.
\end{e-cor}

The next proposition is an application of (Grothendieck) Galois
Theory. This result will not be used in the remaining part of this
paper. A proof can be found in \cite{these}.

\begin{prop}\label{prop-descript-et-topo-ds-loc-section}
The functor $\mathcal{\zeta^*}$ is fully faithful. The essential
image of $\mathcal{\zeta^*}$ is defined by the objects $\mathcal{X}$
of $T_{W;\bar{Y}}$ such that $X_{v_0}$ is a finite set, $f_v$ is
injective for any $v$ and bijective for almost all $v$ (i.e. except
for a finite number of non-trivial valuations). Finally, the étale
topology $\mathcal{J}_{et}$ on $Et_{\bar{Y}}$ is induced via
$\zeta^*$ by the local section topology $\mathcal{J}_{ls}$ on
$T_{W;\bar{Y}}$.
\end{prop}

\begin{e-rem}
Proposition \ref{prop-descript-et-topo-ds-loc-section} suggests to
define the "Weil-étale topology" as the full sub-category of
$T_{W;\bar{Y}}$ consisting of objects $(X_v,f_v)$ such that $f_v$ is
a topological immersion for every valuation $v$ and a homeomorphism
for almost all valuations. Then we endow this full  sub-category of
$T_{W;\bar{Y}}$ with the topology induced by the local section
topology via the inclusion functor.
\end{e-rem}

\subsection{Direct definition of the morphism $\zeta$.}
For any valuation $v$ of $K$, the specialization map
$\bar{Y}_{v}^{\,sh}\rightarrow\bar{Y}$ induces the co-specialization
map
\begin{equation}\label{co-specialization}
f_v:\mathcal{F}_{\bar{v}}\longrightarrow\mathcal{F}_{\bar{v}_0},
\end{equation}
for any étale sheaf $\mathcal{F}$ on $\bar{Y}$. Here
$\mathcal{F}_{\bar{v}}$ and $\mathcal{F}_{\bar{v}_0}$ denote the
stalks of the sheaf $\mathcal{F}$ at the geometric points
$\bar{v}\rightarrow\bar{Y}$ and $\bar{v}_0\rightarrow\bar{Y}$. The
map (\ref{co-specialization}) is $G_{K_v}$-equivariant and
functorial in $\mathcal{F}$. More precisely, denote by
$\mathfrak{q}_v:G_{K_v}\rightarrow G_{k(v)}$ the canonical
projection and by $\mathfrak{o}_v:G_{K_v}\rightarrow G_K$ the
morphism induced by the choice of the valuation $\bar{v}$ of
$\bar{K}$ lying over $v$. One has $u_v^*(\mathcal{F})\in
Ob(B_{G_{k(v)}}^{sm})$ and $u_{v_0}^*(\mathcal{F})\in
Ob(B_{G_{K}}^{sm})$. Then
$$f_v:\mathfrak{q}_v^*(u_v^*\mathcal{F})\longrightarrow\mathfrak{o}_v^*(u_{v_0}^*\mathcal{F})$$
is a map of $B_{G_{K_v}}^{sm}$, where we denote a morphism of
topological groups and the induced morphism of classifying topoi by
the same symbol. Since the squares

\[ \xymatrix{
  W_{K_v}\ar[d]_{\alpha_{K_v}} \ar[r]^{q_v} & W_{k(v)}\ar[d]_{\alpha_v}&
  W_{K_v}\ar[d]_{\alpha_{K_v}} \ar[r]^{{\theta}_v} & W_{K}\ar[d]_{\alpha_{v_0}} \\
  G_{K_v} \ar[r]^{\mathfrak{q}_v}&  G_{k(v)}& G_{K_v} \ar[r]^{\mathfrak{o}_v}&  G_{K}
} \] are both commutative, we get a morphism of $B_{W_{K_v}}$:
$$\alpha_{K_v}^*f_v:q_v^*(\alpha_v^*\circ u_v^*\mathcal{F})=\alpha_{K_v}^*\circ\mathfrak{q}^*_v(u_v^*\mathcal{F})\longrightarrow
\alpha_{K_v}^*\circ\mathfrak{o}^*_v(u_{v_0}^*\mathcal{F})=\theta_v^*(\alpha_{v_0}^*\circ
u_{v_0}^*\mathcal{F}).$$ We obtain an object
$$\zeta^*(\mathcal{F}):=(\alpha_v^*\circ u_v^*\mathcal{F}\,;\,\alpha_{K_v}^*f_v)_{v\in \bar{Y}}$$ of
the category $\mathfrak{F}_{_{W;\bar{Y}}}$. This yields a functor
$$\fonc{\zeta^*}{\bar{Y}_{et}}{\mathfrak{F}_{_{W;\bar{Y}}}}{\mathcal{F}}{(\alpha_v^*\circ
u_v^*\mathcal{F}\,;\,\alpha_{K_v}^*f_v)_{v\in\bar{Y}}}$$ Here the
equivariant map of $G_{K_v}$-sets
$$f_v:\mathfrak{q}_v^*(u_v^*\mathcal{F})\longrightarrow\mathfrak{o}_v^*(u_{v_0}^*\mathcal{F})$$
is induced by the usual co-specialization map between the stalks of
the étale sheaf $\mathcal{F}$.

\begin{prop}\label{prop-direct-def-morph-flsk-etale}
The functor
$\zeta^*:\bar{Y}_{et}\longrightarrow\mathfrak{F}_{_{W;\bar{Y}}}$ is
the inverse image of a morphism of topoi
$$\zeta:\mathfrak{F}_{_{W;\bar{Y}}}\longrightarrow\bar{Y}_{et}.$$
\end{prop}

\begin{e-proof}
Since the functors $\alpha_v^*\circ u_v^*$ and $\alpha_{K_v}^*$
commute with finite projective limits and arbitrary inductive
limits, so does the functor $\zeta^*$, by Proposition
\ref{limitfinies+colim-dans-flask-component-wise}. Since the functor
$\zeta^*$ is left exact and has a right adjoint, it is the pull-back
of a morphism of topoi
$\zeta:\mathfrak{F}_{_{W;\bar{Y}}}\longrightarrow\bar{Y}_{et}.$
\end{e-proof}
\subsection{Equivalence of the two definitions.}
Here we denote by
$$z:(Et_{\bar{Y}};\mathcal{J}_{et})\longrightarrow(T_{W;\bar{Y}};;\mathcal{J}_{ls})$$
the morphism of left exact sites defined in Proposition
\ref{prop-morph-sites-etale-loc-sections}. We have a commutative
square
 \[ \xymatrix{
 \mathfrak{F}_{_{W;\bar{Y}}} & \bar{Y}_{et}\ar[l]_{\zeta^*}   \\
 T_{W;\bar{Y}}\ar[u]^{\mathrm{y}}&Et_{\bar{Y}} \ar[u]^{\varepsilon}\ar[l]_{z}
} \] where $\zeta^*$ is defined in Proposition
\ref{prop-direct-def-morph-flsk-etale}, $\mathrm{y}$ is defined in
Proposition \ref{prop-functor-loc-sect-site-flask-topos} and
$\varepsilon:Et_{\bar{Y}}\rightarrow\bar{Y}_{et}$ is the Yoneda
embedding. By (\cite{SGA4} IV Proposition 4.9.4), the morphism of
topoi induced by the morphism of left exact sites $z$ is isomorphic
to the morphism of topoi
$\zeta:\mathfrak{F}_{_{W;\bar{Y}}}\rightarrow\bar{Y}_{et}$ of
Proposition \ref{prop-direct-def-morph-flsk-etale}.

\begin{prop}The morphism
$\zeta:\mathfrak{F}_{_{W;\bar{Y}}}\rightarrow\bar{Y}_{et}$ is not
connected (i.e. the inverse image functor $\zeta^*$ is not fully
faithful).
\end{prop}
The second definition of the morphism
$\zeta:\mathfrak{F}_{_{W;\bar{Y}}}\rightarrow\bar{Y}_{et}$ yields a
description of its inverse image functor $\zeta^*$. This can be used
to prove the proposition above (see \cite{these} Corollary 4.67).

\subsection{The morphisms $\zeta_{L,S}$.}\label{subsect-zeta-LS}

Let $L/K$ be a finite Galois sub-extension of $\overline{K}/K$. We
denote by $Et_{L/K}$ the full sub-category of $Et_{\bar{Y}}$
consisting of étale $\bar{Y}$-schemes $\bar{X}$ such that the action
of $G_K$ on the finite set
$$X_{v_0}=Hom_{\bar{Y}}(Spec(\overline{K});\bar{X})$$
factors through $G_{L/K}=Gal(L/K)$. This category is endowed with
the topology (again denoted by $\mathcal{J}_{et}$) induced by the
étale topology on $Et_{\bar{Y}}$ via the inclusion functor
$Et_{L/K}\rightarrow Et_{\bar{Y}}$. This functor yields a morphism
of left exact sites
$$(Et_{L/K},\mathcal{J}_{et})\longrightarrow (Et_{\bar{Y}},\mathcal{J}_{et})$$
and a morphism of topoi. These morphisms are compatible hence they
induce a morphism from $\bar{Y}_{et}$ to the projective limit topos
$\underleftarrow{lim}\,\widetilde{(Et_{L/K},\mathcal{J}_{et})}$,
where the limit is taken over all the finite Galois sub-extensions
of $\overline{K}/K$.
\begin{prop}
The canonical morphism
$$\bar{Y}_{et}\longrightarrow\underleftarrow{lim}\,\widetilde{(Et_{L/K},\mathcal{J}_{et})}.$$
is an equivalence.
\end{prop}
\begin{e-proof}
The morphism of let exact sites
$$(Et_{L/K},\mathcal{J}_{et})\longrightarrow
(Et_{L'/K},\mathcal{J}_{et})$$ is given by the inclusion functor,
for $\bar{K}/L'/L/K$. By (\cite{SGA4} VI 8.2.3), the direct limit
site
$$\underrightarrow{lim}(Et_{L/K},\mathcal{J}_{et}):=(\underrightarrow{lim}(Et_{L/K}),\mathcal{J})$$
is a site for the inverse limit topos
$\underleftarrow{lim}\,\widetilde{(Et_{L/K},\mathcal{J}_{et})}$. The
direct limit category $\underrightarrow{lim}(Et_{L/K})$ (see
\cite{Artin} III.3) is canonically equivalent to $Et_{\bar{Y}}$. The
topology $\mathcal{J}$ is the coarsest topology which makes all the
functors
$$(Et_{L/K},\mathcal{J}_{et})\longrightarrow(Et_{\bar{Y}},\mathcal{J})$$
continuous. In other words, $\mathcal{J}$ is the coarsest topology
on $Et_{\bar{Y}}$ such that any covering family of
$(Et_{L/K},\mathcal{J}_{et})$ is a covering family of
$Et_{\bar{Y}}$, for all $L/K$. Hence $\mathcal{J}$ is just the étale
topology, and $(Et_{\bar{Y}},\mathcal{J}_{et})$ is a site for the
inverse limit topos.
\end{e-proof}
\begin{prop}
There is a morphism of topoi $\zeta_{L,S}:
\mathfrak{F}_{_{L/K,S}}\rightarrow\widetilde{(Et_{L/K},\mathcal{J}_{et})}$.
Moreover, the diagram
 \[ \xymatrix{
 \mathfrak{F}_{_{L'/K,S'}}\ar[r]^{\zeta_{L',S'}}\ar[d] & \widetilde{(Et_{L'/K},\mathcal{J}_{et})}\ar[d]\\
 \mathfrak{F}_{_{L/K,S}}\ar[r]^{\zeta_{L,S}}&\widetilde{(Et_{L/K},\mathcal{J}_{et})}
} \] is commutative, for $L'/L/K$ and $S\subset S'$.
\end{prop}
\begin{e-proof}
The functor $$\zeta_{L,S}^*:Et_{L/K}\longrightarrow T_{L/K,S}$$
induced by $\zeta^*:Et_{\bar{Y}}\longrightarrow T_{W,\bar{Y}}$
yields a morphism of left exact sites, hence the first claim of the
proposition follows from Theorem \ref{thm-principal-flasktopos}. The
diagram of the proposition is commutative since the corresponding
diagram of sites is commutative.
\end{e-proof}

\begin{prop}\label{localistion-relative-flask-topos}
Let $\bar{V}$ be a connected étale $\bar{Y}$-scheme lying in the
category $Et_{L/K}$ (i.e. $G_{L/K}$ acts transitively on $V_{v_0}$).
One has an equivalence
$$\mathfrak{F}_{_{L/K,S}}/_{\bar{V}}\simeq\mathfrak{F}_{_{L/K(V),\widetilde{S}}}/_{\bar{V}},$$
where $K(V)$ is the function field of $\bar{V}$ and $\widetilde{S}$
is the set of places of $K(V)$ lying over $S$.
\end{prop}
\begin{e-proof}
The choice of a point of $V_{v_0}$ defines an isomorphism of
$W_{L/K,S}$-sets
$$V_{v_0}\simeq G_L/G_{K(V)}=W_{L/K,S}/W_{L/K(V),\widetilde{S}}$$
We get an isomorphism
$$B_{W_{L/K,S}}/y(W_{L/K,S},V_{v_0})\simeq B_{W_{L/K(V),\widetilde{S}}}.$$
The same result is valid for any closed point of $\bar{V}$ and the
proposition follows.
\end{e-proof}

\section{The spectral sequence relating Weil-étale cohomology to étale cohomology}

\subsection{Strongly compact topoi}

\begin{e-defn}
A topos $T$ is said to be \emph{strongly compact} if the functors
$H^n(T,-)$ commute with filtered colimits of abelian sheaves.
\end{e-defn}

Let $(T_i,f_{ji})_{i\in I}$ be a filtered projective system of
topoi, where the maps $f_{ji}:T_j\rightarrow T_i$ are the transition
maps. We denote by $T_{\infty}:=\underleftarrow{lim}\,T_i$ the limit
topos computed in the 2-category of topoi. We have canonical
morphisms $f_i:T_{\infty}\rightarrow T_i$. Suppose given an abelian
object $A_i$ of $T_i$ for any $i\in I$, and a family of morphism
$\alpha_{ij}:f_{ji}^*A_i\rightarrow A_j$ such that the following
condition holds :
$$\alpha_{ik}=\alpha_{jk}\circ f_{kj}^*(\alpha_{ij}):f_{ki}^*A_i=f_{kj}^*f_{ji}^*A_i\longrightarrow f_{kj}^*A_j\longrightarrow
A_k.$$ In what follows, the data $(A_i,\alpha_{ij})$ is said to be a
\emph{compatible system of abelian sheaves} on the projective system
of topoi $(T_i,f_{ji})_{i\in I}$.

The morphisms $f^*_j(\alpha_{ij})$ yield a filtered inductive system
of abelian objects $(f_i^*A_i)_{i\in I}$ in $T_{\infty}$, and we set
$$A_{\infty}:=\underrightarrow{lim}f_i^*A_i.$$

\begin{e-lem}
If the topos $T_i$ are all strongly compact, then the canonical
morphism
$$\underrightarrow{lim}\,H^n(T_i,A_i)\longrightarrow H^n(T_{\infty},A_{\infty})$$
is an isomorphism, for any integer $n$.
\end{e-lem}

\begin{proof}
By (\cite{SGA4} VI Corollaire 8.7.7) the topos $T_{\infty}$ is
strongly compact as well, and one has
$$H^n(T_{\infty},A_{\infty})=\underrightarrow{lim}\,H^n(T_{\infty},f_i^*A_i)
=\underrightarrow{lim}_{_{i\in
I}}\,(\underrightarrow{lim}_{_{j\rightarrow
i}}H^n(T_{j},f_{ji}^*A_i)).$$ We check easily that the canonical map
$$\underrightarrow{lim}_{_{i\in
I}}\,(\underrightarrow{lim}_{_{j\rightarrow
i}}H^n(T_{j},f_{ji}^*A_i))\longrightarrow\underrightarrow{lim}_{_{i\in
I}}\,H^n(T_i,A_i)$$ is an isomorphism. The result then follows from
the fact that the natural map
$$\underrightarrow{lim}_{_{i\in
I}}\,(\underrightarrow{lim}_{_{j\rightarrow
i}}H^n(T_{j},f_{ji}^*A_i))\longrightarrow
H^n(T_{\infty},A_{\infty})$$ factors through
$\underrightarrow{lim}_{_{i\in I}}\,H^n(T_i,A_i)$.
\end{proof}

Consider now the more general case where the sheaves $A_i$ are
replaced by bounded below complexes of abelian sheaves $C^*_i$.
Denote by $H^q(T_i,C^*_i)$ the hypercohomology of the complex of
sheaves $C_i^*$. We suppose given a compatible family of morphisms
of (bounded below) complexes $\alpha_{ij}:f_{ji}^*C^*_i\rightarrow
C^*_j$ for each transition map $f_{ji}:T_j\rightarrow T_i$, and we
define
$$C^*_{\infty}:=\underrightarrow{lim}f_i^*C^*_i.$$

\begin{e-lem}\label{lem-cohomology-lim-complex}
If the topos $T_i$ are all strongly compact, then the canonical
morphism
$$\underrightarrow{lim}\,H^n(T_i,C^*_i)\longrightarrow H^n(T_{\infty},C^*_{\infty})$$
is an isomorphism for any integer $n$.
\end{e-lem}

\begin{e-proof}
We denote by $H^q(C^*_i)$ (respectively $H^q(C^*_\infty)$) the
cohomology sheaf of the complex $C_i^*$ (respectively $C^*_\infty$)
in degree $q$. The inverse image functor $f_i^*$ is exact hence we
have $H^q(f_i^*C^*_i)=f_i^*H^q(C_i^*)$. By exactness of filtered
inductive limits, we obtain
$$H^q(C^*_\infty)=\underrightarrow{lim}H^q(f_i^*C_i)=\underrightarrow{lim}f_i^*H^q(C_i),$$
for any $q\geq0$. For any $i\in I$, we have a convergent spectral
sequence
$$H^p(T_i,H^q(C^*_i))\Longrightarrow H^{p+q}(T_i,C^*_i).$$
The compatible morphisms of complexes
$\alpha_{ij}:f_{ji}^*C^*_i\rightarrow C^*_j$ induce compatible
morphisms of spectral sequences, hence we have an inductive system
of spectral sequences. We obtain a morphism of spectral sequences
from
$$\underrightarrow{lim}H^p(T_i,H^q(C^*_i))\Longrightarrow \underrightarrow{lim}H^{p+q}(T_i,C^*_i)$$
to
$$H^p(T_\infty,H^q(C^*_\infty))\Longrightarrow
H^{p+q}(T_\infty,C^*_\infty).$$ By the previous lemma, this morphism
is an isomorphism at the $E_2$-term. It therefore induces
isomorphisms on the abutments. The result follows.
\end{e-proof}

Let $\bar{Y}$ be the set of valuations of the number field $K$, and
let $(C^*_{L},\alpha_u)$ be a compatible system of bounded below
complexes of abelian sheaves on the sites
$\widetilde{(Et_{L/K},\mathcal{J}_{et})}_{L}$ (i.e. a bounded below
complex of abelian objects in the total topos
$Top\,\widetilde{(Et_{L/K},\mathcal{J}_{et})}_{L}$). We denote by
$C^*_{\infty}$ the complex of sheaves on
$\bar{Y}_{et}\simeq\underleftarrow{lim}\widetilde{(Et_{L/K},\mathcal{J}_{et})}$
defined as above.

\begin{e-cor}\label{lem-cohomology-lim-complex-etale}
We have an isomorphism
$$\underrightarrow{lim}\,H^n(Et_{L/K},C^*_L)\simeq H^n(\bar{Y}_{et},C^*_{\infty})$$
where $L$ runs over the finite Galois sub-extensions of $\bar{K}/K$.
\end{e-cor}
\begin{proof}
The topoi $\widetilde{(Et_{L/K},\mathcal{J}_{et})}$ are all coherent
hence strongly compact (see \cite{SGA4} VI Cor. 5.2). Thus the
result follows from the previous lemma.
\end{proof}

\subsection{The spectral sequence}

\begin{e-thm}
Let $\mathcal{A}=(\mathcal{A}_{L,S},f_t)$ be an abelian object of
$Top(\mathfrak{F}_\bullet)$. There exists a bounded below complex
$R\mathcal{A}$ of abelian sheaves on $\bar{Y}_{et}$ and an
isomorphism
$$H^*(\bar{Y}_{et},R\mathcal{A})\simeq\underrightarrow{H}^*(\mathfrak{F}_{_{L/K,S}},\mathcal{A}),$$
where the left hand side is the étale hypercohomology of the complex
$R\mathcal{A}$. In particular, one has a spectral sequence relating
Lichtenbaum's Weil-étale cohomology to étale cohomology
$$H^p(\bar{Y}_{et},R^q\mathcal{A})\Longrightarrow\underrightarrow{H}^{p+q}(\mathfrak{F}_{_{L/K,S}},\mathcal{A}),$$
where $R^q\mathcal{A}$ is the cohomology sheaf of the complex
$R\mathcal{A}$ in degree $q$. The complex $R\mathcal{A}$ is well
defined up to quasi-isomorphism and functorial in $\mathcal{A}$.
\end{e-thm}

\begin{e-proof}
Since $Top\,(\mathfrak{F}_\bullet)$ is a topos, the abelian category
$Ab(Top\,(\mathfrak{F}_\bullet))$ has enough injectives. We choose
an injective resolution
$$0\rightarrow\mathcal{A}\rightarrow I^0_\bullet\rightarrow I^1_\bullet\rightarrow...$$
of the abelian objet $\mathcal{A}=(\mathcal{A}_{L,S},f_t)$. By
Proposition \ref{relative-flask-total-flask}, this resolution
provides us with an injective resolution
$$0\rightarrow\mathcal{A}_{L,S}\rightarrow I^0_{L,S}\rightarrow I^1_{L,S}\rightarrow...$$
of the abelian sheaf $\mathcal{A}_{L,S}$ on $\mathfrak{F}_{L,S}$,
for any pair $(L,S)$, and with a morphism of complexes
$$t^*I^*_{L,S}\longrightarrow I^*_{L',S'},$$
for any transition map $t$. These morphisms of complexes are
compatible in the usual way.

For any map $(L',S')\rightarrow(L,S)$ in $I/_K$, the following
diagram commutes
 \[ \xymatrix{
 \mathfrak{F}_{_{L'/K,S'}}\ar[r]^{\zeta_{L',S'}}\ar[d]^t & \widetilde{(Et_{L'/K},\mathcal{J}_{et})}\ar[d]^u\\
 \mathfrak{F}_{_{L/K,S}}\ar[r]^{\zeta_{L,S}}&\widetilde{(Et_{L/K},\mathcal{J}_{et})}
} \] In particular, we have
$u_*\circ\zeta_{L',S',*}\simeq\zeta_{L,S,*}\circ t_*$. By
adjunction, we obtain a natural (Beck-Chevalley) transformation
$$u^*\zeta_{L,S,*}\longrightarrow u^*\zeta_{L,S,*}t_*t^*\simeq u^*u_*\zeta_{L',S',*}t^*\longrightarrow\zeta_{L',S',*}t^*.$$
This transformation induces a morphism of complexes
\begin{equation}\label{une-transition}
u^*\zeta_{L,S,*}I^*_{L,S}\longrightarrow\zeta_{L',S',*}t^*I^*_{L,S}\longrightarrow\zeta_{L',S',*}I^*_{L',S'},
\end{equation}
where the last arrow is given by the morphism
$t^*I^*_{L,S}\longrightarrow I^*_{L',S'}$.

For any fixed Galois extension $L/K$, we have in particular a
filtered inductive system of complexes of sheaves
$(\zeta_{L,S,*}I^*_{L,S})_{S}$ in the topos
$\widetilde{(Et_{L/K},\mathcal{J}_{et})}$ . We set
$$I^*_L:=\underrightarrow{lim}_{_{S}}\,\zeta_{L,S,*}I^*_{L,S}.$$
For any transition map $u$, (\ref{une-transition}) induces a
morphism of complexes $u^*I^*_L\rightarrow I^*_{L'}$, since $u^*$
commutes with inductive limits. In other words, $(I^*_L)_L$ defines
a compatible system of complexes of sheaves on the sites
$(Et_{L/K})_L$.

By definition, the cohomology of the complex
$\zeta_{L,S,*}I^*_{L,S}$ in degree $n$ is the sheaf
$R^n(\zeta_{L,S,*})\mathcal{A}_{L,S}$ of the topos
$\widetilde{(Et_{L/K},\mathcal{J}_{et})}$. By exactness of filtered
inductive limits, the cohomology of the complex
$I^*_L:=\underrightarrow{lim}_{_{S}}\,\zeta_{L,S,*}I^*_{L,S}$ in
degree $n$ is the sheaf
$\underrightarrow{lim}_{_{S}}\,R^n(\zeta_{L,S,*})\mathcal{A}_{L,S}$.
We denote this sheaf by $\mathcal{A}_L^{(n)}$. Then we have
\begin{equation}\label{cohomology-sheaf-of-IL}
H^n(I^*_L)=\underrightarrow{lim}_{_{S}}\,H^n(\zeta_{L,S,*}I^*_{L,S})=\underrightarrow{lim}_{_{S}}\,R^n(\zeta_{L,S,*})\mathcal{A}_{L,S}
=:\mathcal{A}^{(n)}_{L}.
\end{equation}
Since the inverse image functor $u_L^*$ of the morphism
$u_L:\bar{Y}_{et}\rightarrow\widetilde{(Et_{L/K},\mathcal{J}_{et})}$
is exact, the cohomology of the complex $u_L^*(I^*_L)$ is given by
the sheaves
$$H^n(u_L^*(I^*_L))=
u_L^*H^n(I^*_L)=u_L^*\mathcal{A}^{(n)}_{L}.$$ Passing to the limit
over $L/K$, we define the complex
$$R\mathcal{A}:=\underrightarrow{lim}_{_{L/K}}\,u_L^*(I_L^*).$$
The cohomology sheaf of this complex in degree $n$ is given by :
\begin{equation}\label{cohomology-sheaf-of-RA}
R^n\mathcal{A}:=H^n(R\mathcal{A})=\underrightarrow{lim}_{_{L/K}}\,H^n(u_L^*I_L^*)=\underrightarrow{lim}_{_{L/K}}\,u_L^*H^n(I_L^*)
=\underrightarrow{lim}_{_{L/K}}\,u_L^*\mathcal{A}^{(n)}_{L}.
\end{equation}

Consider now a fixed Galois extension $L/K$. For any $S$,  the Leray
spectral sequence associated to the composition
$$\mathfrak{F}_{_{L/K,S}}\longrightarrow\widetilde{(Et_{L/K},\mathcal{J}_{et})}\longrightarrow\underline{Set}.$$
yields an isomorphism
$$H^p(Et_{L/K},\zeta_{L,S,*}I^*_{L,S})\simeq
H^{p+q}(\mathfrak{F}_{_{L/K,S}},\mathcal{A}_{L,S}),$$ where the
first term is the hypercohomology of the complex
$\zeta_{L,S,*}I^*_{L,S}$ on the site $Et_{L/K}$.

The complexes $(\zeta_{L,S,*}I^*_{L,S})_{S}$ form an inductive
system whose colimit is $I^*_L$, when $S$ runs over the finite sets
of valuations of $K$ containing the archimedean ones and those which
ramify in $L$. Passing to the limit over $S$, we obtain an
isomorphism
\begin{equation}\label{compatible1}
H^n(Et_{L/K},I^*_{L})\simeq\underrightarrow{lim}_{_{S}}H^n(Et_{L/K},\zeta_{L,S,*}I^*_{L,S})\simeq
\underrightarrow{lim}_{_{S}}
H^{n}(\mathfrak{F}_{_{L/K,S}},\mathcal{A}_{L,S}).
\end{equation}
Here, the first isomorphism follows from Lemma
\ref{lem-cohomology-lim-complex} (taking $T_i$ to be constantly
$\widetilde{(Et_{L/K},\mathcal{J}_{et})}$).

We have shown above that the family of complexes $I_L$ forms a
compatible system, when $L/K$ runs over the set of finite Galois
sub-extensions of $\overline{K}/K$. Passing to the limit over $L/K$,
we obtain
\begin{equation}\label{compatible2}
H^n(\bar{Y}_{et},R\mathcal{A})=H^n(\bar{Y}_{et},\underrightarrow{lim}_{_{L}}u_L^*(I^*_{L}))\simeq
\underrightarrow{lim}_{_{L}}H^n(Et_{L/K},I^*_{L}).
\end{equation}
by Corolarry \ref{lem-cohomology-lim-complex-etale}. Therefore,
isomorphisms (\ref{compatible1}) and (\ref{compatible2}) yield
$$H^n(\bar{Y}_{et},R\mathcal{A})\simeq\underrightarrow{lim}_{_{L}}H^n(Et_{L/K},I^*_{L})
\simeq\underrightarrow{lim}_{_{L}}\underrightarrow{lim}_{_{S}}
H^{n}(\mathfrak{F}_{_{L/K,S}},\mathcal{A}_{L,S})\simeq\underrightarrow{H}^{n}(\mathfrak{F}_{_{L/K,S}},\mathcal{A}),$$
for any $n\geq0$.

Let us show now that the complex of étale sheaves $R\mathcal{A}$ is
well defined up to quasi-isomorphism. This complex has been defined
by an injective resolution $I_{\bullet}$ of $\mathcal{A}$ in
$\mathfrak{F}_{\bullet}$. Let $I^*_\bullet$ and $J^*_\bullet$ be two
injective resolutions of $\mathcal{A}$. Denote by
$R\mathcal{A}(I_{\bullet})$ and $R\mathcal{A}(J_{\bullet})$ the
étale complexes defined as above. There is a morphism
$$q_{\bullet}:I^*_\bullet\longrightarrow J^*_\bullet$$
of complexes of abelian objects in $\mathfrak{F}_{\bullet}$ well
defined up to homotopy. For any pair $(L,S)$, we have in particular
a morphism
$$q_{L,S}:I^*_{L,S}\longrightarrow J^*_{L,S}$$
over $Id_{\mathcal{A}_{L,S}}$. Applying $\zeta_{L,S,*}$, the
morphism $q_{L,S}$ induces a quasi-isomorphism
$$\zeta_{L,S,*}I^*_{L,S}\longrightarrow \zeta_{L,S,*}J^*_{L,S},$$
since
$$R^q(\zeta_{L,S,*})\mathcal{A}_{L,S}:=H^q(\zeta_{L,S,*}I^*_{L,S})\simeq
H^q(\zeta_{L,S,*}J^*_{L,S}).$$ If $L/K$ is fixed, the morphisms
$q_{L,S}$ induce a morphism of complexes
$$q_L:I^*_L:=\underrightarrow{lim}_{_{S}}\,\zeta_{L,S,*}I^*_{L,S}\longrightarrow J^*_L:=\underrightarrow{lim}_{_{S}}\,\zeta_{L,S,*}J^*_{L,S}$$
which is a quasi-isomorphism by (\ref{cohomology-sheaf-of-IL}).
Passing to the limit over $L/K$, the morphisms $q_L$ induce a
morphism
$$R\mathcal{A}(I_{\bullet}):=\underrightarrow{lim}_{_{L/K}}\,u_L^*(I_L^*)\longrightarrow \underrightarrow{lim}_{_{L/K}}\,u_L^*(J_L^*)=:R\mathcal{A}(J_{\bullet})$$
This is a quasi-isomorphism by (\ref{cohomology-sheaf-of-RA}), hence
$R\mathcal{A}$ is well defined up to quasi-isomorphism.

\end{e-proof}

Let $\mathcal{A}=(\mathcal{A}_{L,S},f_t)$ be an abelian object of
$Top(\mathfrak{F}_\bullet)$ and let $\bar{U}$ be an étale
$\bar{Y}$-scheme. If $L/K$ is big enough so that the $G_K$-action on
the finite set $U_{v_0}$ factors through $G(L/K)$, then $\bar{U}$
defines an object $\zeta^*(\bar{U})$ of $\mathfrak{F}_{_{L/K,S}}$.
We consider the cohomology groups
$$H^*(\mathfrak{F}_{_{L/K,S}},\bar{U},\mathcal{A}_{L,S}):=H^*(\mathfrak{F}_{_{L/K,S}}{/_{\zeta^*(\bar{U})}},\mathcal{A}_{L,S}\times\zeta^*(\bar{U}))$$
of the slice topos $\mathfrak{F}_{_{L/K,S}}{/_{\zeta^*(\bar{U})}}$.
For any étale $\bar{Y}$-scheme $\bar{U}$, the direct limit
$$\underrightarrow{H}^*(\mathfrak{F}_{_{L/K,S}},\bar{U},\mathcal{A}):=\underrightarrow{lim}_{_{L,S}}H^*(\mathfrak{F}_{_{L/K,S}},\bar{U},\mathcal{A}_{L,S})$$
is well defined. It follows from Proposition
\ref{localistion-relative-flask-topos} that the computation of these
cohomology groups can be reduced to the case of an open sub-scheme
$\bar{U}$ of $\bar{Y}$, as defined in Notation
\ref{Not-Licht-cohomology-open}. Therefore, one can apply
Proposition \ref{Licht-cohomology-open-Z} and Proposition
\ref{Lichtenbaum-cohomology-open-R} to obtain explicit computations.

\begin{prop}\label{lem-RqZ-associe}
The sheaf $R^q\mathcal{A}$ is the sheaf associated to the presheaf
$$
\fonc{\mathcal{P}^q\mathcal{A}}{Et_{\bar{Y}}}{\underline{Ab}}{\bar{U}}
{\underrightarrow{H}^{q}(\mathfrak{F}_{_{L/K,S}},\bar{U},\mathcal{A})}
.$$
\end{prop}
\begin{proof}
The sheaf $R^q(\zeta_{L,S,*})\mathcal{A}_{L,S}$ is the sheaf
associated to the presheaf
$$
\fonc{\mathcal{P}_{L,S}^q}{Et_{L/K}}{\underline{Ab}}{\bar{U}}
{H^{q}(\mathfrak{F}_{_{L/K,S}},\bar{U},\mathcal{A}_{L,S})}.$$ It
follows that the sheaf
$$\mathcal{A}^{(q)}_L:=\underrightarrow{lim}_{_{S}}\,R^q(\zeta_{L,S,*})\mathcal{A}_{L,S}$$
is the sheaf associated to the presheaf
$$
\fonc{{\mathcal{P}_{L}^q:=\underrightarrow{lim}_{_{S}}\mathcal{P}_{L,S}^q}}{Et_{L/K}}{\underline{Ab}}{\bar{U}}
{\underrightarrow{lim}_{_{S}}\,H^{q}(\mathfrak{F}_{_{L/K,S}},\bar{U},\mathcal{A}_{L,S})}.$$
Indeed, the associated sheaf functor commutes with inductive limit,
since it is the inverse image of a morphism of topoi. The morphism
of left exact sites
$$u^*_L:(Et_{L/K},\mathcal{J}_{et})\longrightarrow(Et_{\bar{Y}},\mathcal{J}_{et}).$$
induces the following commutative diagram of topoi:
\[ \xymatrix{
\widetilde{(Et_{\bar{Y}},\mathcal{J}_{et})}\ar[r]^{(a,i)}\ar[d]_{u_L}& \widehat{Et}_{\bar{Y}}\ar[d]^{(u_L^p,u_{L,p})}\\
\widetilde{(Et_{L/K},\mathcal{J}_{et})}\ar[r]^{(a_L,i_L)}&\widehat{Et}_{L/K}
}\] where $\widehat{Et}_{\bar{Y}}$ (respectively
$\widehat{Et}_{L/K}$) denotes the category of presheaves on
${Et}_{\bar{Y}}$ (respectively on ${Et}_{L/K}$). To check the
commutativity of this diagram, we observe the direct images of these
morphisms, for which the commutativity is obvious. Therefore we have
$$a\circ u_L^p\simeq u_L^*\circ a_L,$$
where $a$ and $a_L$ are the associated sheaf functors. We obtain
$$u_L^*\mathcal{A}^{(q)}_L:=u_L^*\circ a_L(\mathcal{P}^q_L)=a\circ u_L^p(\mathcal{P}^q_L),$$
and finally
$$R^q\mathcal{A}:=\underrightarrow{lim}_{_{L}}\,u_L^*\mathcal{A}^{(q)}_L=\underrightarrow{lim}_{_{L}}\,a\circ u_L^p(\mathcal{P}^q_L)
=a(\underrightarrow{lim}_{_{L}} u_L^p(\mathcal{P}^q_L)),$$ where the
last identification comes from the fact that $a$ commutes with
inductive limits. In other words, $R^q\mathcal{A}$ is the étale
sheaf on $\bar{Y}$ associated to the presheaf
$\underrightarrow{lim}_{_{L}} u_L^p(\mathcal{P}^q_L).$ This presheaf
can be made explicit as follows. For any connected étale
$\bar{Y}$-scheme $\bar{V}$, one has
$$[\underrightarrow{lim}_{_{L}} u_L^p(\mathcal{P}^q_L)](\bar{V})=
\underrightarrow{lim}_{_{L}} [u_L^p(\mathcal{P}^q_L)(\bar{V})]
=\underrightarrow{lim}_{_{L}}
[\underrightarrow{lim}_{_{\bar{V}\rightarrow\bar{U}}}\,\,\mathcal{P}^q_L(\bar{V})],$$
where the second limit is taken over the category of arrows
$\bar{V}\rightarrow\bar{U}$, for $\bar{U}$ running through the class
of objects of $Et_{L/K}$. The first identification can be justified
by saying that the limits of presheaves are computed component-wise
(i.e. "arguments par arguments").

If $L/K$ is big enough so that $\bar{V}$ is an object of $Et_{L/K}$,
then one has
$$\underrightarrow{lim}_{_{\bar{V}\rightarrow\bar{U}}}\,\,\mathcal{P}^q_L(\bar{U})=\mathcal{P}^q_L(\bar{V})$$
since $Id_{\bar{V}}$ is then the initial object of the category of
arrows $\bar{V}\rightarrow\bar{U}$, for $\bar{U}$ in $Et_{L/K}$. We
obtain the following identifications
$$[\underrightarrow{lim}_{_{L}} u_L^p(\mathcal{P}^q_L)](\bar{V})=\underrightarrow{lim}_{_{L}}
[\underrightarrow{lim}_{_{\bar{V}\rightarrow\bar{U}}}\,\,\mathcal{P}^q_L(\bar{U})]
=\underrightarrow{lim}_{_{L}}\mathcal{P}^q_L(\bar{V})
=\underrightarrow{lim}_{_{L}}\,\underrightarrow{lim}_{_{S}}\,H^{q}(\mathfrak{F}_{_{L/K,S}},\bar{V},\mathcal{A}_{L/K,S}).
$$
Therefore, $R^q\mathcal{A}$ is the étale sheaf on $\bar{Y}$
associated to the presheaf
$$
\fonc{\mathcal{P}^q\mathcal{A}}{Et_{\bar{Y}}}{\underline{Ab}}{\bar{V}}
{\underrightarrow{lim}_{_{L,S}}\,H^{q}(\mathfrak{F}_{_{L/K,S}},\bar{V},\mathcal{A}_{L/K,S})=:
\underrightarrow{H}^{q}(\mathfrak{F}_{_{L/K,S}},\bar{V},\mathcal{A})
}.$$
\end{proof}

We consider below the sheaves $\mathbb{Z}$ and
$\widetilde{\mathbb{R}}$ of the total topos
$Top\,(\mathfrak{F}_\bullet)$ defined in Example
\ref{Licht-sheaf-comes-from-flask-topos}, and the sheaves
$\phi_!\mathbb{Z}$ and $\phi_!\widetilde{\mathbb{R}}$ defined in
section \ref{subsect-cpct-support}. Finally, we consider the étale
sheaves $\varphi_!\mathbb{Z}$ and $\varphi_!\mathbb{R}$ defined via
the open inclusion $\varphi:Y_{et}\rightarrow\bar{Y}_{et}$, where
$\mathbb{R}$ denotes here the constant sheaf on $Y_{et}$ associated
to the discrete abelian group $\mathbb{R}$. We assume below that $K$
is totally imaginary.
\begin{e-cor}\label{computations-Ri}
One has the following results :
\begin{enumerate}
\item $R^0(\mathbb{Z})=\mathbb{Z}$ and $R^q(\mathbb{Z})=0$ for $q\geq1$
odd.
\item $R^q(\widetilde{\mathbb{R}})=\mathbb{R}$ for $q=0,1$ and
$R^q(\widetilde{\mathbb{R}})=0$ for $q\geq2$.
\item $R^0(\phi_!\mathbb{Z})=\varphi_!\mathbb{Z}$ and
$R^q(\phi_!\mathbb{Z})=R^q\mathbb{Z}$ for $q\geq1$.
\item $R^q(\phi_!\widetilde{\mathbb{R}})=\varphi_!\mathbb{R}$ for $q=0,1$
and $R^q(\phi_!\widetilde{\mathbb{R}})=0$ for $q\geq2$.
\end{enumerate}

\end{e-cor}
\begin{proof}
Using Proposition \ref{lem-RqZ-associe} (and Proposition
\ref{localistion-relative-flask-topos}), this follows from
propositions \ref{Licht-cohomology-open-Z} and
\ref{Lichtenbaum-cohomology-open-R}, and equations
(\ref{exact-seq-coh-support-flask-Z}) and
(\ref{exact-seq-coh-support-flask-R}).

\end{proof}

\section{Etale complexes for the Weil-étale cohomology}

In this section, we assume that the number field $K$ is totally imaginary. In order to obtain the
relevant Weil-étale cohomology (i.e. the vanshing of the cohomology
in degrees $i\geq4$), we need to truncate the complex $R\mathbb{Z}$.
However, a non-trivial preliminary condition has to be satisfied.
Namely the sheaf $R^2\mathbb{Z}$, which fills the gap between Weil-étale and étale
cohomology, should be acyclic for the global
sections functor on $\bar{Y}$. We study below this sheaf and we show that it
has the right cohomology using an indirect argument. Then we define complexes of étale sheaves computing
the conjectural Weil-étale cohomology.

\subsection{Cohomology of the sheaf $R^2\mathbb{Z}$}
By Proposition \ref{Licht-cohomology-open-Z} and Proposition
\ref{lem-RqZ-associe}, the étale sheaf $R^2\mathbb{Z}$ is the sheaf
associated to the presheaf
$$
\fonc{\mathcal{P}^2\mathbb{Z}}{Et_{\bar{Y}}}{\underline{Ab}}{\bar{V}}
{\underrightarrow{H}^{2}(\mathfrak{F}_{_{L/K,S}},\bar{V},\mathbb{Z})=
(C^1_{\bar{V}})^{\mathcal{D}} }.$$ The compact group $C^1_{\bar{V}}$
is the kernel of the map $C_{\bar{V}}\rightarrow\mathbb{R}$, where
$C_{\bar{V}}$ is defined in \ref{defn-class-group-barU}. Recall that
if $\bar{V}$ is connected of function field $K(\bar{V})$, then
$C_{\bar{V}}$ is the $S$-idèle class group of $K(\bar{V})$, where
$S$ is the set of places of $K(\bar{V})$ not corresponding to a
point of $\bar{V}$. Note that such a finite set $S$ does not
necessarily contain all the archimedean places. The restrictions maps
of the presheaf $\mathcal{P}^2\mathbb{Z}$ are induced by the
canonical maps $C_{\bar{U}}\rightarrow C_{\bar{V}}$ (well defined
for any étale map $\bar{U}\rightarrow \bar{V}$ of connected étale
$\bar{Y}$-schemes). By unramified class field theory, one has a
covariantly functorial exact sequence of compact topological groups
$$0\rightarrow D^1_{\bar{V}}\rightarrow C^1_{\bar{V}}\rightarrow\pi_1^{ab}(\bar{V})\rightarrow0$$
where $\pi_1^{ab}(\bar{V})$ is the abelian étale fundamental group
of $\bar{V}$ and $D^1_{\bar{V}}$ is the connected component of $1$
in $C^1_{\bar{V}}$. Here $\pi_1^{ab}(\bar{V})$ is defined as the
abelianization of the profinite fundamental group of the
Artin-Verdier étale topos
$\bar{Y}_{et}/_{\bar{V}}\simeq\bar{V}_{et}$. If we denote the
function field of $\bar{V}$ by $K(\bar{V})$ then this group is just
the Galois group of the maximal abelian extension of $K(\bar{V})$
unramified at every place of $K(\bar{V})$ corresponding to a point
of $\bar{V}$.

By Pontryagin duality, we obtain a (contravariantly) functorial
exact sequence of discrete abelian groups
\begin{equation}\label{exact-sequence-presheaves-R2Z}
0\rightarrow \pi_1^{ab}(\bar{V})^{\mathcal{D}}\rightarrow
(C^1_{\bar{V}})^{\mathcal{D}}\rightarrow
(D^1_{\bar{V}})^{\mathcal{D}}\rightarrow0,
\end{equation}
i.e. an exact sequence of abelian étale presheaves on $\bar{Y}$. On
the one hand, the sheaf associated to the presheaf
$$
\appl{Et_{\bar{Y}}}{\underline{Ab}}{\bar{V}}
{\pi_1^{ab}(\bar{V})^{\mathcal{D}}=H^2(\bar{V}_{et},\mathbb{Z})}$$
vanishes and the associated sheaf functor is exact on the other.
Therefore, the exact sequence (\ref{exact-sequence-presheaves-R2Z})
shows that $R^2\mathbb{Z}$ is the sheaf associated to the presheaf
$$
\fonc{P}{Et_{\bar{Y}}}{\underline{Ab}}{\bar{V}}
{(D^1_{\bar{V}})^{\mathcal{D}} }.$$
The connected component $D^1_{\bar{V}}$ of the S-id\`ele class group $C^1_{\bar{V}}$ is not known in general, as pointed out to me by Alexander Schmidt. The computation of the sheaf $R^2\mathbb{Z}$ is a delicate problem. We shall compute the cohomology of $R^2\mathbb{Z}$ using an indirect argument.

\begin{e-lem}\label{unptitlem}
The canonical morphism
$$\underrightarrow{H}^4(\mathfrak{F}_{_{L/K,S}},\mathbb{Z})\longrightarrow H^0(\bar{Y}_{et},R^4\mathbb{Z}):=R^4\mathbb{Z}(\bar{Y})$$
is an isomorphism.
\end{e-lem}
\begin{proof}
The canonical morphism $\underrightarrow{H}^4(\mathfrak{F}_{_{L/K,S}},\mathbb{Z})\rightarrow H^0(\bar{Y}_{et},R^4\mathbb{Z})$ is induced by the morphism of presheaves $\mathcal{P}^4\mathbb{Z}\rightarrow R^4\mathbb{Z}$ (see Proposition \ref{lem-RqZ-associe}). Let $$J:B_{W_{K}}\longrightarrow
B^{sm}_{G_{K}}\longrightarrow\bar{Y}_{et}$$ be the morphism induced by the continuous morphism $W_K\rightarrow G_K$ and by the inclusion of the generic point of $\bar{Y}$. For any $n\geq 0$, the \'etale sheaf
$R^n(J_*)\mathbb{Z}$ is the sheaf associated to the presheaf
$$
\fonc{\mathcal{P}^n(J_*)\mathbb{Z}}{Et_{\bar{Y}}}{\underline{Ab}}{\overline{U}}
{H^n(W_{K(U)},\mathbb{Z})},$$ where $\overline{U}$ is assumed to be connected. Here
$W_{K(U)}$ is the Weil group of the number field $K(U)$. For any finite extension
$K'/K$, one has a surjective map
(see \cite{MatFlach} Proof of Corollary 9)
\begin{equation}\label{morphism-matthias}
H^4(W_{K'},\mathbb{Z})\longrightarrow\sum_{v\in
Y'_{\infty}}H^4(W_{K'_v},\mathbb{Z}) =\sum_{v\in
Y'_{\infty}}H^4(\mathbb{S}^1,\mathbb{Z})=\sum_{v\in
Y'_{\infty}}\mathbb{Z},
\end{equation}
where $Y'_{\infty}$ is the set of archimedean primes of $K'$.
We denote by
$$u:\coprod_{Y_{\infty}}{\underline{Set}}\longrightarrow\bar{Y}_{et}$$
the closed embedding of topoi given by the map $Y_{\infty}\rightarrow\bar{Y}$.
Then (\ref{morphism-matthias}) induces a surjective morphism of presheaves
$$\mathcal{P}^4(J_*)\mathbb{Z}\longrightarrow u_*\mathbb{Z}.$$
We obtain the following exact sequence of presheaves
$$0\longrightarrow \mathcal{P}^4\mathbb{Z}\longrightarrow\mathcal{P}^4(J_*)\mathbb{Z}\longrightarrow u_*\mathbb{Z}\longrightarrow0,$$
as it follows from Proposition \ref{lem-RqZ-associe} and Proposition \ref{Licht-cohomology-open-Z}. The associated sheaf functor is exact, hence we have an exact sequence of sheaves
$$0\longrightarrow R^4\mathbb{Z}\longrightarrow R^4(J_*)\mathbb{Z}\longrightarrow u_*\mathbb{Z}\longrightarrow0,$$
since $u_*\mathbb{Z}$ was already a sheaf. We get a long exact sequence
$$0\longrightarrow R^4\mathbb{Z}(\bar{Y})\longrightarrow R^4(J_*)\mathbb{Z}(\bar{Y})\longrightarrow\sum_{Y_{\infty}}\mathbb{Z}\longrightarrow...$$
Moreover, there is a morphism of exact sequences
\[ \xymatrix{
0\ar[r]&\underrightarrow{H}^4(\mathfrak{F}_{_{L/K,S}},\mathbb{Z})\ar[d]\ar[r]^{}&H^4(W_K,\mathbb{Z})\ar[d]_{\iota}\ar[r]^{}&\sum_{Y_{\infty}}\mathbb{Z} \ar[d]^{Id}\\
0\ar[r]&R^4\mathbb{Z}(\bar{Y})\ar[r]^{}&R^4(J_*)\mathbb{Z}(\bar{Y})\ar[r]^{}&\sum_{Y_{\infty}}\mathbb{Z}
}\]
where the vertical maps are the natural ones. Lemma \ref{lem-R4J} shows that $\iota$ is an isomorphism, and the result follows.
\end{proof}

We denote by $W_K^1$ the maximal compact sub-group of the Weil group
$W_K$. There is a canonical isomorphism of topological groups $W_K\simeq W_K^1\times\mathbb{R}$. We denote by $\mathcal{T}^{lc}$, $B^{lc}_{W_K}$ and $B^{lc}_{W_K^1}$ the topoi obtained by replacing the category of topological spaces $Top$ with the category $Top^{lc}$ of locally compact topological spaces with countable basis. In the following two lemmas, we consider the composite morphism
$$\alpha:B_{W_K}\xrightarrow{\,h\,}B^{lc}_{W_K}\xrightarrow{\alpha^{lc}}B^{lc}_{W_K^1}$$
where $\alpha^{lc}$ is the morphism of classifying topoi induced by the projection $W_K\rightarrow W_K^1$.

\begin{e-lem}\label{lem-alpha-acyclic}
For any $n\geq1$, one has $R^n(\alpha_*)\mathbb{Z}=0$.
\end{e-lem}
\begin{proof}
By \cite{MatFlach} Prop. 9.1, the direct image $h_*$ of the morphism $B_{W_K}\xrightarrow{\,h\,}B^{lc}_{W_K}$ is exact. The Leray spectral sequence associated to the composite morphism $\alpha$ gives
$$R^n(\alpha_*)\mathbb{Z}\simeq R^n(\alpha^{lc}_*)h_*\mathbb{Z}=R^n(\alpha^{lc}_*)\mathbb{Z}$$
It is therefore enough to show that $R^n(\alpha^{lc}_*)\mathbb{Z}=0$ for any $n\geq1$. We consider the pull-back
\[ \xymatrix{
B^{lc}_{\mathbb{R}}\ar[r]^{e_{\mathbb{R}}}\ar[d]_{f}&\mathcal{T}^{lc}\ar[d]^{l}\\
B^{lc}_{W_K}\ar[r]^{\alpha^{lc}}&B^{lc}_{W^1_K}
 }\]
where the vertical arrows are the localization morphisms (one has for example $B^{lc}_{W^1_K}/_{E_{{W^1_K}}}\simeq\mathcal{T}^{lc}$).
This pull-back square induces an isomorphism
$$l^*R^n(\alpha^{lc}_*)\simeq R^n(e_{\mathbb{R}*})f^*.$$ One the other hand, the object of $\mathcal{T}^{lc}$
$$R^n(e_{\mathbb{R}*})f^*\mathbb{Z}=R^n(e_{\mathbb{R}*})\mathbb{Z}$$
is represented by the discrete abelian group
$H^n(B^{lc}_{\mathbb{R}},\mathbb{Z})$ (see \cite{MatFlach} Prop. 9.2). This group is trivial for any $n\geq1$, and we obtain $l^*R^n(\alpha^{lc}_*)\mathbb{Z}=0$. But $l^*$ is faithful hence $R^n(\alpha^{lc}_*)\mathbb{Z}=0$ for any $n\geq1$. The result follows.
\end{proof}

\begin{e-lem}\label{lem-R4J}
The canonical map
$$H^4(W_K,\mathbb{Z})\longrightarrow R^4(J_*)\mathbb{Z}(\bar{Y}).$$
is an isomorphism.
\end{e-lem}
\begin{proof}
We decompose the morphism
$J:B_{W_{K}}\rightarrow\bar{Y}_{et}$ as follows:
$$J=\beta\circ\alpha:B_{W_K}\longrightarrow
B^{lc}_{W_K^1}\longrightarrow\bar{Y}_{et}.$$ The Leray spectral sequence associated to this composite map and the previous Lemma show that the natural morphism of \'etale sheaves
$$R^n(\beta_*)\mathbb{Z}=R^n(\beta_*)(\alpha_*\mathbb{Z})\longrightarrow R^n(J_*)\mathbb{Z}$$
is an isomorphism. It is therefore enough to show that the natural map
\begin{equation}\label{iso-necessary}
H^4(W_K,\mathbb{Z})\simeq H^4(W^1_K,\mathbb{Z})\longrightarrow R^4(\beta_*)\mathbb{Z}(\bar{Y})
\end{equation}
is an isomorphism (where $H^4(W_K,\mathbb{Z})\simeq H^4(W^1_K,\mathbb{Z})$ follows from Lemma \ref{lem-alpha-acyclic}). To this aim, we decompose the morphism $\beta$ as follows:
$$\beta=\mathrm{j}\circ p:B^{lc}_{W^1_{K}}\longrightarrow
B^{sm}_{G_K}\longrightarrow\bar{Y}_{et}.$$
This provides us with the leray spectral sequence
$$R^i(\mathrm{j}_*)\circ R^j(p_*)\mathbb{Z}\Longrightarrow R^{i+j}(\beta_*)\mathbb{Z}.$$
We denote by $M^j$ the $G_K$-module $R^j(p_*)\mathbb{Z}$. By \cite{MatFlach} equation (21) and \cite{MatFlach} Lemma 11, one has
$$M^j=0\mbox{ for $j$ odd}.$$
For any $G_K$-module $M$, the \'etale sheaf
$R^i(\mathrm{j}_*)M$ is the sheaf associated to the presheaf
$$\overline{U}\mapsto H^i(B^{sm}_{G_K},U_{v_0},M)=H^i(G_{K(U)},M).$$
It is well known that a totally imaginary number field is of strict cohomological dimension 2. Hence we have $R^i(\mathrm{j}_*)M=0$ for $i\geq3$. The proof of (\cite{MatFlach} Lemma 12 (b)) shows that
$H^i(G_{K(U)},M^2)=0$ for any $i\geq 1$.
Hence the group
$$R^i(\mathrm{j}_*)\circ R^j(p_*)\mathbb{Z}=R^i(\mathrm{j}_*)M^j$$
is trivial for $i\geq3$, or if the index $j$ is odd, or
for $(j=2,i\geq1)$. The initial term of the spectral sequence
$$R^i(\mathrm{j}_*)M^j:=R^i(\mathrm{j}_*)\circ R^j(p_*)\mathbb{Z}\Longrightarrow
R^{i+j}(\beta_*)\mathbb{Z}$$ therefore looks as follows:
{\small{
$$
\begin{array}{cccccc}
0&0&0&0&0&0  \\
\mathrm{j}_*M^4&R^1\mathrm{j}_*M^4&R^2\mathrm{j}_*M^4&0 &0 &0      \\
0&0&0&0&0&0  \\
\mathrm{j}_*M^2&0&0&0 &0 &0      \\
0&0&0&0&0&0  \\
\mathrm{j}_*M^0&R^1\mathrm{j}_*M^0&R^2\mathrm{j}_*M^0&0 &0 &0  \\
\end{array}
$$
}}This yields a natural isomorphism
$$R^4(\beta_*)\mathbb{Z}\simeq\mathrm{j}_*M^4,$$
and we obtain the following identifications
\begin{equation}\label{iso-meles}
R^4(\beta_*)\mathbb{Z}(\bar{Y})\simeq\mathrm{j}_*M^4(\bar{Y})=H^0(G_K,M^4)\simeq H^4(W^1_K,\mathbb{Z}).
\end{equation}
Indeed, the last isomorphism in (\ref{iso-meles}) is given by the spectral sequence
$$H^i(G_K,M^j)\Longrightarrow H^{i+j}(W_K^1,\mathbb{Z})$$
which is made explicit in (\cite{MatFlach} Lemma 12).
Note that the isomorphisms in (\ref{iso-meles}) are given by the natural maps
$$H^4(W^1_K,\mathbb{Z})\longrightarrow R^4(\beta_*)\mathbb{Z}(\bar{Y})\longrightarrow H^0(G_K,M^4).$$
Hence (\ref{iso-necessary}) is an isomorphism and the result follows.
\end{proof}

Recall that $A^{\mathcal{D}}:=Hom_c(A,\mathbb{R}/\mathbb{Z})$
denotes the Pontryagin dual of a locally compact abelian group $A$. If $A$ is a discrete abelian group, we set $A^{D}:=Hom(A,\mathbb{Q}/\mathbb{Z})$.

\begin{e-thm}\label{prop-R2-acyclique} The \'etale sheaf $R^2\mathbb{Z}$ is acyclic for the global section functor $\Gamma_{\bar{Y}}$. In other words, one has $H^n(\bar{Y}_{et},R^2\mathbb{Z})=0$ for any $n\geq1$.
\end{e-thm}
\begin{proof}
By Corollary \ref{computations-Ri} (1), the initial term of the spectral sequence
\begin{equation}\label{one-spectral-sequence}
H^p(\bar{Y},R^q\mathbb{Z})\Longrightarrow\underrightarrow{H}^{p+q}(\mathfrak{F}_{L/K,S},\mathbb{Z})
\end{equation}
looks as follows: {\small{
$$
\begin{array}{cccccc}
0&0&0&0&\,\,\,\,\,0\,\,\,\,\,\,\,\,\,\, &0 \\[0.5cm]
H^0(\bar{Y},R^4\mathbb{Z})&H^1(\bar{Y},R^4\mathbb{Z})  &H^2(\bar{Y},R^4\mathbb{Z}) &H^3(\bar{Y},R^4\mathbb{Z})  &\,\,\,\,\,0\,\,\,\,\,\,\,\,\,\, &0 \\[0.5cm]
0&0&0&0&\,\,\,\,\,0\,\,\,\,\,\,\,\,\,\, &0 \\[0.5cm]
H^0(\bar{Y},R^2\mathbb{Z})&H^1(\bar{Y},R^2\mathbb{Z})  &H^2(\bar{Y},R^2\mathbb{Z}) &H^3(\bar{Y},R^2\mathbb{Z})  &\,\,\,\,\,0\,\,\,\,\, \,\,\,\,\,&0 \\[0.5cm]
0&0&0&0&\,\,\,\,\,0\,\,\,\,\,\,\,\,\,\, &0 \\[0.5cm]
\mathbb{Z}&0&Pic(Y)^D&U_K^D &\,\,\,\,\,0\,\,\,\,\,\,\,\,\,\, &0  \\[0.5cm]
\end{array}
$$
}}
We obtain the exact sequence
\begin{equation}\label{unebonnesuiteexacte}
0\rightarrow Pic(Y)^D\rightarrow Pic^1(\bar{Y})^{\mathcal{D}}\rightarrow H^0(\bar{Y}_{et},R^2\mathbb{Z})
\rightarrow U_K^D\rightarrow\mu_K^D\rightarrow
H^1(\bar{Y}_{et},R^2\mathbb{Z})\rightarrow0.
\end{equation}
The group $H^1(\bar{Y}_{et},R^2\mathbb{Z})$ is trivial since the canonical map
$U_K^D\rightarrow\mu_K^D$ is surjective. Then, this spectral sequence gives the exact sequence
$$0\rightarrow H^2(\bar{Y}_{et},R^2\mathbb{Z})\rightarrow\underrightarrow{H}^4(\mathfrak{F}_{L/K,S},\mathbb{Z})\rightarrow H^0(\bar{Y}_{et},R^4\mathbb{Z})\rightarrow
H^3(\bar{Y}_{et},R^2\mathbb{Z})\rightarrow\underrightarrow{H}^5(\mathfrak{F}_{L/K,S},\mathbb{Z})=0$$
where the central map is an isomorphism by Lemma \ref{unptitlem}. We get
$$H^2(\bar{Y}_{et},R^2\mathbb{Z})=H^3(\bar{Y}_{et},R^2\mathbb{Z})=0.$$
Finally, the group $H^n(\bar{Y}_{et},R^2\mathbb{Z})$ is trivial for any $n\geq 4$ since the \'etale site of $\bar{Y}$ is of strict cohomological dimension 3 (see \cite{Deninger}).
\end{proof}
In order to compute the group $H^0(\bar{Y}_{et},R^2\mathbb{Z})$, one needs to study the sheaf $R^2\mathbb{Z}$ in more detail. There is a canonical map
$$Pic^1(\bar{Y})^{\mathcal{D}}\longrightarrow Hom(U_K,\mathbb{Z})\longrightarrow Hom(U_K,\mathbb{Q}).$$
One can show that the morphism
$$Pic^1(\bar{Y})^{\mathcal{D}}\simeq\underrightarrow{H}^{2}(\mathfrak{F}_{L/K,S},\mathbb{Z})\longrightarrow H^0(\bar{Y}_{et},R^2\mathbb{Z})$$ factors through an injective map
\begin{equation}\label{iso-comparison}
c:Hom(U_K,\mathbb{Q})\longrightarrow H^0(\bar{Y}_{et},R^2\mathbb{Z}).
\end{equation}
Then one can show that this gives a morphism of exact sequences
\[ \xymatrix{
0\ar[r]&Pic(Y)^D\ar[d]_{Id}\ar[r]&Pic^1(\bar{Y})^{\mathcal{D}}\ar[r]\ar[d]_{\simeq}& Hom(U_K,\mathbb{Q})\ar[d]_c\ar[r]&U_K^D\ar[d]_{\simeq}\ar[r]&\mu_K^D \ar[d]_{\simeq}\ar[r]&0\\
0\ar[r]&Pic(Y)^D\ar[r]&\underrightarrow{H}^2(\mathfrak{F}_{L/K,S},\mathbb{Z})\ar[r]
&R^2\mathbb{Z}(\bar{Y})\ar[r]&U_K^D\ar[r] & \underrightarrow{H}^3(\mathfrak{F}_{L/K,S},\mathbb{Z})\ar[r]&0
}\]
where the bottom row is the exact sequence given by the spectral sequence (\ref{one-spectral-sequence}). It follows that (\ref{iso-comparison}) is an isomorphism. But this fact will not be used in the remaining part of this paper.

\subsection{The complexes}
\begin{e-thm}\label{thm-complexe-etale-Wetale}
There exists a complex $R_W\mathbb{Z}$ of étales sheaves on
$\bar{Y}_{et}$, well defined up to quasi-ismorphism, whose
hypercohomology is the expected Weil-étale cohomology :
\begin{eqnarray*}
H^n(\bar{Y}_{et};R_W\mathbb{Z})&=&\mathbb{Z} \mbox{ for }n=0,\\
                           &=&0 \mbox{ for }n=1,         \\
                           &=& Pic^1(\bar{Y})^{\mathcal{D}}\mbox{ for }n=2,\\
                           &=& \mu_K^D\mbox{ for }n=3,\\
                           &=& 0\mbox{ for } n\geq 4.
\end{eqnarray*}

\end{e-thm}
\begin{proof}
Consider the complex $R\mathbb{Z}$. The truncated complex
$$R_W\mathbb{Z}=\tau_{\leq2} R\mathbb{Z},$$
is also well defined up to quasi-isomorphism. One has
$H^n(\tau_{\leq2} R\mathbb{Z})=H^n(R\mathbb{Z})$ for $n\leq2$, and
$H^n(\tau_{\leq2} R\mathbb{Z})=0$ for $n\geq3$. By Corollary
\ref{computations-Ri} and Theorem \ref{prop-R2-acyclique} the
$E^2$-term of the spectral sequence
$$H^p(\bar{Y}_{et},H^q(R_W\mathbb{Z}))\Longrightarrow H^{p+q}(\bar{Y}_{et},R_W\mathbb{Z})$$
therefore looks like : {
$$
\begin{array}{cccccc}
0\,\,\,\,\,&0\,\,\,\,\,&0\,\,\,\,\,&0\,\,\,\,\,&\,\,\,\,\,0 \\[0.5cm]
H^0(\bar{Y}_{et},R^2\mathbb{Z})\,\,\,\,\,&0\,\,\,\,\,  &0\,\,\,\,\, &0\,\,\,\,\,  &\,\,\,\,\,0 \\[0.5cm]
0\,\,\,\,\,&0\,\,\,\,\,&0\,\,\,\,\,&0\,\,\,\,\,&\,\,\,\,\,0 \\[0.5cm]
\mathbb{Z}\,\,\,\,\,&0\,\,\,\,\,&Pic(Y)^D\,\,\,\,\,&U_K^D\,\,\,\,\, &\,\,\,\,\,0  \\[0.5cm]
\end{array}
$$
} We obtain immediately
$H^{0}(\bar{Y}_{et},R_W\mathbb{Z})=\mathbb{Z}$ and
$H^{1}(\bar{Y}_{et},R_W\mathbb{Z})=0$. Next the spectral sequence
yields the exact sequence
\begin{equation}\label{exact-sequence-RZ-labonne}
0\rightarrow Pic(Y)^D\rightarrow
H^{2}(\bar{Y}_{et},R_W\mathbb{Z})\rightarrow
H^0(\bar{Y}_{et},R^2\mathbb{Z})\rightarrow
Hom(U_K,\mathbb{Q}/\mathbb{Z})\rightarrow
H^{3}(\bar{Y}_{et},R_W\mathbb{Z})\rightarrow0.
\end{equation}
Note that it is already clear that $R_W\mathbb{Z}$ has the expected
hypercohomology. The morphism of complexes $R_W\mathbb{Z}\rightarrow
R\mathbb{Z}$ induces a morphism of spectral sequences from
$$H^p(\bar{Y}_{et},H^q(R_W\mathbb{Z}))\Longrightarrow H^{p+q}(\bar{Y}_{et},R_W\mathbb{Z})$$
to
$$H^p(\bar{Y}_{et},H^q(R\mathbb{Z}))\Longrightarrow\underrightarrow{H}^{p+q}(\mathfrak{F}_{_{L/K,S}},\mathbb{Z})$$
which in turn induces a morphism of exact sequences from
(\ref{exact-sequence-RZ-labonne}) to (\ref{unebonnesuiteexacte}). We obtain

$$H^{n}(\bar{Y}_{et},R_W\mathbb{Z})\simeq\underrightarrow{H}^{n}(\mathfrak{F}_{_{L/K,S}},\mathbb{Z})\mbox{\,\,for
$n=2,3$}.$$ The result for $n=2,3$ then follows from Theorem
\ref{thm-lichten-weil-etale-Zcoef}. Finally, the groups
$H^{n}(\bar{Y}_{et},R_W\mathbb{Z})$ vanish for $n\geq4$, since the
diagonals $\{p+q=n,\,\,n\geq4\}$ of this spectral sequence are
trivial.
\end{proof}

\begin{e-thm}\label{thm-complexe-etale-Wetale}
There exists a complex $R_W(\phi_!\mathbb{Z})$ of étales sheaves on
$\bar{Y}_{et}$, well defined up to quasi-ismorphism, whose
hypercohomology is the expected Weil-étale cohomology with compact
support :
\begin{eqnarray*}
H^n(\bar{Y}_{et};R_W(\phi_!\mathbb{Z}))&=&0 \mbox{ for }n=0,\\
                           &=&(\prod_{Y_{\infty}}\mathbb{Z})/\mathbb{Z} \mbox{ for }n=1,         \\
                           &=& Pic^1(\bar{Y})^{\mathcal{D}}\mbox{ for }n=2,\\
                           &=& \mu_K^D\mbox{ for }n=3,\\
                           &=& 0\mbox{ for } n\geq 4.
\end{eqnarray*}
\end{e-thm}
\begin{e-proof}
The morphism $\phi_!\mathbb{Z}\rightarrow\mathbb{Z}$ in
$Top\,(\mathfrak{F}_{\bullet})$ induces a morphism of étale
complexes $R\phi_!\mathbb{Z}\rightarrow R\mathbb{Z}$. We obtain a
morphism of truncated complexes
$$R_W(\phi_!\mathbb{Z}):=\tau_{\leq2}
R(\phi_!\mathbb{Z})\longrightarrow R_W(\mathbb{Z}):=\tau_{\leq2}
R\mathbb{Z}$$ inducing a morphism of spectral sequences. Using
Corollary \ref{computations-Ri}(3), we obtain
$$H^n(\bar{Y}_{et};R_W(\phi_!\mathbb{Z}))=H^n(\bar{Y}_{et};R_W(\mathbb{Z}))$$
for $n\geq2$. Finally, the spectral sequence
$$H^p(\bar{Y}_{et};H^q(R_W\phi_!\mathbb{Z}))\Longrightarrow H^{p+q}(\bar{Y}_{et};R_W(\phi_!\mathbb{Z}))$$
yields
$H^{n}(\bar{Y}_{et};R_W(\phi_!\mathbb{Z}))=H^{n}(\bar{Y}_{et};\varphi_!\mathbb{Z})$
for $n=0,1$. The result follows from (\ref{cohom-compact-etale-Z}).
\end{e-proof}

\begin{e-thm}\label{thm-complexe-etale-Wetale}
The hypercohomology of the complex of étale sheaves
$R(\phi_!\widetilde{\mathbb{R}})$ is given by
\begin{eqnarray*}
H^n(\bar{Y}_{et};R(\phi_!\widetilde{\mathbb{R}}))&=&0 \mbox{ for }n=0,\\
                           &=&(\prod_{Y_{\infty}}\mathbb{R})/\mathbb{R} \mbox{ for }n=1,2         \\
                           &=& 0\mbox{ for } n\geq 3.
\end{eqnarray*}
\end{e-thm}
\begin{e-proof}
The spectral sequence
$$H^p(\bar{Y}_{et};R^q(\phi_!\widetilde{\mathbb{R}}))\Longrightarrow H^{p+q}(\bar{Y}_{et};R(\phi_!\widetilde{\mathbb{R}}))$$
degenerates and yields
$$H^{n}(\bar{Y}_{et};R(\phi_!\widetilde{\mathbb{R}}))
=H^{1}(\bar{Y}_{et},R^{n-1}(\phi_!\widetilde{\mathbb{R}}))=H^{1}(\bar{Y}_{et},\varphi_!\mathbb{R})\mbox{
for $n=1,2$}$$ and
$H^{n}(\bar{Y}_{et};R(\phi_!\widetilde{\mathbb{R}}))=0$ for
$n\neq1,2$ (see Corollary \ref{computations-Ri}(4)). Hence the
result follows from (\ref{cohom-compact-etale-R}).
\end{e-proof}

\

$\mathbf{Acknowledgments}.$ The author is very grateful to Matthias
Flach for his questions, comments and suggestions which led to a more complete and useful article.

\end{document}